\documentclass[reqno,11pt,a4paper]{amsart}

\usepackage{amsfonts,amssymb,amsmath,pinlabel,array,hhline,euscript}
\usepackage{slashed}
\usepackage[all]{xy}
\usepackage{tabulary}
\usepackage{fancyhdr}
\usepackage{a4wide}
\usepackage{bbm}
\usepackage{enumerate}

\usepackage{setspace}

\usepackage[position=b]{subcaption}

\usepackage[linktocpage=true]{hyperref}

\DeclareRobustCommand{\SkipTocEntry}[5]{}

\newtheorem{theorem}{Theorem}[section]

\newtheorem{lemma}[theorem]{Lemma}
\newtheorem{proposition}[theorem]{Proposition}

\newtheorem{corollary}[theorem]{Corollary}

\theoremstyle{definition}
\newtheorem{definition}[theorem]{Definition}
\theoremstyle{remark}
\newtheorem{remark}{Remark}[section]

\newtheorem{example}{Example}[section]

\newcounter{ticklistc}

\newcommand\old[1]{}

\def\opname{\operatorname}
\renewcommand{\Im}{\opname{Im}}
\renewcommand{\Re}{\opname{Re}}
\newcommand{\Arf}{\opname{Arf}}
\newcommand{\Hom}{\opname{Hom}}
\newcommand{\sign}{\opname{sign}}

\newcommand\wh[1]{\widehat{#1}}
\newcommand\matr[4]{\left(\begin{array}{cc}#1 & #2 \cr #3 & #4\end{array}\right)}
\newcommand\Pf[1]{\opname{Pf}[\,#1\,]} 
\newcommand\Proj[2]{\opname{Proj}\,[\,#1\,;\,#2\,]}

\newcommand{\Z}{\mathbb Z}
\newcommand{\R}{\mathbb R}
\newcommand{\C}{\mathbb C}
\newcommand{\E}{\mathbb{E}}
\renewcommand{\P}{\mathbb{P}}

\renewcommand{\S}{\mathcal{\EuScript S}}
\newcommand{\SI}{\Sigma}

\newcommand{\wind}{\opname{wind}}

\def\veps{\varepsilon} \def\reps{\varepsilon}

\def\rB{\mathrm{B}}
\def\rC{\mathrm{C}}
\def\rD{\mathrm{D}}
\def\rE{\mathrm{E}}
\def\rF{\mathrm{F}}
\def\rI{\mathrm{I}}
\def\rJ{\mathrm{J}}
\def\rK{\mathrm{K}}
\def\rS{\mathrm{S}}
\def\rT{\mathrm{T}}
\def\rU{\mathrm{U}}
\def\rW{\mathrm{W}}
\def\rY{\mathrm{Y}}

\def\cC{\mathcal{C}}
\def\cD{\mathcal{D}}  \def\D{\cD}
\def\cE{\mathcal{E}}
\def\cO{\mathcal{O}}
\def\cS{\mathcal{S}}
\def\cZ{\mathcal{Z}}

\def\lan{\prec\!} \def\ran{\!\succ}
\def\btu{\!\bigtriangleup\!}

\def\EE{\overrightarrow{E}}
\def\KW{\mathcal{K}\mathcal{W}}

\def\rt{\opname{t}}
\def\rw{\opname{w}}

\begin{document}

\title{Revisiting the combinatorics of the 2D Ising model}

\author{Dmitry Chelkak}
\address{Holder of the ENS--MHI chair funded by MHI. D\'epartement de math\'ematiques et applications de l'ENS, Ecole Normale Sup\'erieure PSL Research University, CNRS UMR 8553, Paris 5\`eme.
On leave from St.~Petersburg Department of Steklov Mathematical Institute RAS.}
\email{dmitry.chelkak@ens.fr}

\author{David Cimasoni}
\address{Universit\'e de Gen\`eve, Section de math\'ematiques, 2-4 rue du Li\`evre, Case postale 64, 1211 Gen\`eve 4, Switzerland}
\email{david.cimasoni@unige.ch}

\author{Adrien Kassel}
\address{ETH Z\"urich, Departement Mathematik, R\"amistrasse 101, 8092 Z\"urich, Switzerland}
\email{adrien.kassel@math.ethz.ch}


\begin{abstract}
We provide a concise exposition with original proofs of combinatorial formulas for the 2D Ising model partition function,
multi-point fermionic observables, spin and energy density correlations, for general
graphs and interaction constants, using the language of Kac--Ward matrices. We also give a brief account of the relations between various alternative formalisms which have been used in the combinatorial study of the planar Ising model: dimers and Grassmann variables, spin and disorder operators, and, more recently, s-holomorphic observables. In addition, we point out that these formulas can be extended to the double-Ising model, defined as a pointwise product of two Ising spin configurations on the same discrete domain, coupled along the boundary.
\end{abstract}

\maketitle


\setcounter{tocdepth}{2}

\addtocontents{toc}{\setlength\parskip{4pt}}

\tableofcontents

\section{Introduction}\label{sec:introduction}

\subsection{Overview}

The two-dimensional Ising model -- famous toy-model of magnetic interaction introduced by Lenz and initially studied by Ising~\cite{Ising} in dimension $1$ -- has been the subject of extensive activity following the proof of existence of a phase transition by Peierls~\cite{Peierls}, the prediction of its critical temperature by Kramers and Wannier~\cite{Kramers-Wannier}, and the computation of its free energy by Onsager~\cite{Onsager} and Kaufman~\cite{Kaufman,Kaufman-Onsager}. There have been literally thousands of papers on the subject and a standard gateway to the main developments of the last century is the classical text of McCoy and Wu~\cite[Chapter I]{McCoy-Wu}. This book focuses on one of the mainstream approaches to the study of the Ising model: the~\emph{combinatorial method}, which in contrast to the algebraic method of Onsager--Kaufman, is based on combinatorial bijections. The founding papers of this method include a series of works by Kac, Ward, Potts, Hurst, Green, Kasteleyn, Montroll, Fisher, and others~\cite{Kac-Ward, Potts-Ward, Hurst-Green, Kasteleyn-63,Montroll-Potts-Ward-63,Fisher-66}. Yet an alternative approach was proposed by Baxter and Enting~\cite{Baxter-399th} based on invariance under local star-triangle transform of the underlying planar graph~\cite{Baxter-1978}. These techniques enabled a broad understanding of the (infinite-volume limit of the) model, which by the 1970s was widely considered to be a successfully closed case in the mathematical physics community.

In more recent times, a deep algebraic structure of (infinite-volume) spin correlations on regular 2D lattices was found, in particular due to the work of Wu, McCoy, Tracy, and Barouch~\cite{Wu-McCoy-Tracy-Barouch} on Painlev\'{e} equations, the work of Sato, Miwa, and Jimbo~\cite{SMJ-77,SMJ-80b,Jimbo-89} on isomonodromic deformations and $\tau$-functions, and the related work of Perk~\cite{Perk} {(see also~\cite{McCoy-Perk-Wu})} and Palmer~\cite{Palmer-book}. This is surveyed in~\cite[Chapter XVII]{McCoy-Wu}.
In the 2000s, the model was revived yet again by Smirnov~\cite{smirnov-icm-2006,smirnov-icm-2010}, when Schramm's invention~\cite{Schramm-SLE} of SLE curves led to the emergence of a new field focusing on rigorously proving convergence of 2D lattice models to their conformally invariant continuum counterparts using discrete complex analysis techniques. This allowed new developments concerning the fine understanding of the conformal invariance of the {critical} Ising model in \emph{general} planar domains, both from the geometric (convergence of interfaces to SLE curves) and the analytic (confirming Conformal Field Theory predictions for the scaling limits of correlation functions) viewpoints; see~\cite{Smirnov-10,Duminil-Smirnov-lectures,chelkak-smirnov,hongler-smirnov,Hongler-thesis,Hongler-Kytola,
Chelkak-Izyurov,CHI,Kemppainen-Smirnov,CDCHKS,izyurov-mult,izyurov-free}. A number of related results can be found in~\cite{Duminil-Hongler-Nolin,Beffara-Duminil-nearcrit,Dub,Camia-Garban-Newman2,Camia-Garban-Newman, hongler-kytola-zahabi,hongler-and-co,Chelkak-Duminil-Hongler,Benoist-Duminil-Hongler}. In parallel, further developments were made on various algorithmic and algebraic aspects of the model (e.g., see~\cite{Lubetzky-Sly} and~\cite{Kenyon-Pemantle}).

This paper is about the combinatorics of the 2D {nearest-neighbor} Ising model on \emph{general finite weighted graphs} and one of its main goals is to make the basic methods and formalisms used in most of the works cited above better known and available in a practical way to the \emph{probability and combinatorics} community which has been rather active on this topic recently~\cite{BdT1,BdT-Z,BdT-xor,BdT4,Cimasoni-CMP,Li,Cimasoni-Duminil,Kager-Lis-Meester,Lis-2013,Lis-2014,Duminil-Garban-Pete,MIT-guys,Cambridge-guys,Costantino}.
To that end, we focus on presenting and proving combinatorial formulas for the partition function, multi-point fermionic observables, and spin and energy density correlations (see Section~\ref{sec:planar} for the planar case and Section~\ref{sec:surface} for the extension to surfaces). The preferred language we use throughout the paper is that of \emph{Kac--Ward matrices} (see Section~\ref{sub:intro-KW-and-terminal}), and although it is hard to claim originality in view of the rich and overwhelming history and literature on the subject, we do give simple and general proofs of many results
for which we have not been able to find any explicit reference. A notable example is the famous Kac--Ward formula~\eqref{eqn:KWformula} for which we provide a very short proof in Section~\ref{sec:planar}. A motivation for discussing these combinatorial formulas in their full generality is the many open directions that still remain, including the Ising model in random media, spin glasses, the Ising model on random maps and non-integrable Ising models. Note that some progress was made on the last topic a few years ago~\cite{Giuliani}, where the energy density field of the near-critical Ising model with finite range interactions was shown to be universal in the limit, following a rigorous application of the renormalization group methods and taking advantage of the classical Grassmann variables representation of the model. {Another approach to reveal the Pfaffian structure of correlation functions arising in the limit of the (critical or near-critical) non-integrable Ising model was recently suggested in~\cite{Aizenmann-Duminil-Pfaff16}. It is based on the so-called random current representation of the model, which also has been the subject of renewed interest in the nearest-neighbor case, see~\cite{Camia-Lis,Lupu-Werner}.}

Similarly to the fact that all problems on random walks (whether classical, in random environment, on random graphs, etc.) have the same underlying structure of discrete harmonic functions, the structure underlying the 2D nearest-neighbor Ising model is that of \emph{s-holomorphic} functions, a definition introduced in~\cite{chelkak-smirnov} to encode a stronger version of discrete Cauchy-Riemann identities for some combinatorial observables arising in the model. Similar objects (discrete fermions satisfying some local relations aka propagation equations) go back to the founding papers on the subject, which use several different languages to describe the same structure. Despite the fact that all these languages are essentially equivalent to one another, we do not know of a reference providing an explicit exposition of the links between them (of course, it should be said that such links are part of the folklore surrounding the Ising model).
In view of the recent activity in the field, we believe it useful to provide such an exposition in one place (intended in particular for combinatorialists and probabilists) and thus devote Section~\ref{sec:links} to proving these equivalences, in particular the one between \emph{spin-disorders}~\cite{Kadanoff-Ceva} and \emph{Grassmann variables}~\cite{Lieb-fermions}, considered on double-covers.

In addition, we provide an extension of these combinatorial formulas to the \emph{double-Ising model}; see Section~\ref{sub:intro-dblI}. This model, defined as a pointwise product of two
Ising spin configurations on the same discrete domain, coupled along the boundary, is related to the \emph{bosonization} of the Ising model (e.g., see~\cite[Chapter 12]{YellowBook} {or~\cite[Section 12.4]{MussardoBook}}), a topic which has been revived and studied from the combinatorial point of view recently~\cite{Dub,BdT-xor,BdT4}.
The critical double-Ising model was also studied in the physics literature (see in particular~\cite{Ikhlef1,Ikhlef2,Santachiara}) in the context of the Ashkin--Teller model, a four-state spin model of which it is a special case. Similarly to the known relation between the scaling limit of interfaces arising in the Ising model at criticality~\cite{CDCHKS} and conformal loop ensembles~\cite{Sheffield-Werner}, Wilson~\cite{Wilson-xor} conjectured a relation between the interfaces arising in the double-Ising model at criticality (considered on the hexagonal lattice) and the set of ``level lines'' of the Gaussian Free Field~\cite{Schramm-Sheffield-GFFline,Wang-Wu}. In Section~\ref{sub:dblI-s-hol-fct} we discuss whether the combinatorial formulas could play a role in understanding this passage to the scaling limit as they do for the critical (single-) Ising model~\cite{CDCHKS}.

The paper is organized as follows. In the remainder of Section~\ref{sec:introduction}, we give a detailed presentation of our main results in the case of planar graphs.
In Section~\ref{sec:planar}, we prove the statements concerning the planar Ising model. In Section~\ref{sec:links}, we present an overview of the links between various formalisms that have been used in the study of the Ising model. In Section~\ref{sec:surface}, we generalize the results and proofs of Section~\ref{sec:planar} to the surface case. Finally, in Section~\ref{sec:double}, we prove the results concerning the double-Ising model.

\subsection{Partition function, combinatorial expansions, and embeddings} \label{sub:intro-Z-Ising}

Let $G$ be a finite connected graph with vertex set $V(G)$ and set of unoriented edges $E(G)$.
The Ising model on~$G$ is defined as follows.
A \emph{spin configuration} $\sigma$ is the assignment of a~$\pm 1$ spin to each vertex of the graph. For each (unoriented) edge $e=\{u,v\}\in E(G)$, let $J_e=J_{u,v}\in\R$ be an \emph{interaction constant} and {denote by $J$ the collection of all $J_e$.} Consider the \emph{Hamiltonian}
\begin{equation}\label{eqn:intro-H}
\textstyle H(\sigma)~=~-\sum_{\{u,v\}\in E(G)}J_{u,v} \,\sigma_u\sigma_v\,.
\end{equation}
For a fixed nonnegative real $\beta$, called the \emph{inverse temperature}, the Ising model is the probability distribution on spin configurations given by
\[
\P_{G}(\sigma)=\P_{G,\beta,\{J_e\}_{e\in E(G)}}(\sigma)~:=~[\cZ_{\beta}(G,J)]^{-1}\cdot\exp[-\beta H(\sigma)]\,,
\]
where the normalizing factor
\[
\textstyle \cZ_\beta(G,J)~=~\sum_{\sigma\in\{\pm 1\}^{V(G)}}\exp[-\beta H(\sigma)]
\]
is called the \emph{partition function} of the model.

It is convenient to introduce two polynomials encoding the combinatorial structure of the model. For that matter, we let $x=(x_e)_{e\in E(G)}$ be a collection of variables and view~$(G,x)$ as a weighted graph.
For any subset of edges $\rE\subset E(G)$, we define $x(\rE):=\prod_{e\in \rE}x_e$.

Let $\cE(G)$ be the set of all subgraphs with even vertex-degrees, called \emph{even subgraphs}.
The \emph{high-temperature polynomial} is defined to be
\begin{equation*}
\textstyle \cZ_{\opname{high}}(G,x):=\sum_{P\in \cE(G)}x(P)\,.
\end{equation*}
As first observed by van der Waerden~\cite{VdW}, the fact that the products of spins is always $\pm 1$ and some cancellations due to parity yield
\begin{equation}
\label{eqn:Z-Ising-high}
\textstyle \cZ_\beta(G,J)~ = ~(2^{|V(G)|}\prod_{e\in E(G)}\cosh[\beta J_e])\cdot \cZ_{\opname{high}}(G,(x_e\!:=\!\tanh[\beta J_e])_{e\in E(G)})\,.
\end{equation}

The \emph{low-temperature}, or \emph{domain-walls}, expansion is another useful polynomial expansion which has a straightforward interpretation in terms of spin configurations on the \emph{dual} graph to~$G$. It requires the choice of an embedding of~$G$ into a surface $\Sigma$ possibly with boundary ({given by a disjoint union of topological circles}); the embedding is such that each of the components of $\Sigma\setminus G$ is a topological disk. We let $G^*$ be the dual graph of $G$ with respect to the surface $\Sigma$ to which we glue a topological disk to each boundary component: this ensures that there is one vertex in~$G^*$ per boundary component of~$G$, we denote the set of those by~$V_{\opname{out}}(G^*)$. For any edge $e\in E(G)$ we write $e^*$ for its dual edge. If the graph is planar, we shall write $u_{\opname{out}}$ for the unique element of~$V_{\opname{out}}(G^*)$ corresponding to the unbounded face of~$G$.

We consider an Ising model on $G^*$. The specification of {boundary conditions} for this model consists in assigning a fixed value to each of the spins at vertices from~$V_{\opname{out}}(G^*)$. In particular~`$+$' \emph{boundary conditions} are obtained by fixing all of these values to~$+1$. The set of spin configurations on~$G^*$ with~`$+$' boundary conditions, is in bijection with the set of {\em domain walls\/} between clusters of~$+1$'s and~$-1$'s, i.e. the set
\[
\cE_0(G)~:=~\{P\in\cE(G)\,:\,[P]=0\in H_1(\Sigma;\Z_2)\}
\]
of even subgraphs of~$G$ that bound a collection of faces. (Although we do not write it explicitly in the notation, $\cE_0(G)$ implicitly depends on the embedding of $G$ in~$\Sigma$.) Given a spin configuration $\sigma$ on $G^*$ with `$+$' boundary conditions, we let $P(\sigma)\in\cE_0(G)$ be the even subgraph representing the domain walls of $\sigma$.
Let
\[
\textstyle\cZ_{\beta^*}^+(G^*,J)~=~\sum_{\sigma\in\{\pm 1\}^{V(G^*)}:~
\sigma_u=+1~\text{for~all}~u\in V_{\opname{out}}(G^*)}\,\exp[-\beta^* H^*(\sigma)]
\]
be the partition function of the Ising model on $G^*$ with~`$+$' boundary conditions and the inverse temperature~$\beta^*$, where the Hamiltonian~$H^*(\sigma)$ on~$G^*$ is defined similarly to~(\ref{eqn:intro-H}) via interaction constants~$J_{e^*}$. By defining the \emph{low-temperature polynomial} to be
\[
\textstyle \cZ_{\opname{low},\Sigma}(G,x)~:=~\sum_{P\in \cE_0(G)}x(P)\,,
\]
one readily has
\begin{equation}
\label{eqn:Z-Ising-low}
\textstyle \cZ_{\beta^*}^+(G^*,J)~=~ (\prod_{e^*\in E(G^*)}\exp[\beta^* J_{e^*}])\cdot \cZ_{\opname{low},\Sigma}(G,(x_e\!:=\!\exp[-2\beta^*J_{e^*}])_{e\in E(G)})\,.
\end{equation}

For planar graphs (i.e. when~$\Sigma=\C$ is the plane), we have~$\cE_0(G)=\cE(G)$ and simply denote
\[
\cZ_{\opname{Ising}}(G,x)~:=~\cZ_{\opname{low},\C}(G,x)=\cZ_{\opname{high}}(G,x)\,.
\]
Note that this equality relates the partition function of an Ising model on a planar graph~$G$ and another one on~$G^*$ provided the interaction constants and inverse temperatures satisfy, for each edge $e\in E(G)$ and its dual~$e^*\in E(G^*)$, the relation~$\tanh[\beta J_e]=\exp[-2\beta^*J_{e^*}]$, which can be rewritten in a symmetric way as
\[
\sinh[2\beta J_e]\sinh[2\beta^* J_{e^*}]=1\,.
\]
This is the Kramers-Wannier duality~\cite{Kramers-Wannier} and it has an extension to surface graphs~\cite{cimasoni-AIHP}.

One can also consider \emph{`free' boundary conditions} on (some of) the boundary components of~$\Sigma$ instead of~`$+$' ones. This can be obtained by setting the corresponding interaction parameters~$x_e=\exp[-2\beta J_{e^*}]$ to~$1$ in the right-hand side of~\eqref{eqn:Z-Ising-low} and modifying the prefactor accordingly. In particular, all results we present below for `$+$' boundary conditions can be easily generalized to `free' ones and we shall not comment about this further in the text.

We now briefly explain how the 2D Ising model is naturally associated to the topological notions of double covers and spin structures, see Sections~\ref{sub:intro-spin-correlations},~\ref{sub:grassmann},~\ref{sub:disorders} and~\ref{sec:surface} for more details.
Consider the Ising model on a planar graph~$G^*$ and fix some faces~$u_1,\dots,u_m$ of~$G$. In order to compute \emph{spin correlations}~$\mathbb{E}^+_{G^*}[\sigma_{u_1}\dots\sigma_{u_m}]$, one may take advantage of the domain walls expansion and twist the weights~$x_e$ by changing their signs on cuts linking~$u_1,\dots,u_m$ and~$u_{\opname{out}}$ in such a way that all configurations are weighted by~$\sigma_{u_1}\dots\sigma_{u_m}$. This gives a polynomial~$\cZ_{[u_1,..,u_m]}(G,x)$ such that~$\mathbb{E}^+_{G^*}[\sigma_{u_1}\dots\sigma_{u_m}]=\cZ_{[u_1,\ldots,u_m]}(G,x)/\cZ_{\opname{Ising}}(G,x)$ and the underlying topological structure is that of a canonical \emph{double cover} of~$G$ branching over all the faces~$u_1,\dots,u_m$. For planar graphs, this leads to a representation of spin correlations as ratios of two Pfaffians due to the well-known integrability of the 2D Ising model. There is another way to treat spin correlations: consider a punctured plane $\C\setminus\{u_1,\dots,u_m\}$ and note that the probability of $\sigma_{u_j}$ being $+1$ for all $j$ is simply the ratio~${\cZ_{\opname{low},\C\setminus\{u_1,..,u_m\}}(G,x)}\,/\,{\cZ_{\opname{Ising}}(G,x)}$. Therefore,~$\mathbb{E}^+_{G^*}[\sigma_{u_1}\dots\sigma_{u_m}]$ can be written as a linear combination of such ratios and, vice versa,~${\cZ_{\opname{low},\C\setminus\{u_1,..,u_m\}}(G,x)}$ is a linear combination of~$2^m$ Pfaffians corresponding to all the possible double covers of the punctured plane~$\Sigma=\C\setminus\{u_1,\dots,u_m\}$, which are classified by~$\Hom(\pi_1(\Sigma),\Z_2)=H^1(\Sigma;\Z_2)$. When one works with graphs embedded in a general surface~$\Sigma$, a similar phenomenology comes into play: the polynomial~$\cZ_{\opname{low},\Sigma}(G,x)$ is equal to a sum of several Pfaffians but one needs a clever topological tool to index them, the so-called \emph{spin structures}~\cite{Mercat-CMP,Cim2}, which form an affine space over~$H^1(\Sigma;\Z_2)$; see Section~\ref{sec:surface} for details.

Throughout the introduction, Section~\ref{sec:planar} and Section~\ref{sec:links} we assume that the finite weighted graph~$(G,x)$ is embedded in the plane, with edges given by straight line segments. However, in Section~\ref{sec:surface}, we show that most of these results (and proofs) extend to the general case of finite weighted graphs embedded in surfaces. The main additional tool needed is the notion of spin structures mentioned above.

\subsection{The Kac--Ward matrix and the terminal graph} \label{sub:intro-KW-and-terminal}
Let~$\EE(G)$ be the set of {\em oriented\/} edges of~$G$. We shall denote by~$o(e)\in V(G)$ the origin of~$e\in\EE(G)$, by~$t(e)\in V(G)$ its terminal vertex, by~$\overline{e}\in\EE(G)$ the oriented edge {with the same support as $e$ but the opposite orientation}, and extend~$x$ to a symmetric {(under change of orientation)} function on~$\EE(G)$.
Given two oriented edges~$e,e'$ such that~$t(e)=o(e')$ and~$e'\neq\overline{e}$, one can consider the oriented angle~$\rw(e,e')\in (-\pi,\pi)$ between~$e$ and~$e'$, see Fig.~\ref{Fig:Wind}.

The {\em Kac--Ward matrix\/} associated to the weighted graph~$(G,x)$ is the~$|{\EE}(G)|\times|{\EE}(G)|$ matrix
\[
\KW(G,x):=\rI-\rT\,,
\]
where~$\rI$ is the identity matrix and~$\rT$ is defined by
\begin{equation}
\label{eqn:T-def}
\rT_{e,e'}=\begin{cases}
\exp[\frac{i}{2}\rw(e,e')]\cdot(x_ex_{e'})^{1/2}& \text{if~$t(e)=o(e')$ but~$e'\neq \bar{e}$;} \\
0 & \text{otherwise.}
\end{cases}
\end{equation}
The famous \emph{Kac--Ward formula} \cite{Kac-Ward} claims that
\begin{equation}
\label{eqn:KWformula}
\det[\KW(G,x)]=[\cZ_{\opname{Ising}}(G,x)]^2,
\end{equation}
and it was an intricate story~\cite{Sherman,Vdovichenko,DZMSS,Kager-Lis-Meester,Helmuth} to give a fully rigorous proof of this identity for general planar graphs; see a recent paper~\cite{Lis-2015} by Lis for a streamlined version of the classical approach to~(\ref{eqn:KWformula}). This formula was generalized to graphs embedded in surfaces in~\cite{Cim2}. Note that the classical Kac--Ward matrix~$\KW(G,x)$ is neither Hermitian nor anti-symmetric but there is a simple {transformation} revealing these symmetries. Indeed, denote
\begin{equation}
\label{eqn:K-def}
\rK=\rK(G,x):=\rJ\cdot\KW(G,x),\quad\text{where}\quad \rJ_{e,e'}=\delta_{\overline{e},e'}
\end{equation}
(this multiplication by~$\rJ$ barely changes the determinant of~$\KW(G,x)$, which simply gets multiplied by~$\det\rJ=(-1)^{|E(G)|}$). Then,
\begin{equation}
\label{eqn:K-entries}
\rK_{e,e'}=\begin{cases}
1 &\text{if~$e'=\overline{e}$;}\\
-\exp[\frac{i}{2}\rw(\overline{e},e')]\cdot(x_ex_{e'})^{1/2}& \text{if~$o(e)=o(e')$ but~$e'\neq e$;} \\
0 & \text{otherwise;}
\end{cases}
\end{equation}
and it is easy to see that $\rK=\rK^*$. Furthermore, for each oriented edge~$e\in\EE(G)$, {let us fix a \emph{square root} of the direction of the straight segment representing~$e$ on the plane and denote by~$\eta_e$ its \emph{complex conjugate} multiplied by a fixed unimodular factor $\zeta$. The latter factor plays no role in the most part of the paper (and hence the reader can simply think of $\zeta=1$) except Sections~\ref{sub:s-observables} and~\ref{sub:dblI-s-hol-fct} where it is convenient to use the value~$\zeta=e^{i\frac{\pi}{4}}$.} Denote by~$\rU$ the diagonal matrix with coefficients~$\{\eta_e\}_{e\in\EE(G)}$, and set
\begin{equation}
\label{eqn:hat-K-def}
\wh{\rK}:={i\rU^*\rK\rU^{\vphantom{*}}}\,.
\end{equation}
Because of {the above} choice of square roots, the matrix $\wh{\rK}$ is not canonical, whereas $\rK$ is (given the embedded graph).
It is easy to see that the matrix $\wh{\rK}$ is anti-Hermitian with \emph{real} entries, and thus \emph{anti-symmetric}. Moreover, it has the same determinant as the original Kac--Ward matrix~$\KW(G,x)$. The last simple observation is that $\wh{\rK}$ can be thought of as a weighted adjacency matrix of the \emph{terminal graph}~$G^\rK$ which was introduced by Kasteleyn~\cite[Section~V]{Kasteleyn-63} and initially called the ``cluster lattice''. Let us now recall the definition of this graph.

\begin{figure}
\captionsetup[subfigure]{position=below,justification=justified,singlelinecheck=false,labelfont=bf}
\begin{subfigure}[t]{.48\textwidth}
\labellist\small\hair 2.5pt
\pinlabel {$e$} at 46 125
\pinlabel {$\rw(e,e')$} at 310 105
\pinlabel {$e'$} at 300 210
\endlabellist
\centering{\psfig{file=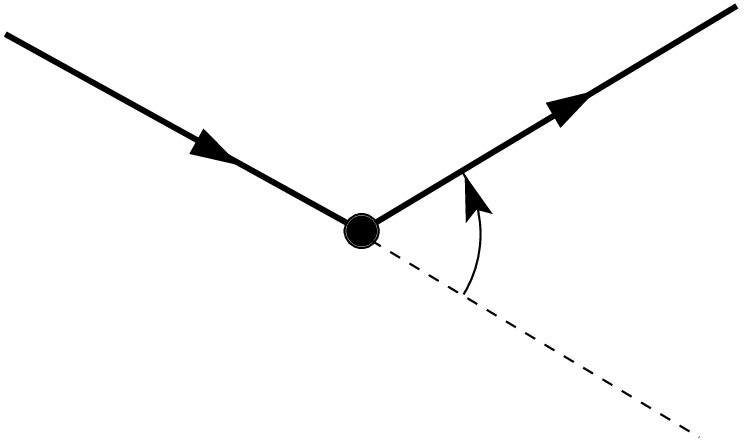,width=.7\linewidth}}
\caption{The angle $\rw(e,e')$ for two edges $e,e'$ satisfying $t(e)=o(e')$}
\label{Fig:Wind}
\end{subfigure}
\hfill
\begin{subfigure}[t]{.48\textwidth}
\labellist\small\hair 2.5pt
\pinlabel {$x_1$} at 237 162
\pinlabel {$x_2$} at 130 262
\pinlabel {$x_d$} at 185 14
\pinlabel {$1$} at 670 140
\pinlabel {$1$} at 520 262
\pinlabel {$1$} at 585 14
\pinlabel {$(x_1x_2)^{1/2}$} at 645 195
\pinlabel {$(x_dx_1)^{1/2}$} at 635 90
\endlabellist
\centering{\psfig{file=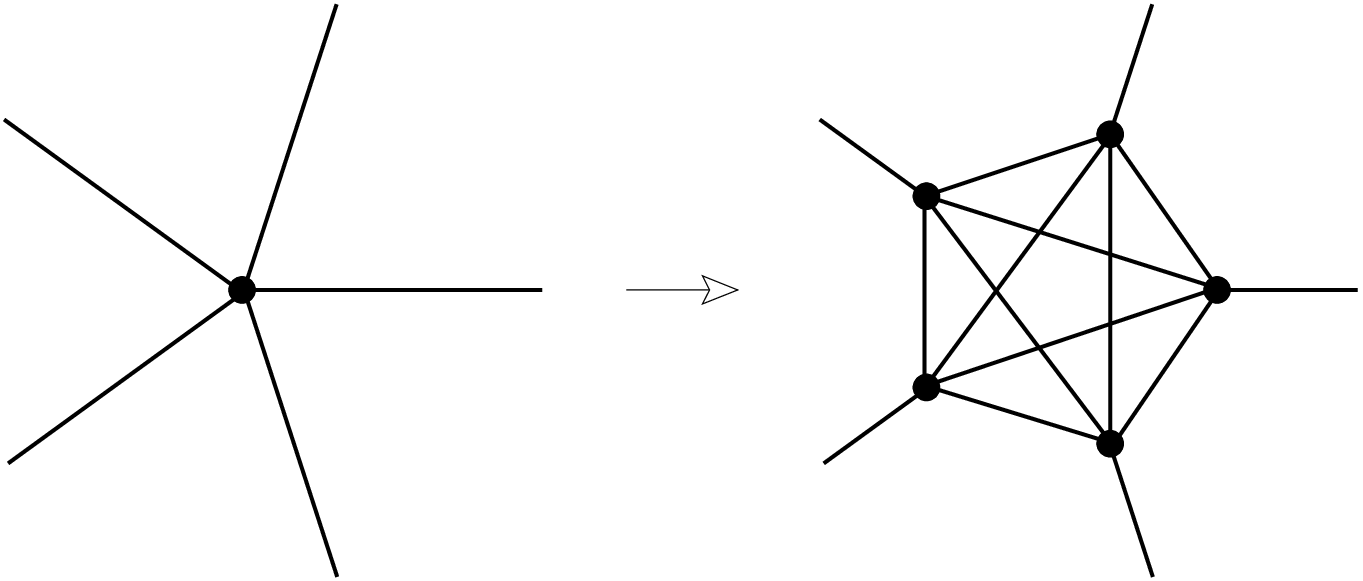,width=\linewidth}}
\caption{The local procedure used to construct the terminal graph $G^\rK$ from $G$}
\label{Fig:GK}
\end{subfigure}
\caption{}
\end{figure}

Given a weighted graph~$(G,x)$, its terminal graph $G^\rK$ is obtained by inserting at each vertex~$v$ of~$G$ (of degree~$d(v)$) a clique~$K_{d(v)}$, as illustrated in Fig.~\ref{Fig:GK}. We shall say that an edge of~$G^\rK$ is {\em short\/} if it is part of one of these complete subgraphs, and {\em long\/} otherwise (i.e. if it comes from an edge of~$G$). Given edge weights~$x=(x_e)_{e\in E}$ on~$G$, we shall denote by~$x^\rK$ the edge weights on~$G^\rK$ obtained by assigning weight~$1$ to all long edges and weight~$(x_ex_{e'})^{1/2}$ to the short edge corresponding to the two adjacent edges~$e,e'\in \EE(G)$. Note that the mapping of an oriented edge of~$G$ to the origin vertex of the corresponding long edge in~$G^\rK$ yields a natural bijection between the sets~$\EE(G)$ and~$V(G^\rK)$, which we use to identify them. Note that the terminal graph is in general neither planar nor bipartite.

Surprisingly enough, this, almost trivial, link between the two combinatorial techniques: expansions of the Kac--Ward determinant and the study of dimers on the terminal graph, seems to have remained almost unknown to date. It is even more astonishing that some version of the above reduction of~$\KW(G,x)$ to~$\wh\rK$ was known as early as~1960 to Hurst and Green, who worked with the translationally invariant Ising model on the square lattice and mentioned it to claim the first complete derivation of the Kac--Ward formula in this particular case~\cite[p.~1062]{Hurst-Green}.

\subsection{Pfaffian formulas for the partition function and combinatorial observables} \label{sub:intro-Pfaffian}

Recall that~$\wh{\rK}$ is a real anti-symmetric matrix obtained from the Kac--Ward matrix~$\KW(G,x)$ via~\eqref{eqn:K-def} and~\eqref{eqn:hat-K-def}. In particular,
\[
\det[\KW(G,x)]~=~\det \wh \rK~=~(\Pf{\wh\rK})^2\,.
\]

\begin{theorem}
\label{thm:KW1} For any planar weighted graph~$(G,x)$, one has
\[
\cZ_{\opname{Ising}}(G,x)=\pm \Pf{\wh{\rK}}\,,
\]
where the sign in the right-hand side is fixed by the condition that the constant (in $x$) term equals $+1$. As a consequence, the Kac--Ward formula~\eqref{eqn:KWformula} holds for any planar graph~$(G,x)$.
\end{theorem}

\begin{remark}
On the square lattice, this theorem provides a standard way to express the partition function of the Ising model in terms of the so-called Grassmann variables, see further details in Section~\ref{sub:grassmann}. For this particular case, its proof can be found in many textbooks but not in connection with the Kac--Ward formula: the matrix~$\wh\rK$, corresponding to some particular choices of~$\eta_e$ for oriented edges of four different types, is introduced \emph{per~se}, as a clever tool to count the \emph{signed} partition function of dimers (aka perfect matchings) on the corresponding \emph{non-planar} terminal lattice. (Such a partition function can be easily seen to be equal to~$\cZ_{\opname{Ising}}(G,x)$, see Section~\ref{sub:prelim}.) To the best of our knowledge, it does not appear in the literature in the full generality of finite planar weighted graphs. However, its key ingredients were known to Kasteleyn and Fisher in the~1960s though presumably not in connection with the Kac--Ward matrix and the induced orientations of~$G^\rK$; see~\cite[Section~V]{Kasteleyn-63} and~\cite[Section~1]{Fisher-66} and note that the descriptions of configurations $P\in\cE(G)$ via dimers on~$G^\rK$ used in~\cite{Kasteleyn-63} and~\cite{Fisher-66} differ from one another.
\end{remark}

It is well-known~\cite{Lis-2013,cimasoni-AIHP} that the \emph{entries} of the inverse Kac--Ward matrix can be represented as the \emph{two-point combinatorial observables} proposed by Smirnov~\cite{smirnov-icm-2006,Smirnov-10} as a convenient tool to study the scaling limit of the critical Ising model in arbitrary planar domains. Since then, these observables are usually defined in a self-contained way, but it is worth noting that their definition grew from considerations made by Smirnov jointly with Kenyon~\cite[Remark~4]{smirnov-icm-2010} on dimer techniques applied to the so-called Fisher graph; see more details in Section~\ref{sec:links}.

The next theorem, Theorem~\ref{thm:multipoint}, extends this combinatorial interpretation to the \emph{Pfaffian minors} (that is Pfaffian of square submatrices) of the inverse matrix~$\wh\rK^{-1}$, which correspond to the~\emph{$2n$-point observables}. The latter were recently used~\cite{Hongler-thesis,Chelkak-Izyurov,izyurov-mult} in the context of the critical Ising model but, to the best of our knowledge, Theorem~\ref{thm:multipoint} does not appear in the literature in this generality (e.g., the proofs given in~\cite{Hongler-thesis,Chelkak-Izyurov,izyurov-mult} rely upon some particular feature of the model at criticality, see also a discussion in~\cite[Section~4.5]{hongler-kytola-zahabi}). This expansion is also important to justify the link between the two classical formalisms developed to study the 2D Ising model: Grassmann variables and disorder insertions; see further details in Section~\ref{sub:equivalence}.

We need some notation. Let~$(G^\diamondsuit,x)$ be the weighted graph obtained from~$(G,x)$ by adding a vertex~$z_e$ in the middle of each edge~$e$ of~$G$,
and by assigning the weight~${x_e}^{1/2}$ to both resulting edges of~$G^\diamondsuit$. Given a collection~$\rE=\{e_1,\dots,e_{2n}\}$ of oriented edges of~$G$,
let~$\mathcal{C}(e_1,\dots,e_{2n})$ denote the set of subgraphs~$P$ of~$G^\diamondsuit$ that do \emph{not} contain the
edges~$(o(e_k),z_{e_k})$, \emph{do} contain the edges~$(z_{e_k},t(e_k))$ provided~$\overline{e}_k\not\in\rE$, and such that each vertex of~$G^\diamondsuit$ different from~$z_{e_1},\dots,z_{e_{2n}}$ has an even degree in~$P$.
Note that~$\cC(\emptyset)$ is nothing but the set~$\cE(G)$ of even subgraphs of~$G$.

To each configuration~$P\in\mathcal{C}(e_1,\dots,e_{2n})$, we shall now assign a sign~$\tau(P)\in\{\pm 1\}$. In order to do so, we resolve all its crossings (vertices with degree more than~$2$ in~$P$) to obtain a decomposition~$P=C\sqcup\gamma_1\sqcup\dots\sqcup\gamma_n$, where~$C$ is a collection of disjoint simple loops, and~$\gamma_1,\dots,\gamma_n$ are simple paths matching the half-edges~$(z_{e_1},t(e_1)),\dots,(z_{e_{2n}},t(e_{2n}))$; in case there are pairs~$e,\overline{e}$ in~$\rE$, we declare the corresponding~$\gamma_k$ to be empty paths formally matching such pairs. Let us choose arbitrary orientations of the paths~$\gamma_k$ and denote by~$s$ a permutation of~$\{1,\dots,2n\}$ such that each of~$\gamma_k$ goes from~$z_{e_{s(2k-1)}}$ to~$z_{e_{s(2k\vphantom{1})}}$. Following~\cite{Hongler-thesis}, we set
\begin{equation}
\label{eqn:tau-def}
\tau(P)~:=~{\sign(s)\,\cdot}\prod_{k=1}^n \left({i{\eta}_{e_{s(2k-1)}}^{\vphantom{|}}\overline{\eta}_{e_{s(2k\vphantom{1})}}} \exp[-{\textstyle\frac{i}{2}}\wind(\gamma_k)]\right),
\end{equation}
where~$\wind(\gamma_k)$ denotes the total rotation angle of the velocity vector of~$\gamma_k$ when it runs from~$z_{e_{s(2k-1)}}$ to~$z_{e_{s(2k\vphantom{1})}}$; we formally set~$\wind(\gamma_k):=0$ in case $e_{s(2k-1)}=\overline{e}_{s(2k\vphantom{1})}$ and so~$\gamma_k=\emptyset$.

The sign~$\tau(P)$ is obviously independent of the numbering of the paths~$\gamma_k$ and one can easily see that it also does not depend on their orientations. Moreover, one can check that it is independent of the smoothing of~$P$, i.e. the way how~$P$ is split into~$C$ and~$\gamma_k$, see identity~\eqref{eqn:claim-tau} in Section~\ref{sec:planar} for more comments. We are now able to formulate the next result.

\begin{theorem}
\label{thm:multipoint}
For any planar weighted graph~$(G,x)$ and any set of oriented edges~$e_1,\dots,e_{2n}$, the following combinatorial expansion is fulfilled:
\[
\Pf{\wh{\rK}^{-1}_{e_j,e_k}}_{j,k=1}^{2n} ~=~ [\cZ_{\opname{Ising}}(G,x)]^{-1}\cdot\!\!\! \sum_{P\in\mathcal{C}(e_1,\dots,e_{2n})}\tau(P)\,x(P)\,.
\]
\end{theorem}
\begin{remark}
The case~$n=1$ leads to the standard combinatorial definition of two-point observables as sums over the set~$\cC(e,e')$, while for~$n>1$ one recovers the combinatorial definition of multi-point observables as sums over~$\cC(e_1,\dots,e_{2n})$, cf.~\cite{Hongler-thesis,Chelkak-Izyurov,izyurov-mult}. Theorem~\ref{thm:multipoint} claims that the latter are Pfaffians of the former, see also Section~\ref{sub:s-observables} for the discussion of the combinatorial definition of complex-valued fermionic observables.
\end{remark}

\subsection{Spin and energy density correlations} \label{sub:intro-spin-correlations}
We now move on to combinatorial formulas for spin correlations.
In this section we deal with the (domain walls expansion of the) Ising model on the dual graph~$G^*$ or, equivalently, with the Ising model on \emph{faces} of~$G$ with~`$+$' boundary conditions, which means that we fix the spin of the outer face~$u_{\opname{out}}$ of~$G$ to be~$+1$. Given faces~$u_1,\dots,u_m$ of~$G$, let us fix some collection $\varkappa=\varkappa_{[u_1,..,u_m]}$ of edge-disjoint paths on~$G^*$, which link $u_1,\dots,u_m$ and, possibly,~$u_{\opname{out}}$ so that each of~$u_1,\dots,u_m\in V(G^*)$ has an odd degree in the union of these paths. Further, let~$\rI_{[u_1,..,u_m]}$ denote the diagonal matrix with entries
\[
\left(\rI_{[u_1,..,u_m]}\right)_{e,e}=\begin{cases}-1 & \text{if}~e\in\EE(G)~\text{intersects}~\varkappa; \\ +1 & \text{otherwise.}\end{cases}
\]
We now define the modified Kac--Ward matrix, labeled by oriented edges of $G$, to be
\[
\KW_{[u_1,..,u_m]}=\KW_{[u_1,..,u_m]}(G,x):=\rI_{[u_1,..,u_m]}-\rT\,,
\]
where~$\rT$ is given by~\eqref{eqn:T-def}. Similarly to~\eqref{eqn:K-def},~\eqref{eqn:hat-K-def}, we define
\[
\rK_{[u_1,..,u_m]}=\rK_{[u_1,..,u_m]}(G,x):=\rJ\cdot \KW_{[u_1,..,u_m]}(G,x)\quad\text{and}\quad \wh{\rK}_{[u_1,..,u_m]}:={ i\rU^*\rK_{[u_1,..,u_m]}\rU}\,.
\]
Note that all these matrices depend on the choice of the collection of paths~$\varkappa=\varkappa_{[u_1,..,u_m]}$ linking the faces~$u_1,\dots,u_m$ and~$u_{\opname{out}}$, which is implicit in the notation. Let~$|\varkappa|$ denote the number of edges in~$\varkappa$. The following result is a simple consequence of Theorem~\ref{thm:KW1}.
\begin{proposition}
\label{prop:spin1}
For any planar weighted graph~$(G,x)$ and $u_1,\dots,u_m\in V(G^*)$, we have
\begin{equation} \label{eqn:multi-spin-plane}
\mathbb{E}^+_{G^*}[\sigma_{u_1}\dots\sigma_{u_m}]~=~(-1)^{|\varkappa|}\frac{\Pf{\wh{\rK}_{[u_1,..,u_m]}}}{\Pf{\wh{\rK}}}~=~\pm \left[\frac{\det\KW_{[u_1,\dots,u_m]}}{\det\KW}\right]^{1/2},
\end{equation}
where~$\mathbb{E}^+_{G^*}$ denotes the expectation in the Ising model at inverse temperature $\beta$ on the dual graph~$G^*$ conditional on~$\sigma_{u_{\opname{out}}}=+1$, and $x_e=\exp[-2\beta J_{e^*}]$ for all~$e\in E(G)$.
\end{proposition}

\begin{remark}
Equation~\eqref{eqn:multi-spin-plane} can be rewritten as a ratio of determinants of discrete $\bar\partial$-type operators; see Section~\ref{sub:three-term} in particular Remark~\ref{rem:ad-hoc-combinatorics} and Remark~\ref{rk:d-bar}(i). These operators share many important properties with their continuous counterparts, especially if one starts with the self-dual Ising model considered on the so-called isoradial graphs (we refer the reader interested in this subject to~\cite{Mercat-CMP,Kenyon-Dirac,chelkak-smirnov-adv} and~\cite{dub-dimers}). Identities similar to~(\ref{eqn:multi-spin-plane}) also appeared in~\cite{BdT4,Dub} in connection with the double-Ising model and dimer techniques for it.
\end{remark}

 We now focus on the particular case when~$m=2n$ is even and~$u_1,\dots,u_m$ are given by~$n$ pairs of neighboring faces~$u_{2k-1},u_{2k\vphantom{1}}$, each pair being separated by an edge~$e_k\in E(G)$. Let
\[
\reps_{e_k}:=\sigma_{u_{2k-1}}\sigma_{u_{2k\vphantom{1}}}
\]
denote the so-called \emph{energy density} at the edge~$e_k$ and let~$\rE=\cup_{k=1}^n\{e_k,\bar{e}_k\}\subset\EE(G)$ be the corresponding set of $2n$ oriented edges (each edge $e_k$ is taken along with is reverse $\bar{e}_k$).
In this case, there is a natural choice of the collection of paths~$\varkappa=\varkappa_{[u_1,..,u_{2n}]}$ linking~$u_1,\dots,u_{2n}$ simply given by taking all the dual edges~$e_1^*,\dots,e_n^*$. For this choice of~$\varkappa$, Proposition~\ref{prop:spin1} reads
\[
\mathbb{E}^+_{G^*}[\reps_{e_1}\dots\reps_{e_n}]~=~(-1)^{n}\frac{\Pf{\wh{\rK}-2\wh{\rJ}_\rE}}{\Pf{\wh{\rK}}}\,,\quad\text{with}\quad ({\wh\rJ}_\rE)_{e,e'}=\begin{cases}{i\overline{\eta}_e\eta_{e'}} & \text{if}~e'=\overline{e}; \\ 0 & \text{otherwise.}\end{cases}
\]
Note that~$\wh{\rJ}_\rE$ is a real anti-symmetric matrix with~$\pm 1$ entries depending on the choices of~$\eta_{e_k}$. Since~$\wh{\rJ}_\rE$ vanishes on~$\EE(G)\setminus\rE$ and~$-\wh{\rJ}_\rE^{\,2}$ is the identity matrix on~$\rE$, we have
\[
\det [\,\rI-2\wh{\rJ}_\rE \wh{\rK}^{-1}\,]~=~ \det [\,\rI-2\wh{\rJ}_\rE \wh{\rK}^{-1}\,]_{e,e'\in\rE} ~=~
\det [\,\wh{\rJ}_\rE+2\wh{\rK}^{-1}\,]_{e,e'\in\rE}\,,
\]
and hence
\[
\mathbb{E}^+_{G^*}[\reps_{e_1}\dots\reps_{e_n}]= \pm\Pf{ \wh{\rJ}_\rE\!+\!2\wh{\rK}^{-1}}_{e,e'\in\rE}\,,
\]
with the $\pm$ sign depending on the ordering of~$\rE=\{e_1,\overline{e}_1,\dots,e_n,\overline{e}_n\}$ and the choices of~$\eta_{e_k}$.

\begin{remark} \label{rem:formulas-via-K^-1}
\emph{(i)} This Pfaffian formula for multi-point energy density expectations can be also deduced from Theorem~\ref{thm:multipoint}. Indeed, if~$e\in E(G)$ and a spin configuration~$\sigma\in\{\pm 1\}^{V(G^*)}$ is encoded by domain walls~$P\in\cE(G)$, then~$\frac{1}{2}(\reps_{e}+1)$ is the indicator of the event~$e\not\in P$. On the other hand,
\[
\cC(e_1,\overline{e}_1,\dots,e_n,\overline{e}_n)=\{P\in\cE(G):e_1,\dots,e_n\notin P\},
\]
and the sign~$\tau(P)=\prod_{k=1}^n({i\eta_{e_k}^{\vphantom{|}}\!\overline{\eta}_{\overline{e}_k}})$ is independent of~$P$ on this set. Therefore, if one chooses a proper ordering of the oriented edges of~$G$ according to the choices of~$\eta_{e_k}$, Theorem~\ref{thm:multipoint} implies
\begin{equation}\label{eqn:energy-Pfaff}
\mathbb{E}^+_{G^*}[{\textstyle\frac{1}{2}}(\reps_{e_1}\!+1)\dots{\textstyle\frac{1}{2}}(\reps_{e_n}\!+1)]= \Pf{\wh{\rK}^{-1}}_{e,e'\in\rE}\,.
\end{equation}
In other words, edges of~$G$ carrying the values~$\reps_e=+1$ form a Pfaffian process with the kernel~$\wh{\rK}^{-1}$, and the formula for~$\mathbb{E}^+_{G^*}[\reps_{e_1}\dots\reps_{e_n}]$ given above easily follows {by multilinearity}. This approach was used in~\cite{Hongler-thesis}
to prove the existence of scaling limits for multi-point energy density correlations in general simply-connected domains at criticality; see more comments in Section~\ref{sub:s-observables}.

\smallskip

\noindent \emph{(ii)} In a similar manner, one can consider the faces $u_1,\ldots,u_m$ in formula~\eqref{eqn:multi-spin-plane} as variables. If we then replace~$u_1$ by one of its neighboring faces~$u_1'$ (or, more generally, move each~$u_k$ by several faces to some other face~$u_k'$) and adjust the collection of cuts~$\varkappa_{[u_1,..,u_m]}$ accordingly, then the matrix~$\wh\rK_{[u_1',\dots,u_m']}$ is a small rank perturbation of the matrix $\wh\rK_{[u_1,\dots,u_m]}$ and hence the ratio
\[
\frac{\mathbb{E}^+_{G^*}[\sigma_{u_1'}\dots\sigma_{u_m'}]}{\mathbb{E}^+_{G^*}[\sigma_{u_1}\dots\sigma_{u_m}]}= \pm\frac{\Pf{\wh\rK_{[u_1',\dots,u_m']}}}{\Pf{\wh\rK_{[u_1,\dots,u_m]}}}
\]
admits a simple expression in the entries of $\wh\rK^{-1}_{[u_1,\dots,u_m]}$. In particular, this yields a short proof of~\cite[Lemma 2.6]{CHI} which is a starting point for the analysis of the scaling limit of multi-point spin correlations in general simply-connected domains at criticality. This observation can be also used for the systematic study of other spin pattern correlations, cf.~\cite{hongler-and-co}.
\end{remark}

\subsection{The double-Ising model} \label{sub:intro-dblI}
The aim of this section is to indicate that all the combinatorial formulas discussed above admit modifications for the so-called \emph{double}-Ising model which is defined as a pointwise product of two independent Ising models on the (faces) of the same planar weighted graph~$(G,x)$, \emph{coupled along the boundary} in a way which we now describe. Below we assume that the graph~$G$ contains a number of univalent (i.e. degree~$1$) vertices incident to the outer face~$u_{\opname{out}}$; note that adding/removing such vertices does not affect the Ising model defined on faces of~$G$, nor the value~$\cZ_{\opname{Ising}}(G,x)$. We call such vertices the \emph{boundary vertices} of~$G$ and edges linking them to the bulk of~$G$ the \emph{boundary edges} of~$G$.

\begin{figure}
\captionsetup[subfigure]{position=below,justification=justified,singlelinecheck=false,labelfont=bf}
\begin{subfigure}[t]{.47\textwidth}
\labellist\small\hair 2.5pt
\pinlabel {$u_{\opname{out}}$} at 25 80
\pinlabel {$1$} at 210 80
\pinlabel {$1$} at 250 92
\pinlabel {$1$} at 290 85
\pinlabel {$1$} at 325 83
\endlabellist
\centering{\psfig{file=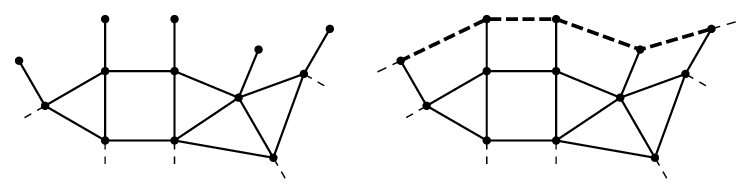,width=\linewidth}}
\caption{A portion of the boundary of a graph~$G$ with univalent vertices and the corresponding portion of $\widetilde{G}$ where the additional edges are dashed}\label{dbl-I-tilde}
\end{subfigure}
\hfill
\begin{subfigure}[t]{.47\textwidth}
\labellist\small\hair 2.5pt
\pinlabel {$x_e$} at 109 102
\pinlabel {$i x_e$} at 320 130
\endlabellist
\centering{\psfig{file=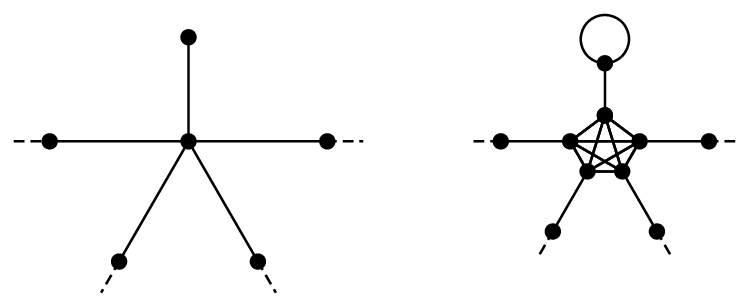,width=.9\linewidth}}
\caption{The mapping of a univalent boundary vertex to a self-loop attached to the terminal graph}\label{dbl-I-1}
\end{subfigure}
\hfill
\caption{}
\end{figure}

Let us add auxiliary edges carrying weights~$1$, linking the boundary vertices of~$G$ inside of~$u_{\opname{out}}$ in a cyclic order, to obtain a weighted graph~$(\widetilde{G},x)$; see Figure~\ref{dbl-I-tilde}. The faces of~$\widetilde{G}$ adjacent to the boundary edges of~$G$ are called~\emph{boundary faces}. The states of the associated \emph{double-Ising model} (with `$+$' boundary conditions) are pairs of spin configurations~$\sigma,\sigma'$ on the faces of the new graph~$\widetilde{G}$ such that
\[
\widetilde{\sigma}_u:=\sigma_u\sigma'_u=+1\quad\text{for~all~boundary~faces~$u$.}
\]
Similarly to the domain walls representation of the single-Ising model discussed in Section~\ref{sub:intro-Z-Ising}, these spin configurations can be encoded by \emph{pairs} of domain walls~$P,P'\in\cE(\widetilde{G})$ such that each boundary edge of~$G$ is either occupied by both these walls, or by none. The associated partition function is given by
\[
\cZ_{\opname{dbl-I}}(G,x)~=~\sum\nolimits_{P,P'\in\cE(\widetilde{G})\,:\, (P\btu P')\cap E_\partial(G)=\emptyset} x(P)x(P')\,,
\]
where~$E_\partial(G)\subset E(G)$ denotes the set of (unoriented) boundary edges of~$G$ and $\Delta$ stands for the symmetric difference. To compute this partition function, we introduce a modified Kac--Ward matrix~$\widetilde{\rK}$, indexed by the set of oriented edges of~$G$ in the same way as the matrix~$\rK$ given by~\eqref{eqn:K-entries}, with the entries
\[
\widetilde{\rK}_{e,e'}=\rK_{e,e'}+ \begin{cases}ix_e &\text{if~$e'=e$ is an inward oriented boundary edge;} \\ 0 & \text{otherwise.} \end{cases}
\]
Note that $\widetilde{\rK}$ can be understood as the weighted adjacency matrix of the graph obtained from~$G^\rK$ by adding a loop to each vertex corresponding to an inward oriented boundary edge; see Figure~\ref{dbl-I-1}. The following result is an analog of Theorem~\ref{thm:KW1} for the double-Ising model.

\begin{theorem}
\label{thm:KW3}
For any planar weighted graph~$(G,x)$, one has
\[
\cZ_{\opname{dbl-I}}(G,x)=(-1)^{|E(G)|}\det \widetilde{\rK}\,.
\]
\end{theorem}

Note that, contrary to the Kac--Ward formula~\eqref{eqn:KWformula} and Theorem~\ref{thm:KW1}, the determinant of the modified matrix~$\widetilde{\rK}$ cannot be written as the square of a Pfaffian, since the matrix~${i\rU^*\widetilde{\rK}\rU}$ contains non-vanishing diagonal entries and thus is not anti-symmetric. This reflects the fact that~$\cZ_{\opname{dbl-I}}(G,x)$ is \emph{not} the square of any single-Ising model partition function since the two spin configurations~$\sigma,\sigma'$ are now coupled along~$u_\mathrm{out}$.

Similarly to Section~\ref{sub:intro-spin-correlations}, for a given collection of (inner) faces~$u_1,\dots,u_m$, let~$\varkappa=\varkappa_{[u_1,..,u_m]}$ be a collection of paths in the dual graph~$G^*$ linking these faces to each other and, possibly, to the outer face~$u_{\opname{out}}$. Denote
\[
\widetilde{\rK}_{[u_1,..,u_m]}:= \widetilde{\rK}-\rJ+\rJ\cdot\rI_{[u_1,..,u_m]}\,.
\]
In other words, {to construct~$\widetilde{\rK}_{[u_1,..,u_m]}$ we replace by~$-1$ all the entries~$\widetilde{\rK}_{e,\overline{e}}=+1$ of~$\widetilde{\rK}$ that correspond to the edges~$e$ intersecting with~$\varkappa_{[u_1,..,u_m]}$,
exactly as in the definition of the matrix~$\rK_{[u_1,..,u_m]}$ in Section~\ref{sub:intro-spin-correlations}.}

\begin{proposition} \label{prop:spin4}
Let~$u_1,\dots,u_m$ be a collection of (inner) faces of~$G$. Then,
\[
\mathbb{E}_{\opname{dbl-I}}^+ [\widetilde{\sigma}_{u_1}\dots\widetilde{\sigma}_{u_m}] = \frac{\det\widetilde{\rK}_{[u_1,..,u_m]}}{\det\widetilde{\rK}}\,,
\]
where~$\mathbb{E}_{\opname{dbl-I}}^+$ stands for the expectation in the double-Ising model with~`$+$' boundary conditions.
\end{proposition}

We now briefly discuss {geometric} objects arising in the double-Ising model, the so-called \emph{XOR-Ising loops}, see \cite{Wilson-xor, BdT-xor}. For simplicity, let us assume that the graph~$G$ is trivalent (except at boundary vertices, which have degree~$1$). Then, given a double-Ising model configuration~$\widetilde{\sigma}=\sigma\sigma'$, the set~$P(\widetilde{\sigma})\in \cE(G)$ of edges separating the faces~$u$ with~$\widetilde{\sigma}_u=+1$ from those with~$\widetilde{\sigma}_u=-1$ is a collection of non-intersecting loops, which can be thought about as a result of the XOR (exclusive-or) operation applied to the two single-Ising domain walls configurations~$P(\sigma)$ and~$P(\sigma')$.

In~\cite{Wilson-xor}, Wilson conjectured that the scaling limits of these XOR-loops in the \emph{critical}
model (considered in discrete domains drawn on the honeycomb lattice) can be described as the union of level sets of the Gaussian Free Field with an appropriately tuned spacing. Recently, this conjecture was strongly supported by the results of Boutillier and de Tili\`ere~\cite{BdT-xor}, who showed that, at the discrete level, these loops have the same distribution as contour lines of a \emph{single}-dimer height function on a related bipartite graph. At the same time, the convergence result for these height functions known to date does not allow one to derive enough information about the behavior of their level lines, so one can wonder if some generalization of Theorem~\ref{thm:KW3} could help for that matter. Let us briefly discuss why this is not straightforward.

\begin{remark} A natural idea, motivated by the recent works of Kenyon~\cite{Kenyon-bundle-Laplacian,Kenyon-dd} and Dub\'edat~\cite{Dubedat-isomonodromy} on the double-dimer model, would be to study a twisted partition function of the double-Ising model in order to track the topology of the loops using some \emph{quaternionic} version of the matrix~$\widetilde{\rK}$. Unfortunately, the combinatorial expansions of such Q-determinants are no longer given by weighted sums over double-Ising configurations and additional terms come into play, similarly to the odd-length cycles in the double-dimer model on a non-bipartite graph, see~\cite[p.~482]{Kenyon-dd}. Nevertheless, it seems worthwhile to understand the arising expansions better and in particular one can try to interpret these additional terms as encoding some interaction between the loops.
\end{remark}

One can also try another common strategy and focus on a single interface (domain wall)~$\gamma_{a,b}$ generated by the so-called \emph{Dobrushin boundary conditions}. The latter are defined as follows: for a given pair~$a,b$ of boundary edges, let us condition the two spin configurations~$\sigma,\sigma'$ to satisfy
\begin{equation}
\label{eqn:dblI-Dobrushin-bc}
\widetilde{\sigma}_u=\sigma_u\sigma'_u= \begin{cases} -1 & \text{for~boundary~faces~$u$~on~the~arc~$(ab)$;} \\
+1 & \text{for~boundary~faces~$u$~on~the~arc~$(ba)$,} \end{cases}
\end{equation}
where~$(ab)$ (resp.~$(ba)$) denotes the part of the boundary of~$G$ from~$a$ to~$b$ when going counterclockwise (resp. clockwise). In Section~\ref{sub:dblI-Dobrushin} we prove an analog of Theorem~\ref{thm:KW3} for these boundary conditions and discuss how one can construct the so-called s-holomorphic martingales, which track the evolution of~$\gamma_{a,b}$.
This could pave a way to the understanding of a scaling limit of these interfaces, e.g. following the strategy implemented for the critical (single-) Ising model in~\cite{chelkak-smirnov,CDCHKS}. Nevertheless, it is also not straightforward, and we expect some conceptual obstacles when passing to a limit in the arising discrete boundary value problems for these s-holomorphic functions, see Section~\ref{sub:dblI-s-hol-fct} and Remark~\ref{rem:dblI-martingale-strategy-oops} for more details.

\addtocontents{toc}{\SkipTocEntry}
\subsection*{Acknowledgements}

We thank C\'edric Boutillier, Sunil Chhita, B\'eatrice de Tili\`{e}re, Hugo Duminil-Copin, Alexander Glazman, Cl\'ement Hongler, Konstantin Izyurov, Richard Kenyon and Stanislav Smirnov for many useful discussions on the subject, as well as Steffen Rohde, Wendelin Werner and David Wilson for their kind feedback on the double-Ising model. Also, we would like to thank the referee for useful suggestions.
This work started while the first-named author was visiting the Forschungsinstitut f\"{u}r Mathematik and completed at the Institut for Theoretical Studies at ETH Z\"urich. The support of Dr.~Max R\"ossler, the Walter Haefner Foundation and the ETH Foundation is gratefully acknowledged. The work of the second-named author was supported by a grant of the Swiss FNS.


\setcounter{equation}{0}
\section{The planar case}
\label{sec:planar}
We start this section with some preliminaries, then prove Theorem~\ref{thm:KW1}, and then show how the proof extends to give Theorem~\ref{thm:multipoint}. We conclude this section with the proof of Proposition~\ref{prop:spin1} and a discussion of the corresponding generalization of Theorem~\ref{thm:multipoint}.

\subsection{Preliminaries}
\label{sub:prelim}
Recall that~$\wh{\rK}$ is nothing but the signed skew-symmetric adjacency matrix of the weighted terminal graph~$(G^\rK,x^\rK)$. Namely, for two adjacent vertices~$e$ and~$e'$ of~$G^\rK$ (which are identified with oriented edges of~$G$), {we have~$\wh{\rK}_{e,e'}=\veps_{e,e'}\cdot x_{e,e'}^\rK$} with the sign~$\veps_{e,e'}=\pm 1$ given by
\begin{equation}
\label{eqn:veps-def}
\veps_{e,e'}=\begin{cases}
{i\overline{\eta}_e\eta_{e'}} &\text{if~$e$ and~$e'$ are linked by a long edge of~$G^\rK$;}\\
{-i\overline{\eta}_e\eta_{e'}}\exp[\frac{i}{2}\rw(\overline{e},e')]& \text{if~$e$ and~$e'$ are linked by a short edge of~$G^\rK$.}
\end{cases}
\end{equation}

Further, recall that a {\em dimer configuration\/} (aka {\em perfect matching\/}) on a graph~$\Gamma$ is a collection of edges of~$\Gamma$ {(called dimers)} such that each vertex
of~$\Gamma$ is incident to exactly one of these edges. We shall denote by~$\mathcal{D}(\Gamma)$ the set of dimer configurations on~$\Gamma$.

For $D\in\mathcal{D}(G^\rK)$, let $x^\rK\!(D)$ denote the product of weights of all dimers in~$D$, and let~$\rt(D)$ denote the {\em self-intersection number\/} of~$D$, that is, the number of times different edges of~$D$ cross one another (note that this can happen only inside of the cliques $K_{d(v)}$).

\begin{lemma}
\label{lemma:terminal}
For any planar weighted graph~$(G,x)$, the Ising partition function on~$G$ can be expressed as the signed dimer partition function on the terminal graph~$(G^\rK,x^\rK)$ as follows:
\[
\cZ_{\opname{Ising}}(G,x)=\sum_{D\in\cD(G^\rK)}(-1)^{\rt(D)}x^\rK\!(D)\,.
\]
\end{lemma}
\begin{proof}
Given a dimer configuration~$D\in\cD(G^\rK)$, let~$D_G$ denote the subgraph of~$G$ consisting of the edges of~$G$ corresponding to the long edges of~$D$.
Note that~$G\setminus D_G$ is an even subgraph of~$G$; therefore, the assignment~$D\mapsto G\setminus D_G$ defines a map~$\rho\colon\cD(G^\rK)\to\cE(G)$.
Note also that~$D\in\cD(G^\rK)$ is mapped to~$P\in\cE(G)$ if and only if its short edges match the vertices of~$G^\rK$ corresponding to edges of~$P$.
In other words, the set~$\rho^{-1}(P)$ is in bijection with~$\prod_{v\in V}\cD(K_{d(v,P)})$, where~$d(v,P)=2n(v,P)$ is the degree of~$v$ in~$P${; here, for any $n\ge 1$, the symbol~$K_{2n}$ denotes the complete graph of size $2n$.} This also shows the identity~$x^\rK\!(D)=x(\rho(D))$ for all~$D\in\cD(G^\rK)$. Hence, for any $P\in\cE(G)$, we have
\[
\sum_{D\in\rho^{-1}(P)}(-1)^{\rt(D)}x^\rK\!(D)= \Big(\prod_{v\in V}\sum_{D_v\in\cD_{2n(v,P)}}(-1)^{\rt(D_v)}\Big)\cdot x(P)\,.
\]

Further, it is easy to see that, for any~$n\ge 1$, one has~$\sum_{D\in\mathcal{D}(K_{2n})}(-1)^{\rt(D)}=1$. Indeed, let us fix two adjacent vertices of $K_{2n}$ and consider the involution~$\varsigma\colon\mathcal{D}(K_{2n})\to\mathcal{D}(K_{2n})$ given by exchanging them. The set of fixed points~$\mathrm{Fix}(\varsigma)$
consists of the dimer configurations matching these two vertices. Since~$\varsigma$ is a bijection and~$\rt(\varsigma(D))=\rt(D)+1$ for all~$D\notin\mathrm{Fix}(\varsigma)$, we have
\[
\sum_{D\in\mathcal{D}_{2n}}(-1)^{\rt(D)}=\sum_{D\in\mathrm{Fix}(\varsigma)}(-1)^{\rt(D)}=\sum_{D'\in\mathcal{D}_{2n-2}}(-1)^{\rt(D')}=\dots = 1
\]
by induction over~$n\ge 1$. Therefore,~$\sum_{D\in\rho^{-1}(P)}(-1)^{\rt(D)}x^\rK\!(D)=x(P)$ for any~$P\in\cE(G)$, and we complete the proof by the summation over all configurations~$P\in\cE(G)$.
\end{proof}

In order to handle the signs~$\veps_{e,e'}$ given by~\eqref{eqn:veps-def}, we shall also need the following well-known fact, traditionally attributed to Whitney~\cite{Whitney}, whose easy proof we include for completeness.

\begin{lemma}
\label{lemma:Whitney}
Let~$C$ be an oriented piecewise smooth planar closed curve, and let~$\wind(C)$ denote the total rotation angle of its velocity vector.
If~$C$ is in general position, i.e. if all of its self-intersections are transverse double points, then
\[
-\exp[\textstyle\frac{i}{2}\wind(C)]=(-1)^{\rt(C)},
\]
where~$\rt(C)$ denotes the number of these self-intersections.
\end{lemma}
\begin{proof}
Consider the union~$C'$ of oriented simple closed curves obtained by smoothing out all of the self-intersection points of~$C$ as follows:
\includegraphics[height=0.4cm]{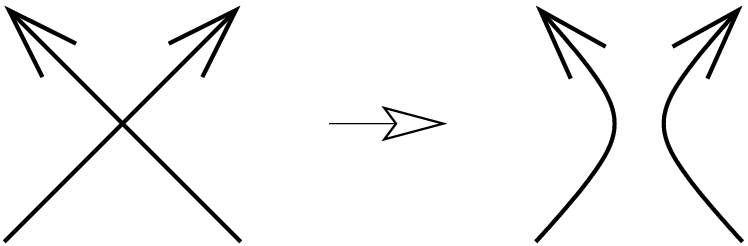}. {Letting $\wind(C')$ be the sum of the total rotation angles of these curves and since the total rotation angle of a simple closed curve is either~$2\pi$ or~$-2\pi$, we have}
\[
\exp[\textstyle\frac{i}{2}\wind(C)]=\exp[\textstyle\frac{i}{2}\wind(C')]=(-1)^{m}\,,
\]
where~$m$ denotes the number of components of~$C'$. The lemma now follows from the fact that~$m$ has the same parity as~$\rt(C)+1$ {since each of the smoothing operations used above to construct~$C'$ changes the number of components in~$C$ by~$\pm 1$.}
\end{proof}

\subsection{Proofs of Theorem~\ref{thm:KW1} and Theorem~\ref{thm:multipoint}}
\label{sub:proof}
\begin{proof}[Proof of Theorem~\ref{thm:KW1}]
Expanding the Pfaffian of~$\wh{\rK}$ leads to
\begin{equation}
\label{eqn:PfK=}
\Pf{\wh{\rK}}~=~\sum_{D\in\cD(G^\rK)}\veps(D)x^\rK\!(D)\,,\quad\text{with}\quad
\veps(D)={\sign(\sigma)\cdot}\veps_{\sigma(1)\sigma(2)}\dots\veps_{\sigma({2N-1})\sigma({2N})}\,,
\end{equation}
where~$N=|E(G)|$ and~$\sigma\in \cS(V(G^\rK))$ is any permutation representing the matching~$D$ {(i.e. such that given an ordering of the vertices, for all $i\in\{1,\ldots,N\}$, the dimers in $D$ are the edges of the form $\{\sigma(2i-1),\sigma(2i)\}$)}; note that~$\veps(D)$ does not depend on the choice of~$\sigma$ provided that some numbering of the set $V(G^\rK)\cong \{1,\dots,2N\}$ is fixed once and for all. Let $D_0\in \cD(G^\rK)$ be the standard reference matching consisting of long edges only. Comparing \eqref{eqn:PfK=} with Lemma~\ref{lemma:terminal}, we see that the following claim directly implies Theorem~\ref{thm:KW1}. (In the terminology of Tesler~\cite{Tesler}, this amounts to checking that the signs~$\veps_{e,e'}$ define a {\em crossing orientation\/} on the terminal graph~$G^\rK$.)

\smallskip

\noindent{\em Claim~A. For any~$D\in\cD(G^\rK)$, one has~$\veps(D)=(-1)^{\rt(D)}\veps(D_0)$.}

\smallskip

\noindent Indeed, one easily deduces from Lemma~\ref{lemma:terminal} and Claim~A that
\begin{equation}
\label{eqn:x-Z=epsPf}
\cZ_{\opname{Ising}}(G,x)=\!\!\!\sum_{D\in\cD(G^\rK)}(-1)^{\rt(D)}x^\rK\!(D) =\veps(D_0)\!\!\!\sum_{D\in\cD(G^\rK)}\veps(D)x^\rK\!(D)=\veps(D_0)\Pf{\wh{\rK}}\,.
\end{equation}

\smallskip

\noindent{\em Proof of Claim A.} Given two dimer configurations~$D,D_0\in\cD(G^\rK)$, their symmetric difference~$D\btu D_0$ is a union of~$\ell\ge 0$ vertex disjoint {(on~$G^\rK$)} cycles~$C_j$ of even length. Moreover, due to the particular choice of $D_0$, each $C_j$ is composed of \emph{alternating} short and long edges of $G^\rK$. Let us choose representatives~$\sigma,\sigma_0\in \cS(V(G^\rK))$ of~$D,D_0$ such that~$\sigma\circ\sigma^{-1}_0$ is the rotation by one edge of each of these cycles, with respect to some arbitrary but fixed orientation. Using this particular choice of representatives, we find
\[
\veps(D)\veps(D_0)= {\sign(\sigma)\sign(\sigma_0)\,\cdot} \prod_{j=1}^\ell \omega(C_j)= \prod_{j=1}^\ell (-\omega(C_j))\,,
\]
where~$\omega(C_j)$ denotes the product of the coefficients~$\exp[\frac{i}{2}\rw(\overline{e},e')]$ along the short edges of~$C_j$ {(the factors of $i$ and $-i$ from~\eqref{eqn:veps-def} contribute in total $1$ because of the alternation of long and short edges)}. Relating this product with the total rotation angle~$\wind(C_j)$ of the velocity vector of (a~smoothed version of)~$C_j$ and applying Lemma~\ref{lemma:Whitney} leads to
\[
\prod_{j=1}^\ell(-\omega(C_j))= \prod_{j=1}^\ell(-\exp[{\textstyle\frac{i}{2}}\wind(C_j)]) = \prod_{j=1}^\ell (-1)^{\rt(C_j)}=(-1)^{\rt(D)}\,,
\]
since all the intersections and self-intersections of $C_j$ are produced by \emph{short} edges, which all belong to $D$, and each pair of different cycles intersects an even number of times.
\end{proof}

\begin{proof}[Proof of Theorem~\ref{thm:multipoint}]
Let us write~$\rE:=\{e_1,\dots,e_{2n}\}\subset V(G^\rK)$ and denote by $G^\rK_\rE$ the subgraph of the terminal graph~$G^\rK$ obtained by removing all vertices~$e\in\rE\subset V(G^\rK)$, together with adjacent edges. Fix some numberings of the sets $V(G^\rK)$,~$V(G^\rK_\rE)$, and
denote by~$\nu_\rE$ the permutation of the ordered set~$V(G^\rK)\cong \{1,\dots,2N\}\cong V(G^\rK_\rE)\sqcup\rE$ induced by the trivial identification of~$V(G^\rK)$ and~$V(G^\rK_\rE)\cup \rE$. Using the Pfaffian identity
\[
\Pf{\wh{\rK}^{-1}_{e,e'}}_{e,e'\in\rE}~=~(-1)^{n}{\sign(\nu_\rE)\,\cdot}\Pf{\wh{\rK}_{e,e'}}_{e,e'\notin\rE}\cdot (\Pf{\wh{\rK}})^{-1}
\]
and~\eqref{eqn:x-Z=epsPf}, we only need to check the equality
\[
(-1)^{n}{\sign(\nu_\rE)}\veps(D_0)\!\!\sum_{P\in\mathcal{C}(\rE)}\tau(P)x(P) ~=~ \Pf{\wh{\rK}_{e,e'}}_{e,e'\notin\rE}\,,
\]
where the sign~$\veps(D_0)$ of the standard reference matching on~$G^\rK$ is given by~\eqref{eqn:veps-def} {and~\eqref{eqn:PfK=}}. We shall do so by generalizing the proof of Theorem~\ref{thm:KW1} given above. Observe that Lemma~\ref{lemma:terminal} (which deals with the case $n=0${, i.e. $\rE=\emptyset$}) extends in a straightforward way, yielding the equation
\[
\sum_{P\in\mathcal{C}(\rE)}\tau(P)x(P)~=\sum_{D\in\cD(G^\rK_\rE)}\tau(\rho_\rE(D))(-1)^{\rt(D)}x^\rK(D)\,,
\]
where the configuration~$P=\rho_\rE(D)\in\mathcal{C}(\rE)$ is obtained from~$E(G^\diamondsuit_\rE)$ by removing all edges corresponding to long dimers of~$D$ as well as the half-edges~$\{(o(e),z_e)\}_{e\in\rE}$\,. At the same time,
\[
\Pf{\wh{\rK}_{e,e'}}_{e,e'\notin\rE}~=\sum_{D\in\cD(G^\rK_\rE)}\veps_\rE(D)x^\rK(D)\,,
\]
where the sign~$\veps_\rE(D)$ of a dimer configuration on~$G^\rK_\rE$ is defined similarly to~\eqref{eqn:veps-def} according to the fixed ordering of~$V(G^\rK_\rE)$. Thus we are left with the proof of the following claim, which generalizes Claim~A from the proof of Theorem~\ref{thm:KW1}.

\smallskip

\noindent{\em Claim~B. For any~$D\in\cD(G^\rK_\rE)$, one has $\veps_\rE(D)=\tau(\rho_\rE(D))(-1)^{\rt(D)}\cdot(-1)^{n}{\sign(\nu_\rE)}\veps(D_0)$.}

\smallskip

\noindent {\em Proof of Claim~B.} Given a dimer configuration $D\in\cD(G^\rK_E)$, the symmetric difference~$D\btu D_0$ consists of a union of~$\ell$
vertex disjoint cycles~$C_j$ of even length, together with~$n$ vertex disjoint paths~$\gamma_k$ of odd length matching the vertices~$e\in\rE\subset V(G^\rK)$. (Note that the paths~$\gamma_k$ start and end with \emph{long} edges~$(e,\overline{e})$ of $G^\rK$, and if both~$e,\overline{e}\in\rE$, then one of these paths is just the single long edge~$(e,\overline{e})$.) Moreover, one can choose representatives~$\sigma\in \cS(V(G^\rK_\rE))$ of~$D$,~$\sigma_0\in \cS(V(G^\rK))$ of~$D_0$, and~$s\in \cS(\rE)$ of the corresponding matching of~$\rE$ such that the permutation
\[
V(G^\rK)\stackrel{\sigma_0^{-1}}{\longrightarrow} V(G^\rK)\cong\{1,\dots,2N\}\stackrel{\nu_\rE}{\longrightarrow}\{1,\dots,2N\}\cong V(G^\rK_\rE)\sqcup\rE\stackrel{\sigma\sqcup s}{\longrightarrow}V(G^\rK)
\]
is the rotation by one edge of each of the cycles~$C_j$ and of each of the paths~$\gamma_k$ closed up by an artificial link oriented from~$e_{s(2k-1)}$ to~$e_{s(2k)\vphantom{1}}$, thus~$\gamma_k$ is always oriented in a \emph{backward} direction. Since all of these cycles are of even length, the diagram above implies the equality
{\[
\sign(\sigma_0)\sign(\nu_\rE)\sign(\sigma)\sign(s)=(-1)^{\ell+n}.
\]}
Computing the signs~$\veps_\rE(D)$ and~$\veps(D_0)$ using this particular choice of representatives and the fact that all the~$C_j$ and~$\gamma_k$ are formed by {alternating} long and short edges of~$G^\rK$ (this is a consequence of the particular choice of the reference matching~$D_0$), we get
\[
\veps_\rE(D)\veps(D_0)~=~{\sign(\sigma)\sign(\sigma_0)\,\cdot}\prod_{j=1}^\ell\omega(C_j)\cdot \prod_{k=1}^n ({ i\eta_{e_{s(2k-1)}}^{\vphantom{|}}\overline{\eta}_{e_{s(2k\vphantom{1})}}}\omega(\overleftarrow{\gamma_k}))\,,
\]
where $\omega(C_j)$ and~$\omega(\overleftarrow{\gamma_k})$ denote the products of the coefficients~$\exp[\frac{i}{2}\rw(\overline{e},e')]$ along the short edges of~$C_j$ and~$\gamma_k$, respectively, with the paths~$\gamma_k$ being {traversed} from~$e_{s(2k)\vphantom{1}}$ to~$e_{s(2k-1)}$.

Let us shorten the extremities of the paths~$\gamma_k$ so that they link the points~$z_e\in V(G^{\diamondsuit})$ instead of~$e\in \rE\subset V(G^\rK)$.
Similarly to the proof of Theorem~\ref{thm:KW1}, we have
\begin{align*}
\omega(C_j)&=\exp[{\textstyle\frac{i}{2}}\wind(C_j)]=(-1)^{\rt(C_j)+1}\,,\\
{\omega(\overleftarrow{\gamma_k})} & =
\exp[{\textstyle\frac{i}{2}}\wind(\overleftarrow{\gamma_k})]=\exp[-{\textstyle\frac{i}{2}}\wind(\gamma_k)]\,.
\end{align*}
Denote
\begin{equation}\label{eqn:taugammak}
\tau(\gamma_k):={i\eta_{e_{s(2k-1)}}^{\vphantom{|}}\overline{\eta}_{e_{s(2k\vphantom{1})}}}{\omega(\overleftarrow{\gamma_k})}\,.
\end{equation}
Combining all the computations given above, we reduce Claim~B to the following statement.

\smallskip

\noindent {\em Claim~C. For any~$D\in\cD(G^\rK_\rE)$, if~$D\btu D_0$ consists of cycles~$C_j$ and paths~$\gamma_k$, then
\begin{equation}
\label{eqn:claim-tau}
{\sign(s)\,\cdot}\prod_{k=1}^n\tau(\gamma_k)\cdot (-1)^{\rt(D)-\sum_{j=1}^\ell\rt(C_j)}~=~\tau(\rho_\rE(D))\,.
\end{equation}}

Recall that long edges of~$D\btu D_0$ correspond to edges of~$P=\rho_\rE(D)\in\cC(\rE)$ while short edges of~$D\btu D_0$ define a decomposition of~$P$ into cycles~$C_j$ and paths~$\gamma_k$. In particular, if all~$C_j$ and~$\gamma_k$ do not intersect or self-intersect, then~$\rt(D)=\rt(C_j)=0$ and~\eqref{eqn:claim-tau} coincides with \emph{definition}~\eqref{eqn:tau-def} of~$\tau(P)$. Essentially, Claim~C states that the left-hand side of~\eqref{eqn:claim-tau} does not depend on the choice of $D\in\rho_\rE^{-1}(P)$, which also implies that the sign~$\tau(P)$ is well-defined (i.e. independent of the smoothing of~$P$).

\noindent {\em Proof of Claim~C.} {Let $P\in \cC(E)$ and let us fix some non-intersecting smoothing of~$P$ into~cycles~$C_j^0$ and~paths~$\gamma_k^0$ and compute the sign $\tau(P)$ by \eqref{eqn:tau-def} using these paths. Note that the result does not depend on their numbering and orientations. Given~$D\in\rho_\rE^{-1}(P)$, we number $\gamma_k^0$ and choose} their orientations (from~$e_{s^0(2k-1)}$ to~$e_{s^0(2k\vphantom{1})}$) so that
{\[
\{s^0(1),s^0(3),\dots,s^0(2n\!-\!1)\}\cup\{s(1),s(3),\dots,s(2n\!-\!1)\}=\rE\,.
\]}
We now push each path~$\gamma_k$ slightly to its right and denote the result by~$\gamma_k^+$. If we consider the union of all~$\gamma_k^+$ and all~$\gamma_k^0$ and match their endpoints by $2n$ counterclockwise $180^\circ$\!--turns, the result is a collection~$\Delta$ of~$m\le n$ oriented closed curves~$\Delta_j$. It is easy to see that
\[
{\sign(s^0)\sign(s)}=(-1)^{m},
\]
which leads{, by~\eqref{eqn:tau-def} and~\eqref{eqn:taugammak},} to
\begin{align*}
\tau(P)\cdot{\sign(s)\,\cdot}\prod_{k=1}^n\tau(\gamma_k) ~ & = ~{(-1)^m\cdot \prod_{k=1}^n (i\exp[-{\textstyle\frac{i}{2}}\wind(\gamma_k^0)])\cdot \prod_{k=1}^n(i\exp[-{\textstyle\frac{i}{2}}\wind(\gamma_k)])} \\
& = ~ (-1)^m \exp[-{\textstyle\frac{i}{2}}\wind(\Delta)]\,.
\end{align*}
Applying Lemma~\ref{lemma:Whitney} to each~$\Delta_j$, we are left with the proof of the following fact:
\[
(-1)^{\rt(D)-\sum_{j=1}^\ell\rt(C_j)}= (-1)^{\sum_{j=1}^m\rt(\Delta_j)}.
\]
Note that this is equivalent to the equality
\begin{equation}
\label{eqn:x-claim-intersections}
\rt(D)-\rt(C)=\rt(\Delta)\mod~2\,,
\end{equation}
{where $C=\bigsqcup_{j=1}^\ell C_j$}, since the total number of intersections of the closed curves~$C_j$ with each other is even due to topological reasons and the same is true for $\Delta_j$.

We need some additional notation. Similarly to~$\gamma_j$, let us push each~$C_j$ slightly to its right and denote the result by $C_j^+$. Let
\[
\textstyle \gamma:=\bigsqcup_{k=1}^n\gamma_k\,,\quad \gamma^+:=\bigsqcup_{k=1}^n\gamma_k^+\quad {\text{and}\quad C^+:=\bigsqcup_{j=1}^\ell C^+_j\,}.
\]
It is easy to see that
\[
(\gamma^+\!\sqcup C^+)\cdot(\gamma\sqcup C)=0\mod~2\,,
\]
where we denote by~$\alpha\cdot\beta$ the number of intersections of (the collections of) curves~$\alpha$ and~$\beta$. Indeed, as~$\gamma\sqcup C=D\btu D_0$, all these intersections come from the intersections of short dimers in~$D$ and each pair of such dimers contributes \emph{two} intersections to~$(\gamma^+\!\sqcup C^+)\cdot(\gamma\sqcup C)$. {Since the collection of loops and paths $\gamma\sqcup C$ can be deformed to the configuration~$P=\rho_\rE(D)\in\cC(G^\diamondsuit)$ and then further to its smoothing~$\gamma^0\!\sqcup C^0$}, we conclude that
\[
(\gamma^+\!\sqcup C^+)\cdot(\gamma^0\!\sqcup C^0)=0\mod~2\,.
\]
As $\gamma^+\!\sqcup\gamma^0$ is essentially a collection of closed curves, we also have
\[
(\gamma^+\!\sqcup \gamma^0)\cdot(C^+\!\sqcup C^0)=0\mod~2\,.
\]
By construction, $\gamma^0$ and $C^0$ do not intersect and the number of intersections of~$C^+$ and~$C^0$ is always even. Therefore, we obtain
\[
\gamma^+\!\cdot C^+ = \gamma^+\!\cdot \gamma^0\mod~2\,.
\]
The last two simple observations are
\[
\rt(D)-\rt(C)=\rt(\gamma)+\gamma\cdot C \quad \text{and}\quad \rt(\Delta)=\rt(\gamma^+)+\gamma^+\!\cdot \gamma^0\,,
\]
where we used the fact that~$\gamma^0$ is non-intersecting. Since $\rt(\gamma)=\rt(\gamma^+)$ and~$\gamma\cdot C=\gamma^+\!\cdot C^+$, the identity~\eqref{eqn:x-claim-intersections} follows, as well as Claim~C and Theorem~\ref{thm:multipoint}.
\end{proof}

\subsection{Proof of Proposition~\ref{prop:spin1}}
\label{sub:spin} In this section we work with the (domain walls expansion of the) Ising model on the dual graph~$G^*$ with~`$+$' boundary conditions, which means that we set the spin of the outer face~$u_{\opname{out}}$ of~$G$ to be~$+1$.

Recall that the matrices~$\KW_{[u_1,\dots,u_m]}$,~$\rK_{[u_1,\dots,u_m]}$ and~$\wh\rK_{[u_1,\dots,u_m]}$ are defined via some (arbitrary but fixed) collection~$\varkappa=\varkappa_{[u_1,..,u_m]}$ of edge-disjoint paths on~$G^*$ linking $u_1,\dots,u_m$ and, possibly,~$u_{\opname{out}}$ so that each of~$u_1,\dots,u_m$ has an odd degree in the union of these paths. Note that if a spin configuration~$\sigma\in\{\pm 1\}^{V(G^*)}$ with $\sigma_{u_{\opname{out}}}=+1$ corresponds to a domain walls configuration~$P\in\cE(G)$, then
\begin{equation}
\label{eqn:prod_s=kappa_P}
\sigma_{u_1}\dots\sigma_{u_m}=(-1)^{\varkappa\cdot P}\,,
\end{equation}
where~$\varkappa\cdot P$ denotes the number of intersections of~$P$ with~$\varkappa$. Below we assume that the terminal graph~$G^\rK$ is drawn in such a way that its long edges intersect~$\varkappa$ if and only if the corresponding edge of~$G$ does the same, while short edges never intersect these ``cuts''.
The following statement generalizes~Lemma~\ref{lemma:terminal}.
\begin{lemma}
\label{lemma:x-terminal}
For any planar graph~$(G,x)$, and any set of edges~$\varkappa\subset E(G^*)$, one has
\[
\sum_{P\in\cE(G)} (-1)^{\varkappa\cdot P}x(P)= (-1)^{|\varkappa|}\!\!\!\sum_{D\in\cD(G^\rK)} (-1)^{\rt(D)}(-1)^{\varkappa\cdot D}x^\rK(D)\,,
\]
where~$|\varkappa|$ denotes the total number of edges in the collection of (dual) paths~$\varkappa=\varkappa_{[u_1,..,u_m]}$\,.
\end{lemma}
\begin{proof}
Following the proof of Lemma~\ref{lemma:terminal}, the only additional fact to check is the identity
\[
(-1)^{\varkappa\cdot\rho(D)}=(-1)^{|\varkappa|}\cdot (-1)^{\varkappa\cdot D}\,.
\]
By construction of the mapping~$\rho$, for each~$e\in E(G)$ and~$D\in\cD(G^\rK)$, we have~$e\in\rho(D)$ if and only if the corresponding long edge of~$G^\rK$ does \emph{not} belong to~$D$. Therefore,
\[
\varkappa\cdot\rho(D) + \varkappa\cdot D  = \varkappa\cdot E(G) = |\varkappa|\,,
\]
{and the claim is proved}.
\end{proof}
We are now able to prove formula~\eqref{eqn:multi-spin-plane} for multi-point spin expectations.
\begin{proof}[Proof of Proposition~\ref{prop:spin1}]
It follows from~\eqref{eqn:prod_s=kappa_P} and Lemma~\ref{lemma:x-terminal} that
\[
\E^+_{G^*}[\sigma_{u_1}\dots\sigma_{u_m}]=\frac{\sum_{P\in\cE(G)} (-1)^{\varkappa\cdot P}x(P)}{\sum_{P\in\cE(G)}x(P)}= (-1)^{|\varkappa|}\cdot \frac{\sum_{D\in\cD(G^\rK)} (-1)^{\rt(D)}(-1)^{\varkappa\cdot D}x^\rK(D)}{\sum_{D\in\cD(G^\rK)} (-1)^{\rt(D)}x^\rK(D)}\,.
\]
It was shown in the proof of Theorem~\ref{thm:KW1} that
\[
\textstyle \Pf{\wh\rK}= \veps(D_0)\sum_{D\in\cD(G^\rK)} (-1)^{\rt(D)}x^\rK(D)\,.
\]
Repeating the same proof, we find
\[
\textstyle \Pf{\wh\rK_{[u_1,..,u_m]}}=\veps(D_0)\sum_{D\in\cD(G^\rK)} (-1)^{\rt(D)}(-1)^{\varkappa\cdot D}x^\rK(D)\,,
\]
thus arriving at~\eqref{eqn:multi-spin-plane}.
\end{proof}

We conclude this section with one last statement which naturally generalizes Theorem~\ref{thm:multipoint}.

\begin{proposition}
\label{prop:multipoint-with-spins}
Let~$u_1,\dots,u_m$ be some faces of~$G$ and the matrix~$\wh\rK_{[u_1,..,u_m]}$ be defined using the collection $\varkappa=\varkappa_{[u_1,..,u_m]}$ of edge-disjoint paths on~$G^*$ linking~$u_1,\dots,u_m$ and~$u_{\opname{out}}$.
For any set of~$2n$ oriented edges~$\rE=\{e_1,..,e_{2n}\}\subset V(G^\rK)$, one has
\[
\Pf{(\wh\rK^{-1}_{[u_1,..,u_m]})_{e_j,e_k}}_{j,k=1}^{2m}~=~ \frac{\sum_{P\in\cC(e_1,\dots,e_{2n})} (-1)^{\varkappa\cdot P}\tau(P)x(P)} {\E^+_{G^*}[\sigma_{u_1}\dots\sigma_{u_m}]\cdot \cZ_{\opname{Ising}}(G,x)}\,.
\]
Above, if some edge~$e_k\in\rE$ intersects~$\varkappa$ and $\overline{e}_k\not\in\rE$, we \emph{do} count the corresponding half-edge~$(z_{e_k},t(e_k))\in P$ in the intersection number~$\varkappa\cdot P$.
\end{proposition}

\begin{proof} Mimicking the proof of Lemma~\ref{lemma:x-terminal} for the mapping~$\rho_\rE:\cD(G^\rK_\rE)\to\cC(\rE)$, one gets
\[
\sum_{P\in\cC(e_1,\dots,e_{2n})} (-1)^{\varkappa\cdot P}\tau(P)x(P) ~ =~(-1)^{|\varkappa|}\!\!\!\sum_{D\in\cD(G^\rK_\rE)}(-1)^{\rt(D)}\tau(\rho_\rE(D))(-1)^{\varkappa\cdot D}x^\rK(D)\,.
\]
On the other hand, from the proof of formula~\eqref{eqn:multi-spin-plane} given above we know that
\begin{align*}
\E^+_{G^*}[\sigma_{u_1}\dots\sigma_{u_m}]\cdot \cZ_{\opname{Ising}}(G,x)~&=~(-1)^{|\varkappa|}\!\!\!{\sum_{D\in\cD(G^\rK)} (-1)^{\rt(D)}(-1)^{\varkappa\cdot D}x^\rK(D)} \\
&=~(-1)^{|\varkappa|}\veps(D_0)\Pf{\wh\rK_{[u_1,..,u_m]}}\,,
\end{align*}
which is a proper replacement for~\eqref{eqn:x-Z=epsPf} in the twisted setup. We now simply follow the proof of Theorem~\ref{thm:multipoint} with the weights~$x^\rK(D)$ replaced by~$(-1)^{\varkappa\cdot D}x^\rK(D)$.
It is worth noting that the check of signs performed in Claim~B (and further Claim~C) does not depend on~$x^\rK$.
\end{proof}


\setcounter{equation}{0}
\section{Various formalisms}
\label{sec:links}

This section is devoted to a self-contained exposition of the relations between several classical approaches designed to study the planar Ising model: dimer representations~\cite{Kasteleyn-63,Fisher-66}, Grassmann variables~\cite{Hurst-Green,Lieb-fermions}, disorder insertions~\cite{Kadanoff-Ceva}, as well as a more recent language of `combinatorial' s-holomorphic observables~\cite{smirnov-icm-2006,Smirnov-10,smirnov-icm-2010}. Of course, most if not all claims below are part of the folklore surrounding the Ising model. However, we hope that such an exposition, intended in particular for combinatorialists and probabilists, will be useful in view of the renewed activity in the field. We also refer the reader interested in a more advanced discussion of spin-disorder techniques to the recent papers~\cite{dub-abelian,Dub} by Dub\'edat.

\subsection{Dimer representations} \label{sub:planar-dimers}
In Section~\ref{sec:planar}, one of the main tools used was a dimer representation of the Ising model on the non-bipartite and, in general, \emph{non-planar} terminal graph~$G^\rK$. In a famous variation on this construction, Fisher introduced a \emph{planar} but non-bipartite graph such that the Ising configurations were in $1$--to--$1$ correspondence with dimers. His mapping may be described, paraphrasing his own words~\cite[p. 1777]{Fisher-66}, as follows: one starts with the graph $G$ and makes all its vertices trivalent (without changing the probability measure) and then uses the mapping of the configurations to dimers on the terminal graph of this new graph, which in virtue of trivalence, is planar and the dimer measure is unsigned. A more symmetric version of this construction was later proposed in~\cite{BdT-Z} forcing the mapping from configurations to dimers to be $2^{|V(G)|}$--to--$1$. Yet another (slighly simpler) variation on Fisher's idea was proposed in~\cite{Dub}, where long dimers are in direct correspondence with Ising configurations, and this is the one we shall use here, calling the corresponding graph the Fisher graph.

In this section we briefly discuss representations of the Ising model on~$G$ via dimers on the two following graphs: the corner graph~$G^\rC$ and the Fisher graph~$G^\rF$ discussed above, the latter being \emph{planar} but also never bipartite, see Fig.~\ref{Fig:Corner-Fisher}.
It is worth noting that there exist several combinatorial ways to represent a \emph{pair} of independent Ising models via \emph{single} dimers on some other graph constructed from~$G$, which is both planar and bipartite (see~\cite{Dub,BdT-xor,BdT4}), but we do not discuss these constructions here.

\begin{figure}
\captionsetup[subfigure]{position=below,justification=justified,singlelinecheck=false,labelfont=bf}
\begin{subfigure}[t]{.49\textwidth}
\labellist\small\hair 2.5pt
\pinlabel {$e$} at 220 120
\endlabellist
\centering{\psfig{file=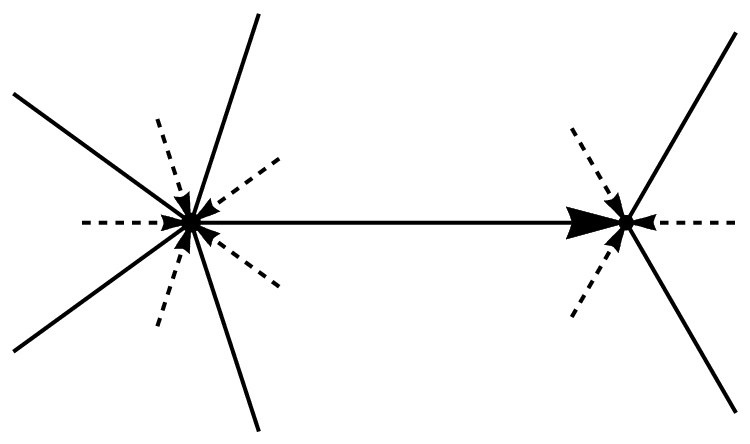,width=.9\linewidth}}
\caption{A portion of a graph~$G$ around an edge $e$, with decorations at corners (dashed)}
\end{subfigure}
\hfill
\begin{subfigure}[t]{.48\textwidth}
\labellist\small\hair 2.5pt
\pinlabel {$c^+(e)$} at 180 145
\pinlabel {$c^{-}(e)$} at 180 45
\pinlabel {$c^-(\bar{e})$} at 250 145
\pinlabel {$c^{+}(\bar{e})$} at 250 45
\endlabellist
\centering{\psfig{file=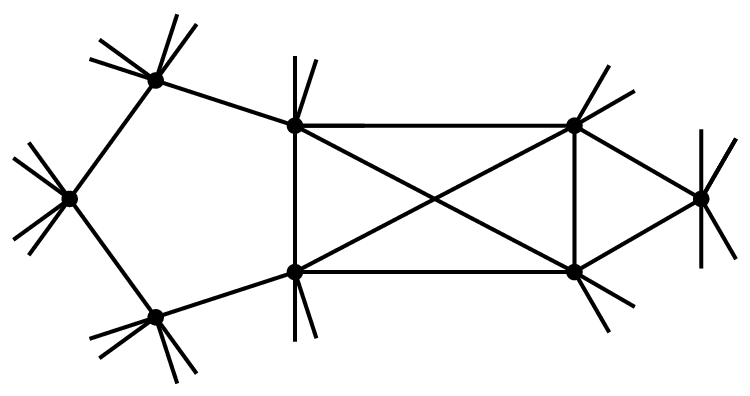,width=.9\linewidth}}
\caption{The corresponding portion of the corner graph~$G^\rC$ with the vertices $c^{\pm}(e)$}
\label{Fig:Corner}
\end{subfigure}
\hfill
\begin{subfigure}[t]{.49\textwidth}
\labellist\small\hair 2.5pt
\pinlabel {$e$} at 175 90
\pinlabel {$\bar{e}$} at 220 90
\pinlabel {$c^+(e)$} at 150 149
\pinlabel {$c^-(e)$} at 150 65
\pinlabel {$c^-(\bar{e})$} at 253 149
\pinlabel {$c^+(\bar{e})$} at 253 65
\endlabellist
\centering{\psfig{file=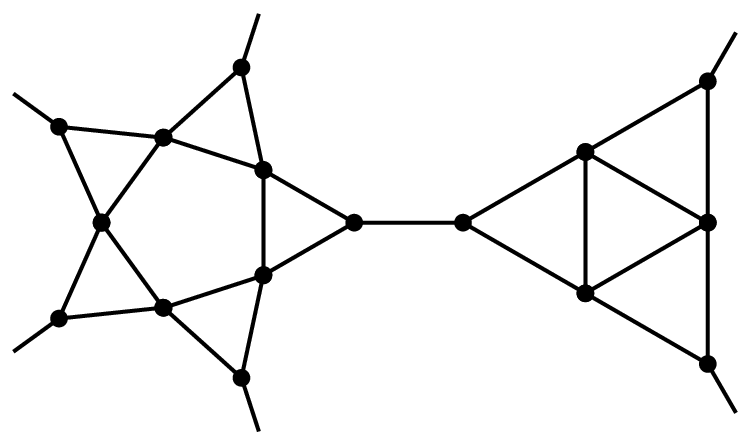,width=.9\linewidth}}
\caption{The corresponding portion of the Fisher graph~$G^\rF$ and mapping of its vertex set $V(G^\rF)$ to $V(G^\rC)\cup V(G^\rK)$}
\label{Fig:Fisher}
\end{subfigure}
\hfill
\begin{subfigure}[t]{.48\textwidth}
\labellist\small\hair 2.5pt
\pinlabel {$v_1$} at 40 50
\pinlabel {$e_1$} at 85 50
\pinlabel {$c_1$} at 180 50
\endlabellist
\centering{\psfig{file=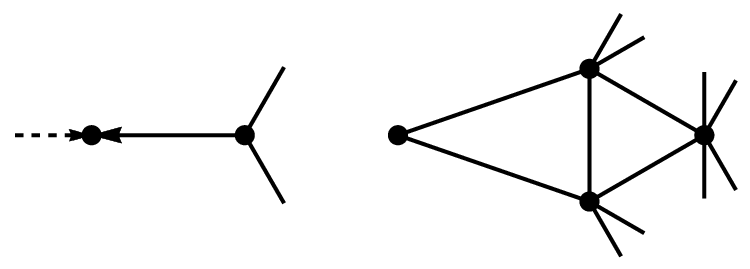,width=.9\linewidth}}
\caption{Special case of a univalent vertex $v_1=t(e_1)$ (with decoration at the corner dashed) in $G$, and the corresponding corner $c_1$ in~$G^\rC$}
\label{Fig:Corner-1}
\end{subfigure}
\caption{The different graphs and the relations between them}
\label{Fig:Corner-Fisher}
\end{figure}
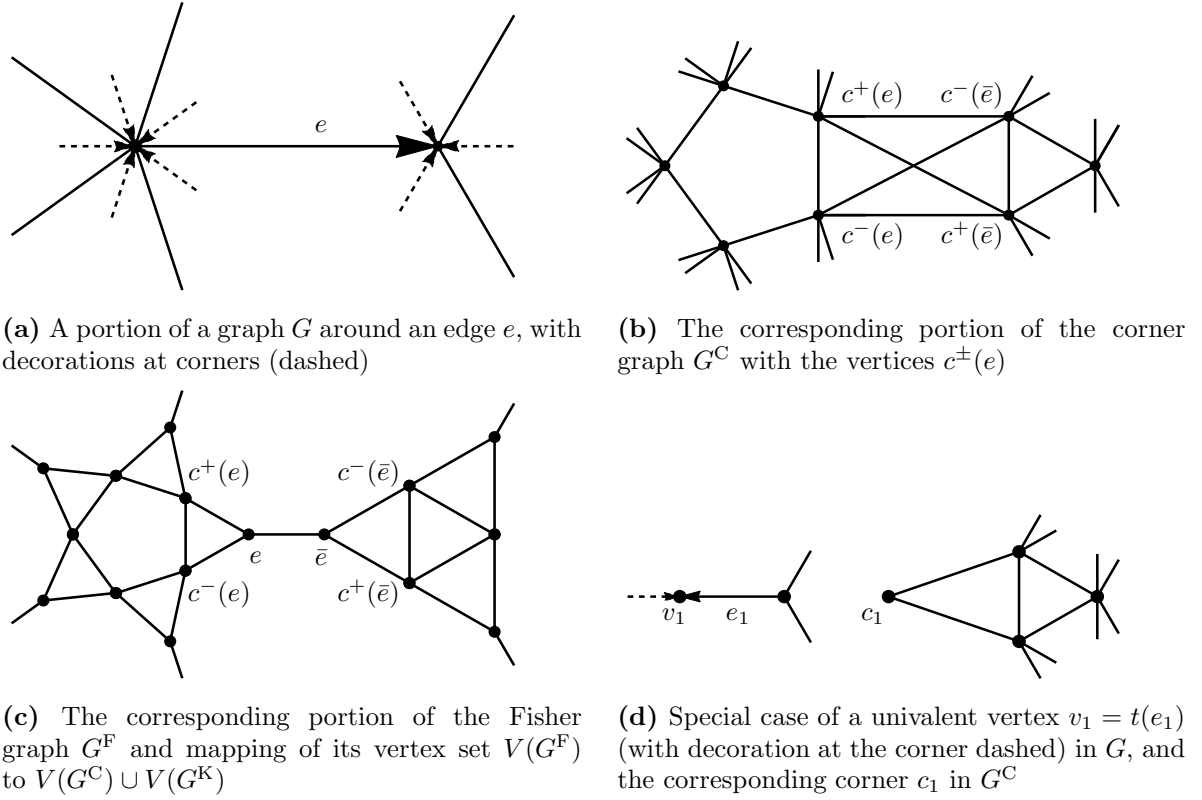

For a while, assume that~$G$ has no vertices of degree~$1$ (clearly, we do not lose generality by making this assumption, though it will be convenient to allow such vertices later on). Let~$G^\rC$ denote the \emph{corner graph} obtained from~$G$ as follows: each vertex~$v$ of~$G$ (of degree~$d(v)$) is replaced by a simple cycle of length~$d(v)$, with each pair of cycles corresponding to neighboring vertices~$v,v'\in G$ being cross-linked by four edges of~$G^\rC$, as shown on Fig.~\ref{Fig:Corner}.
The vertices of~$G^\rC$ are called \emph{corners} of~$G$; note that~$|V(G^\rC)|=|V(G^\rK)|$. With each corner~$c\in V(G^\rC)$ we associate a straight segment on the plane, the so-called~\emph{decoration} of~$G$ at~$c$, oriented \emph{towards} the corresponding vertex of~$G$ which we denote by~$v(c)$. Given an oriented edge~$e\in V(G^\rK)$, we denote by~$c^\pm(e)$ the two neighboring corners of~$e$ satisfying~$v(c^\pm(e))=o(e)$, see~Fig.~\ref{Fig:Corner}.

Let the square matrix~$\rB=(\rB_{c,e})_{c\in V(G^\rC),~e\in V(G^\rK)}$ be defined by
\begin{equation}
\label{eqn:rB-def}
\rB_{c,e}=\begin{cases}
\exp\left[\frac{i}{2}\rw(c,e)\right]\cdot x_e^{-1/2}& \text{if~$c$~is~one~of~the~two~corners~$c^\pm(e)$;} \\
0 & \text{otherwise,}
\end{cases}
\end{equation}
where~$\rw(c,e)$ denotes the rotation angle from the decoration of~$G$ at~$c$ to the oriented edge~$e$. Note that~$\rB$ has a block-diagonal structure with~$d(v)\times d(v)$ blocks~$\rB^{(v)}$ corresponding to vertices~$v\in G$. It is easy to see that
\[
|\det\rB|~=\prod_{v\in V(G)}|\det\rB^{(v)}|=\prod_{v\in V(G)}\biggl[\,2\!\!\!\prod_{v'\in G:v\sim v'} x_{vv'}^{-1/2}\biggr] = ~2^{|V(G)|}\!\!\!\prod_{e\in E(G)}x_e^{-1}.
\]

Let us now define the matrix
\[
\rC:=\rB\rK\rB^*,
\]
whose entries are labeled by the corners~$c\in V(G^\rC)$. A straightforward computation gives
\[
\rC_{c,c'}=\begin{cases}
\exp[\frac{i}{2}\rw(c,\overline{c}')]\cdot x_e^{-1}& \text{if~$v(c)=o(e)$~and~$v(c')=t(e)$~for~some~$e\in V(G^\rK)$;} \\
-\exp[\frac{i}{2}\rw(c,\overline{c}')]& \text{if~$v(c)=v(c')$~and~$c$~is~adjacent~to~$c'$~in~$G^\rC$;} \\
0 & \text{otherwise.}
\end{cases}
\]
Above,~$\rw(c,\overline{c}')$ denotes the rotation angle from the oriented decoration~$c$ to the oppositely oriented decoration~$\bar{c}'$,
measured in a natural way: along the path~$c\oplus e\oplus \bar{c}'$ in the first line, and along~$c\oplus\bar{c}'$ in the second.

Note that $\rC$ is a weighted adjacency matrix of the graph~$G^\rC$ which is ``almost planar'': the only pairs of intersecting edges go along edges of~$G$. To get a weighted adjacency matrix of a \emph{planar} graph, let us introduce a twice bigger matrix $\rF$, whose entries are labeled by the set $V(G^\rC)\cup V(G^\rK)$, as follows:
\[
\rF:=\matr{\rI}{\rB}{0}{\rJ}\matr{\rC}{0}{0}{-\rJ}\matr{\rI}{0}{\rB^*}{\rJ}=\matr{\rC-\rB\rJ\rB^*}{-\rB}{-\rB^*}{-\rJ}.
\]
Again, a straightforward computation shows that, for $c,c'\in V(G^C)$,
\[
\rF_{c,c'}=(\rC-\rB\rJ\rB^*)_{c,c'}=\begin{cases}
-\exp[\frac{i}{2}\rw(c,\overline{c}')]& \text{if~$v(c)=v(c')$~and~$c$~is~adjacent~to~$c'$~in~$G^\rC$;} \\
0 & \text{otherwise.}
\end{cases}
\]

Therefore, $\rF$ is a weighted adjacency matrix of the graph~$G^\rF$ with~$V(G^\rF)=V(G^\rC)\cup V(G^\rK)$, which is constructed from~$G$ as follows: for each~$v\in V(G)$, the~$d(v)$ corners~$c\in V(G^\rC)$ satisfying~$v(c)=v$ are linked cyclically around~$v$; for each edge of $G$, the two corresponding vertices $e,\bar{e}\in V(G^\rK)$ are linked with each other; and, finally, each of the vertices~$e\in V(G^\rK)$ is linked with two neighboring corners $c^\pm(e)\in V(G^\rC)$, see Fig.~\ref{Fig:Fisher}.

As discussed in the introductory paragraph of this section, the \emph{planar} graph $G^\rF$ is not Fisher's original graph~\cite[Fig.~6]{Fisher-66} nor its symmetric modification~\cite{BdT-Z}, but we still call it the \emph{Fisher graph}, following~\cite{Dub}. We shall denote by~$x^\rF$ the edge weights on~$G^\rF$ obtained by assigning weights~${x_e}^{-1/2}$ to short edges of~$G^\rF$ linking vertices ~$e\in V(G^\rK)$ with~$c_e^\pm\in V(G^\rC)$ and weight~$1$ to all other edges.

It is well-known that there exists a simple~$2^{|V(G)|}$--to--$1$ correspondence~$\varrho:\cD(G^\rF)\to\cE(G)$ between perfect matchings of~$G^\rF$ and even subgraphs of~$G$: given~$D\in\cD(G^\rF)$, take all the edges~$e\in E(G)$ corresponding to long dimers~$\{e,\overline{e}\}\in D$. It is easy to check that, for any~$P\in\cE(G)$,
\[
\sum_{D\in\varrho^{-1}(P)}\!\!x^\rF(D)~=~ 2^{|V(G)|}\cdot(x(E(G)\setminus P))^{-1} ~=~ |\det\rB|\cdot x(P)\,,
\]
which leads to the equality
\[
\cZ_{\opname{dimers}}(G^\rF,x^\rF)~=~|\det\rB|\cdot \cZ_{\opname{Ising}}(G,x)\,.
\]

{Similarly to the choice of~$\eta_e$ for~$e\in V(G^\rK)$, for each~$c\in V(G^\rC)$ we fix a square root of the direction of the decoration corresponding to~$c\in V(G^\rC)$ and denote by~$\eta_c$ its complex conjugate, multiplied by the same global unimodular factor $\zeta$.} We denote by~$\rU_\rC$ the diagonal matrix with entries~$\{\eta_c\}_{c\in V(G^\rC)}$, by~$\rU_{\rF}$ the diagonal matrix with entries~$\{\eta_p\}_{p\in V(G^\rF)}$, and set
\begin{equation}
\label{eqn:whB-whC-whF-def}
\wh\rB:={\rU_\rC^*\rB\,\rU_\rK^{\vphantom{*}}}\,,\qquad
\wh\rC:={i\rU_\rC^*\rC\,\rU_\rC^{\vphantom{*}}}=\wh\rB\,\wh\rK\,\wh\rB^\top, \qquad \wh\rF:={i\rU_\rF^*\rF\,\rU_\rF^{\vphantom{*}}}\,,
\end{equation}
where~$\rU_\rK:=\rU$, see~\eqref{eqn:hat-K-def}. Note that all matrices~$\wh\rB,\wh\rC$ and~$\wh\rF$ are real-valued. Moreover,~$\wh\rC$ and~$\wh\rF$ are anti-symmetric since~$\rC$ and~$\rF$ are self-adjoint, similarly to the symmetries of~$\wh\rK$ and~$\rK$. The identities
\[
\det\rF= (-1)^{|E(G)|}\det\rC = |\det\rB|^2\cdot(-1)^{|E(G)|}\det\rK
\]
imply
\[
|\Pf{\wh\rF}|=|\Pf{\wh\rC}|= |\det\rB|\cdot |\Pf{\wh\rK}|.
\]

Finally, it is easy to check that, for any choice of the square roots {in the definition of}~$\eta_e$ {and}~$\eta_c$, the signs of the matrix entries~$\wh\rF_{p,q}$ provide a Kasteleyn orientation \cite{Kast61,Kast67} of the {planar} graph $G^\rF$. Therefore,
\[
\cZ_{\opname{Ising}}(G,x)= |\det\rB|^{-1}\cdot \cZ_{\opname{dimers}}(G^\rF,x^\rF)= |\det\rB|^{-1}|\Pf{\wh\rF}|=|\Pf{\wh\rK}|\,,
\]
which gives an alternative proof of Theorem~\ref{thm:KW1} in the planar case.

\begin{remark}\label{rk:Lieb}
In~\cite[Theorem A.1]{Lieb-Loss}, Lieb and Loss gave a new proof (which is canonical in some sense) of Kasteleyn's theorem for counting dimers on any planar graph.
If we rephrase their result in the special case of our planar graph $G^\rF$, it is interesting to note that their Hermitian matrix $T$ is exactly the matrix $\rF$ we introduced above, thus showing that their method is related to the symmetries of the Kac--Ward matrix.
Indeed, their main theorem is that the square root of the modulus of the determinant of~$T$ counts dimers and that~$T$ can be gauge-transformed (that is, conjugated by a diagonal unitary matrix) to be equal to $i$ times an antisymmetric matrix (which is therefore a Kasteleyn matrix and Kasteleyn's theorem is proved).
Our identity $\wh\rF={i\rU_\rF^*\rF\,\rU_\rF^{\vphantom{*}}}$ expresses the same thing (by furthermore specifying the gauge transform) in the special case of~$G^{\rF}$.
\end{remark}

In the same spirit, one can use the real anti-symmetric matrix $\wh\rC$ to obtain another proof of Theorem~\ref{thm:KW1} based on the considerations of the signed dimer model on the corner graph~$G^\rC$. Similarly to the dimer model on~$G^\rK$ considered in Section~\ref{sec:planar}, in this case one should account the sign~$(-1)^{\rt(D)}$, where $\rt(D)$ now denotes the number of intersections in a dimer configuration~$D\in\cD(G^\rC)$, so that the combinatorial correspondence of~$\cE(G)$ with~$\cD(G^\rC)$ and the expansion of $\Pf{\wh\rC}$ work properly. Again, this essentially amounts to checking that the signs of matrix entries~$\wh\rC_{c,c'}$ provide a crossing orientation of~$G^\rC$ in the terminology of Tesler~\cite{Tesler}.

We conclude this section with a remark on modifications needed to include univalent (i.e. having degree~$1$) vertices of~$G$ into considerations, this will be useful to discuss boundary conditions for discrete fermionic observables in Section~\ref{sub:s-observables} below.

\begin{remark}
\label{rem:degree-one-vertices}
Let~$v_1\in V(G)$ be a degree~$1$ vertex and $e_1=e_1(v_1)\in\EE(G)$ be the unique oriented edge of~$G$ satisfying~$t(e_1)=v_1$. In this case the corner graph~$G^\rC$ contains only \emph{one} corner~$c_1=c_1(v_1)\in V(G^\rC)$ near~$v_1$, and we always draw the corresponding decoration as pointing in the same direction as the edge~$\overline{e}_1$, see Fig.~\ref{Fig:Corner-1}. The corresponding $1\times 1$ block of the matrix~$\rB$ is then defined as
\[
\rB_{c_1,\overline{e}_1}=\exp[{\textstyle\frac{i}{2}}\rw(c_1,\overline{e}_1)]\cdot x_{e_1}^{-1/2}=x_{e_1}^{-1/2},
\]
while the mapping $\varrho:\cD(G^\rF)\to\cE(G)$ is $1$-to-$1$ near such~$v_1$, not $2$-to-$1$: the edge~$e_1$ never participates in~$P\in\cE(G)$ while the dimer $\{c_1,\overline{e}_1\}$ presents in all $D\in\cD(G^\rF)$. With these modifications, the arguments given above remain true in presence of degree~$1$ vertices.
\end{remark}

\subsection{Grassmann variables and double-covers} \label{sub:grassmann}

In this section we discuss the well-known formalism of Grassmann variables (e.g. see \cite[Chapter 2.A]{YellowBook}), i.e. \emph{anti-}commuting variables~$\phi_1,\dots,\phi_{2N}$ associated with a real \emph{anti-}symmetric matrix~\mbox{$A=(A_{j,k})_{j,k=1}^{2N}$\,,} in the context of the planar Ising model, when the matrix $A$ is equal to~$\wh\rK$ or~$\wh\rK_{[u_1,..,u_m]}$\,. In the latter case, we introduce a \emph{double-cover}~$G_{[u_1,..,u_m]}$ of~$G$ in order to make the formal correlation functions independent of the choice of collection of cuts~$\varkappa=\varkappa_{[u_1,..,u_m]}$.

For a real anti-symmetric matrix~$A$, let
\[
\mathcal{Z}_{A}=\int \exp[-{\textstyle \frac{1}{2}}\phi^\top\!A\phi]~d\phi_{1}\dots d\phi_{2N} = \Pf{A}
\]
denote the coefficient of the highest monomial~$\phi_{2N}\dots\phi_{1}$ in the formal Taylor expansion of~$\exp[-{\textstyle \frac{1}{2}}\phi^\top\!A\phi]$ (note that, since~$\phi_k^2=0$, this expansion contains only a finite number of terms). Further, for an even subset~$k_1,\dots,k_{2n}$ of indices~$1,\dots,2N$, let
\[
\mathcal{Z}_{A}[\phi_{k_1}\dots\phi_{k_{2n}}] =\int \phi_{k_1}\dots\phi_{k_{2n}}\!\exp[-{\textstyle \frac{1}{2}}\phi^\top\!A\phi]~d\phi_{1}\dots d\phi_{2N}
\]
be the coefficient of the highest monomial~$\phi_{2N}\dots\phi_1$ in the formal expansion of the expression~$\phi_{k_1}\dots\phi_{k_{2n}}\!\exp[-{\textstyle \frac{1}{2}}\phi^\top\!A\phi]$. Note that this coefficient trivially vanishes if some of these~$\phi_k$ variables is repeated twice (or more) or if the number of these variables is odd.

Let~$A=\wh\rK$. Recall that, up to a global~$\pm 1$ sign, $\mathcal{Z}_{\wh\rK}=\Pf{\wh\rK}$ is equal to the Ising model partition function~$\cZ_{\opname{Ising}}(G,x)$ due to Theorem~\ref{thm:KW1}.
The \emph{formal correlation function} of the Grassmann variables~$\phi_{e}$ with~$e\in\rE=\{e_1,\dots,e_{2n}\}\subset V(G^\rK)$ is {defined} as
\[
\lan\phi_{e_1}\dots\phi_{e_{2n}}\!\ran_{\wh{\rK}}~:= \frac{\mathcal{Z}_{\wh{\rK}}[\phi_{e_1}\dots\phi_{e_{2n}}]}{\mathcal{Z}_{\wh{\rK}}}\,.
\]
By definition, this function is anti-symmetric with respect to the ordering of the variables~$\phi_{e_k}$. In particular, $\lan\phi_{e_j}\phi_{e_k}\!\ran_{\wh{\rK}}~=-\!\lan\phi_{e_k}\phi_{e_j}\!\ran_{\wh{\rK}}$\,. It is an easy exercise to check that
\[
\lan\phi_{e_1}\dots\phi_{e_{2n}}\!\ran_{\wh{\rK}}~= ~\pm \frac{\Pf{\wh{\rK}}_{e,e'\not\in\rE}}{\Pf{\wh\rK}}=~
\Pf{\wh{\rK}^{-1}_{e_j,e_k}}_{j,k=1}^{2n} = \Pf{\lan\phi_{e_j}\phi_{e_k}\!\ran_{\wh{\rK}}}_{j,k=1}^{2n} \,,
\]
where the sign in the middle depends on the ordering of the sets~$V(G^\rK)\setminus \rE$ and~$V(G^\rK)$.

\begin{remark}
\label{rem:grassmann-combinatorics}
Working with the Grassmann variables formalism, one can wonder about the~\emph{combinatorial} interpretation of the arising formal correlation functions~\mbox{$\lan\phi_{e_1}\dots\phi_{e_{2n}}\!\ran_{\wh{\rK}}$}\,. The answer to this question is the matter of Theorem~\ref{thm:multipoint}.
\end{remark}

Now let~$A=\wh\rK_{[u_1,..,u_m]}$\,. Recall that, for a given collection of faces~$u_1,\dots,u_m\in V(G^*)$ and a fixed collection of dual paths~$\varkappa=\varkappa_{[u_1,..,u_m]}$ linking $u_1,\dots,u_m$ and, possibly,~$u_{\opname{out}}$ on~$V(G^*)$, we have~$\wh\rK_{[u_1,..,u_m]}={i\rU^*\rK_{[u_1,..,u_m]}\rU}$, where
\[
(\rK_{[u_1,..,u_m]})_{e,e'}=\begin{cases}
(-1)^{\varkappa\cdot e} &\text{if~$e'=\overline{e}$;}\\
-\exp[\frac{i}{2}\rw(\overline{e},e')]\cdot(x_ex_{e'})^{1/2}& \text{if~$o(e)=o(e')$ but~$e'\neq e$;} \\
0 & \text{otherwise.}
\end{cases}
\]
Given a collection of oriented edges~$\rE=\{e_1,\dots,e_{2n}\}\subset V(G^\rK)$, we define the ``twisted'' correlation functions of Grassmann variables~$\phi_e$ as
\[
\lan\phi_{e_1}\dots\phi_{e_{2n}}\!\ran_{[u_1,..,u_m]}~:= \mathcal{Z}_{\wh{\rK}_{[u_1,..,u_m]}^{-1}}\!\!\cdot\,{\mathcal{Z}_{\wh{\rK}_{[u_1,..,u_m]}}[\phi_{e_1}\dots\phi_{e_{2n}}]}= \Pf{\wh{\rK}_{[u_1,..,u_m]}^{-1}}_{e,e'\in\rE}
\]
(in this case, the relevant combinatorial expansions are provided by Proposition~\ref{prop:multipoint-with-spins}). As usual, these formal correlations implicitly depend on the choice of the paths~$\varkappa_{[u_1,..,u_m]}$ but there is a standard way to make the notation more invariant.

Let~$\C_{[u_1,..,u_m]}$ denote the canonical double-branched cover of the complex plane~$\C$ with branching points~$u_1,\dots,u_m\in V(G^*)\subset\C$, which is endowed with an involution~$z\mapsto z^\sharp$ and with a projection onto~$\C$. Then, any graph~$G$ embedded in~$\C\setminus\{u_1,\dots,u_m\}$ lifts to the canonical double-cover~$G_{[u_1,..,u_m]}$ embedded in~$\C_{[u_1,..,u_m]}$. Now, any cut set~$\varkappa_{[u_1,..,u_m]}$ gives us a particular way to construct this cover as two copies of the plane cut and pasted, and idem for the embedded graphs. Equivalently, such a set of cuts gives us two sections of this cover
(i.e. the choice of one point above each of the points in the plane). Given~$\varkappa_{[u_1,..,u_m]}$ and~$e\in V(G^\rK_{[u_1,..,u_m]})$ lying on one of the corresponding sections, we set
\begin{equation}
\label{eqn:phi-spinor}
\phi_{e^\sharp}:=-\phi_{e}\,.
\end{equation}
This notation allows us to speak about formal correlation functions~$\lan\phi_{e_1}\dots\phi_{e_{2n}}\!\ran_{[u_1,..,u_m]}$ with~$e_1,\dots,e_{2n}$ on the canonical double-cover of the terminal graph~$G^\rK$ and it is easy to see that these quantities are \emph{independent of the choice of~$\varkappa=\varkappa_{[u_1,..,u_m]}$}.
Indeed, shifting~$\varkappa$ across some vertex~$v\in V(G)$ amounts to the multiplication by~$-1$ of all the rows and the columns of the matrix~$\rK_{[u_1,..,u_m]}$ that are labeled by oriented edges~$e$ with~$o(e)=v$. This leads to the multiplication of all the corresponding variables~$\phi_e$ by~$-1$ and agrees with~\eqref{eqn:phi-spinor}.

\subsection{Disorder insertions} \label{sub:disorders}

In this section we briefly discuss the formalism of disorder insertions developed in~\cite{Kadanoff-Ceva}, see also~\cite{dub-abelian}. In this approach, one easily finds \emph{combinatorial} expansions of correlation functions similar to Theorem~\ref{thm:multipoint} and Proposition~\ref{prop:multipoint-with-spins}, as shown in~Lemma~\ref{lemma:mu-sigma} and Remark~\ref{rem:x-mu-sigma} below. {On the other hand, following this approach then requires additional efforts} to reveal the underlying \emph{Pfaffian} structure of correlation functions. (Note that this is exactly opposite to the discussion of the Grassmann variables formalism given above, cf. Remark~\ref{rem:grassmann-combinatorics}.)

Recall that we prefer to work with the domain walls representations of the Ising model, thus the spins~$\sigma$ are associated with faces of $G$ while \mbox{\emph{disorders}}~$\mu_v$ will be associated to \emph{vertices} of~$G$ ({this is dual to the more common convention which assigns spins to vertices of~$G$ and disorders to its faces}). Given an even number of vertices~$v_1,\dots,v_{2n}\in V(G)$, let us fix a collection of edge-disjoint paths~$\varkappa=\varkappa^{[v_1,..,v_{2n}]}\subset E(G)$ matching them so that each vertex~$v_k$ has an odd degree in~$\varkappa$ while all other vertices have even degrees, and let
\[
\langle\mu_{v_1}\dots\mu_{v_{2n}}\rangle:=\mathbb{E}_{G^*}^+ \left[\prod\nolimits_{e\in\varkappa} x_e^{\reps_e}\right]\,,
\]
where~$\reps_e=\pm 1$ denotes the energy density (product of two nearby spins) on an edge~$e$.
Using domain walls representations, this can be written as
\[
\langle\mu_{v_1}\dots\mu_{v_{2n}}\rangle=\frac{\cZ_{\opname{low}}^{[v_1,..,v_{2n}]}(G,x)}{\cZ_{\opname{low}}(G,x)}\qquad\text{with}\qquad
\cZ_\mathrm{low}^{[v_1,..,v_{2n}]}(G,x)\,:=\,\sum\nolimits_{P\in\cC(v_1,..,v_{2n})}x(P)\,,
\]
where $\cC(v_1,\dots,v_{2n}):=\{P:~ P\btu\varkappa^{[v_1,..,v_{2n}]}\in\cE(G)\}$ is the set of subgraphs~$P$ of~$G$ with all vertices of~$G$ except~$v_1,\dots,v_{2n}$ having even degrees, and all the~$v_k$ odd. (Above, we use the subscript~\emph{low} in order to emphasize that these subgraphs should be thought about as domain walls aka low-temperature expansions of the Ising model defined on faces of~$G$. Note that the Kramers-Wannier duality allows one to interpret~$\langle\mu_{v_1}\dots\mu_{v_{2n}}\rangle$ as the high-temperature expansion of the corresponding spin correlation in the dual model defined on vertices of~$G$.) It is clear that the quantity~$\langle\mu_{v_1}\dots\mu_{v_{2n}}\rangle$ does not depend on the choice of~$\varkappa^{[v_1,..,v_{2n}]}$. {However, the notation should not be directly interpreted probabilistically as the expectation of a product of random variables since the~$\mu_{v}$ themselves \emph{cannot} be thought of as random variables.}

Similarly to Section~\ref{sub:grassmann}, let us consider the canonical double-cover~$G^{[v_1,..,v_{2n}]}$ of the graph~$G$ with the branch set~$v_1,\dots,v_{2n}$, endowed with the involution~\mbox{$u\mapsto u^\sharp$} acting on its \emph{faces}. It is easy to see that~$\cZ_{\opname{low}}^{[v_1,..,v_{2n}]}(G,x)$ is the partition function of the Ising model defined on faces of~$G^{[v_1,..,v_{2n}]}$ with the \emph{spin-flip symmetry} constrain~$\sigma_{u^\sharp}=-\sigma_{u}$ and a fixed spin of the outer face. From this perspective, the choice of the collection of paths~$\varkappa^{[v_1,..,v_{2n}]}$ is nothing but a choice of a section of~$G^{[v_1,..,v_{2n}]}$. Given faces~$u_1,\dots,u_m$ of the \emph{double-cover}~$G^{[v_1,..,v_{2n}]}$, we set
\begin{equation}
\label{eqn:mu-sigma-def}
\langle\mu_{v_1}\dots\mu_{v_{2n}}\sigma_{u_1}\dots\sigma_{u_m}\rangle:=\mathbb{E}_{G^*}^{[v_1,..,v_m]}[\sigma_{u_1}\dots\sigma_{u_m}]\cdot \langle\mu_{v_1}\dots\mu_{v_{2n}}\rangle,
\end{equation}
where~$\mathbb{E}_{G^*}^{[v_1,..,v_m]}$ stands for the expectation in the Ising model described above. By definition, this quantity changes sign when one of~$u_k$ is replaced by~$u_k^\sharp$. Note that we allow repeating faces~$u_k$ in~\eqref{eqn:mu-sigma-def}, in which case the corresponding spins cancel out. By a slight abuse of the notation, one can also allow repeating disorders~$\mu_{v_k}$ with the same cancellation effect.

Let us now consider a special situation when~$m=2n$ and each of the faces~$u_k$ is incident to the corresponding vertex~$v_k$. More precisely, we consider a collection of~$2n$ \emph{pairwise distinct} corners~$c_1,\dots,c_{2n}$ of~$G$, put~$v_k:=v(c_k)$, and denote by~$u_k:=u(c_k)$ the face of~$G$ that contains~$c_k$. Note that we \emph{do} allow repetitions of these vertices and faces. Let
\begin{equation}
\label{eqn:def-corners-conf}
\cC(c_1,\dots,c_{2n}):=\{Q=P_Q\oplus c_1\oplus\dots\oplus c_{2n},~P_Q\in\cC(v_1,\dots,v_{2n})\}
\end{equation}
be the set of subgraphs from~$\cC(v_1,\dots,v_{2n})$ with decorations~$c_1,\dots,c_{2n}$ attached to the vertices~$v_1,\dots,v_{2n}$. Similarly to the case of oriented edges, we introduce a sign~$\tau(Q)\in\{\pm 1\}$ by resolving all the crossings of a configuration~$Q$ so that~$Q=C\sqcup\gamma_1\sqcup\dots\sqcup\gamma_n$, where~$n$ simple paths~$\gamma_k$ run from~$c_{s(2k-1)}$ to~$c_{s(2k)\vphantom{1}}$ and~$C$ is a collection of disjoint simple loops, and setting
\[
\tau(Q):={\sign(s)\,\cdot}\prod_{k=1}^n \left({i\eta_{c_{s(2k-1)}}^{\vphantom{|}}\overline{\eta}_{c_{s(2k\vphantom{1})}}} \exp[-{\textstyle\frac{i}{2}}\wind(\gamma_k)]\right)\,.
\]
Again, it is not hard to see that~$\tau(Q)$ is well-defined (i.e. independent of the smoothing of~$Q$).

\begin{lemma}
\label{lemma:mu-sigma}
Let~$c_1,\dots,c_{2n}\in V(G^\rC)$ be a collection of~$2n$ pairwise distinct corners of~$G$. Denote $v_k:=v(c_k)$ and let~$u_k=u(c_k)$ be the face of~$G$ that contains~$c_k$. Then,
\[
\langle\mu_{v_1}\dots\mu_{v_{2n}}\sigma_{u_1}\dots\sigma_{u_{2n}}\rangle ~=~ \pm\, [\cZ_{\opname{Ising}}(G,x)]^{-1}\,\cdot\!\!\!\!\!\sum_{Q\in\cC(c_1,\dots,c_{2n})}\!\!\!\tau(Q)x(P_Q)\,,
\]
with the sign depending on the choice of representatives of the faces~$u_1,\dots,u_{2n}$ on~$G^{[v_1,..,v_{2n}]}$.
\end{lemma}

\begin{proof}
Using the domain walls representation~$P_Q$ of the Ising model on~$G^{[v_1,..,v_{2n}]}$, it is easy to see that the lemma follows from the equality
\begin{equation}
\label{eqn:x-sigma-tauQ}
\sigma_{u_1}\dots\sigma_{u_{2n}}= \tau^0\cdot\tau(Q)\quad\text{for~all}\ \ Q\in\cC(c_1,\dots,c_{2n})\,,
\end{equation}
where the sign~$\tau^0\in\{\pm 1\}$ {is} independent of~$Q$. To prove this, let us fix a collection of edge-disjoint paths~$\varkappa^0=\varkappa_{[u_1,..,u_{2n}]}^0$ matching the faces~$u_1,\dots,u_{2n}$ on the \emph{dual} graph~$G^*$. We attach decorations~$c_1,\dots,c_{2n}$ to the endpoints of these paths and denote the result by~$\gamma_1^0,\dots,\gamma_n^0$, without loss of generality we can assume that~$\gamma_k^0$ runs from~$c_{2k-1}$ to~$c_{2k\vphantom{1}}$. Let
\[
\tau(\varkappa^0):=\prod_{k=1}^n ({-i\eta_{c_{2k-1}}^{\vphantom{|}}\overline{\eta}_{c_{2k\vphantom{1}}}} \exp[-{\textstyle\frac{i}{2}}\wind(\gamma_k^0)])\,.
\]
Note that, for a proper choice of representatives of~$u_1,\dots,u_{2n}$ on the double-cover~$G^{[v_1,\dots,v_{2n}]}$ and any configuration~$Q\in\cC(c_1,\dots,c_{2n})$, one has
\[
\sigma_{u_1}\dots\sigma_{u_{2n}} = (-1)^{\varkappa^0\cdot Q},
\]
here and below~$\alpha\cdot\beta$ denotes the intersection number of~$\alpha$ and~$\beta$. Let us consider a smoothing~$Q=C\sqcup\gamma$, where~$C$ is a collection of closed curves and~$\gamma$ a collection of~$n$ paths~$\gamma_k$ running from~$c_{s(2k-1)}$ to~$c_{s(2k\vphantom{1})}$, which are oriented so that the concatenation~$\gamma\oplus\varkappa^0$ becomes a collection~$\Delta$ of~$m$~oriented cycles~$\Delta_j$\,. Since~$C$ and~$\gamma$ do not intersect or self-intersect, one has
\[
\varkappa^0\!\cdot Q~=~\varkappa^0\!\cdot(C\sqcup\gamma)~=~\Delta\cdot (C\sqcup\gamma)~=~\Delta\cdot\gamma~=~\rt(\Delta)-\varkappa^0\!\cdot\varkappa^0\mod~2\,,
\]
where~$\rt(\Delta)$ is the number of self-intersections in~$\Delta$. On the other hand, using the
equality~{$\sign(s)=(-1)^m$} and Lemma~\ref{lemma:Whitney}, we see that
\[
\tau(Q)\tau(\varkappa^0)=(-1)^m\prod_{j=1}^m\exp[-{\textstyle\frac{i}{2}}\wind(\Delta_j)] =(-1)^{\sum_{j=1}^m\rt(\Delta_j)}=(-1)^{\rt(\Delta)},
\]
which implies~\eqref{eqn:x-sigma-tauQ} with~$\tau^0:=(-1)^{\varkappa^0\!\cdot\varkappa^0}\tau(\varkappa^0)$.
\end{proof}

\begin{remark}
\label{rem:x-mu-sigma}
Lemma~\ref{lemma:mu-sigma} can be easily generalized in the following way. In addition to the collection of~$2n$ corners~$c_1,\dots,c_{2n}$,
let us consider another~$m$ faces~$u_1',\dots,u_m'$ of~$G$ and let~$\varkappa'$ be a proper collection of edge-disjoint paths linking~$u_1',\dots,u_m'$ and, possibly,~$u_{\opname{out}}$ on~$G^*$. Note that we do not assume that these new faces are distinct from~$u_1,\dots,u_{2n}$ and we allow~$\varkappa'$ and~$\varkappa^0=\varkappa^0_{[u_1,..u_{2n}]}$ to share edges of~$G^*$. Repeating the proof of Lemma~\ref{lemma:mu-sigma}, we obtain
\[
\langle\mu_{v_1}\dots\mu_{v_{2n}}\sigma_{u_1}\dots\sigma_{u_{2n}}\sigma_{u_1'}\dots\sigma_{u_m'}\rangle ~=~ \pm\, [\cZ_{\opname{Ising}}(G,x)]^{-1}\,\cdot\!\!\!\!\!\sum_{Q\in\cC(c_1,\dots,c_{2n})}\!\!\!(-1)^{\varkappa'\cdot Q}\tau(Q)x(P_Q)
\]
since~$\sigma_{u_1}\dots\sigma_{u_{2n}}\sigma_{u_1'}\dots\sigma_{u_m'} = (-1)^{\varkappa^0\cdot Q}\cdot (-1)^{\varkappa'\cdot Q}$.
\end{remark}

\subsection{Equivalence of the two previous formalisms}
\label{sub:equivalence}
The aim of this section is to show that the two formalisms (Grassmann variables and disorder insertions) discussed above are essentially equivalent. This fact is quite well-known in the folklore but we do not know of a reference explaining this correspondence in an explicit manner, especially when working in presence of additional spin variables in the formal correlation functions. Note that Theorem~\ref{thm:multipoint} and its generalization provided by Proposition~\ref{prop:multipoint-with-spins} are the crucial ingredients needed to justify this equivalence.

We begin by introducing some additional notation. Let $\{\chi_{c}\}_{c\in V(G^\rC)}$ be another~$2|E(G)|$ Grassmann variables assigned to the \emph{corners} of the graph~$G$, which are related to the ``edge'' variables~$\phi_e$ discussed in Section~\ref{sub:grassmann} by the linear transform
\begin{equation}
\label{eqn:B-chi-to-phi}
\phi={\textstyle\frac{1}{2}}\,\wh{\rB}^\top\!\chi\,,\qquad\text{where}\quad \wh{\rB}^\top_{e,c}=\begin{cases}
{\overline{\eta}_c\eta_e^{\vphantom{|}}}\exp\left[\frac{i}{2}\rw(c,e)\right]\cdot x_e^{-1/2}& \text{if~$c=c^\pm(e)$;} \\
0 & \text{otherwise,}
\end{cases}
\end{equation}
see~\eqref{eqn:rB-def} and~\eqref{eqn:whB-whC-whF-def}. Note that this change of variables is local in the following sense: for a given vertex~$v\in V(G)$, the variables~$\phi_e$ with~$o(e)=v$ are linear combinations of the variables~$\chi_c$ with $v(c)=v$, and vice versa. Since~$\wh\rB\,\wh\rK\,\wh\rB^\top=\wh\rC$, the quadratic forms~$\chi^\top(\frac{1}{4}\wh\rC)\chi$ and~$\phi^\top\wh\rK\phi$ coincide. Therefore, one can think about the new variables~$\chi_c$ as being associated to the anti-symmetric matrix~$A=\frac{1}{4}\wh\rC$ in a standard way described in Section~\ref{sub:grassmann} so that
\[
\lan\chi_{c_1}\dots\chi_{c_{2n}}\!\ran_{\frac{1}{4}\wh\rC}~:= \Pf{4\wh\rC^{-1}_{c_j,c_k}}_{j,k=1}^{2n}\,.
\]
At the same time, for any~$c_1,\dots,c_{2n}$, one has
\[
\lan\chi_{c_1}\dots \chi_{c_{2n}}\!\ran_{\frac{1}{4}\wh\rC} ~=~ \lan \chi_{c_1}\dots \chi_{c_{2n}}\!\ran_{\wh\rK}\,,
\]
where the right-hand side should be understood as follows: write each of the variables~$\chi_{c_k}$ as a linear combination of the old variables~$\phi_e$ and then compute the arising linear combination of the terms~$\lan\phi_{e_1}\dots \phi_{e_{2n}}\!\ran_{\wh\rK}$\,. This allows us to drop the subscripts~$\wh\rK$ or~$\frac{1}{4}\wh\rC$ from the notation. The next lemma provides combinatorial expansions of the quantities~$\lan\chi_{c_1}\dots\chi_{c_{2n}}\!\ran$\,.

\begin{lemma}
\label{lemma:chi-2n-combinatorial}
 Let~$c_1,\dots,c_{2n}\in V(G^\rC)$ be a collection of corners of~$G$ adjacent to pairwise distinct vertices~$v_k=v(c_k)$. Then,
\begin{equation}
\label{eqn:chi-2n-combinatorial}
\lan\chi_{c_1}\dots\chi_{c_{2n}}\!\ran ~=~ [\cZ_{\opname{Ising}}(G,x)]^{-1}\,\cdot\!\!\!\!\!\sum_{Q\in\cC(c_1,\dots,c_{2n})}\!\!\!\tau(Q)x(P_Q)\,,
\end{equation}
where the set of configurations~$\cC(c_1,\dots,c_{2n})$ is given by~\eqref{eqn:def-corners-conf}.
\end{lemma}
\begin{proof} Recall that each of the variables~$\chi_{c_k}$ is a linear combination of the variables~$\phi_{e_k}$ with~$o(e_k)=v(c_k)$, and the \emph{inverse} transform is given by~\eqref{eqn:B-chi-to-phi}. Thus, in order to prove~\eqref{eqn:chi-2n-combinatorial}, it is enough to check that these equalities yield the correct combinatorial expansions of formal correlations $\lan\phi_{e_1}\dots\phi_{e_{2n}}\!\ran$, which are given by Theorem~\ref{thm:multipoint}.

According to~\eqref{eqn:B-chi-to-phi}, we have
\begin{equation}
\label{eqn:phi-is-two-chi}
\phi_e= {\eta_e} x_e^{-1/2}\cdot {\textstyle\frac{1}{2}}\left({ \overline{\eta}_{c^-(e)}}\exp[{\textstyle\frac{i}{2}}\rw(c^-(e),e)]\chi_{c^-(e)}+ {\overline{\eta}_{c^+(e)}}\exp[{\textstyle\frac{i}{2}}\rw(c^+(e),e)]\chi_{c^+(e)}\right),
\end{equation}
where $c^\pm(e)$ are the two decorations attached to the vertex $o(e)$ neighboring~$e$. Given a pair of configurations~$Q^\pm\in\cC(c^\pm(e_1),c_2,\dots,c_{2n})$ which differ by the decorations~$c^\pm(e_1)$ only, it is easy to check that the quantities~${\overline{\eta}_{c^\pm(e)}}\exp[{\textstyle\frac{i}{2}}\rw(c^\pm(e),e)]\tau(Q^\pm)$ coincide if~$e\in P_Q$ and are {opposite to each other} otherwise. In particular, the two corresponding contributions to~$\lan\phi_{e_1}\dots\phi_{e_{2n}}\!\ran$ cancel out in the latter case. Repeating the same argument for all the other edges~$e_2,\dots,e_{2n}$, one concludes that~\eqref{eqn:chi-2n-combinatorial} is equivalent to the following claim:
\[
\lan\phi_{e_1}\dots\phi_{e_{2n}}\!\ran ~=~ [\cZ_{\opname{Ising}}(G,x)]^{-1}\cdot\!\!\!\!\!\sum_{Q\in\cC^+(e_1,\dots,e_{2n})}\!\!\!\tau^+(Q)x^+(Q)\,,
\]
where the sum is taken over the set
\[
\cC^+(e_1,\dots,e_{2n}):=\{Q\in\cC(c^+(e_1),\dots,c^+(e_{2n})):\,e_1,\dots,e_{2n}\in Q\}\,,
\]
the modified weight of a configuration~$Q$ is given by~$x^+(Q):=x(P_Q)\cdot\prod_{k=1}^{2n}x_{e_k}^{-1/2}\!$ and the modified sign~$\tau^+(Q)$ is given by
\[
\tau^+(Q):=\tau(Q)\cdot\prod_{k=1}^{2n}({ \eta_{e_{k}}^{\vphantom{|}}\overline{\eta}_{c^+(e_{k})}}\exp[{\textstyle\frac{i}{2}}\rw(c^+(e_k),e_k)])\,.
\]

There exists a trivial bijection~$\varsigma:\cC^+(e_1,\dots,e_{2n})\to \cC(e_1,\dots,e_{2n})$: erase all the decorations~$c_k$ and the half-edges~$(o(e_k),z_{e_k})$ from a given configuration~$Q$ to get $\varsigma(Q)$. Clearly, one has~$x(\varsigma(Q))=x^+(Q)$ and it is easy to check that~$\tau(\varsigma(Q))=\tau^+(Q)$ for all~$Q$. Therefore, the collection of equalities~\eqref{eqn:chi-2n-combinatorial} is reduced to the claim of Theorem~\ref{thm:multipoint} and we are done.
\end{proof}

Let us now discuss modifications needed to include additional spin variables in the considerations above. For a given collection of faces~$u_1',\dots,u_m'$, denote
\[
\wh{\rC}_{[u_1,..,u_m]}:=\wh{\rB}\,\wh{\rK}_{[u_1',..,u_m']}\wh{\rB}^\top
\]
and, for a given collection of corners~$c_1,\dots,c_{2n}\in V(G^\rC)$, let
\[
\lan\chi_{c_1}\dots\chi_{c_{2n}}\!\ran_{[u_1',..,u_m']}~:=
\Pf{4(\wh\rC_{[u_1',..,u_m']}^{-1})_{c_j,c_k}}_{j,k=1}^{2n}\,.
\]
Similarly to Section~\ref{sub:grassmann}, this notation implicitly depends on the cuts~$\varkappa'=\varkappa'_{[u_1',..,u_m']}$ linking~$u_1',\dots,u_m'$ and, possibly,~$u_{\opname{out}}$ on~$G^*$, but can be made canonical by lifting to the double-cover~$G^\rC_{[u_1',..,u_m']}$\,. Using Proposition~\ref{prop:multipoint-with-spins} instead of Theorem~\ref{thm:multipoint}, one obtains the following combinatorial expansion which generalizes Lemma~\ref{lemma:chi-2n-combinatorial} in the ``twisted'' setup:
\begin{equation}
\label{eqn:chi-2n-spin-combinatorial}
\lan\chi_{c_1}\dots\chi_{c_{2n}}\!\ran_{[u_1',..,u_m']} ~=~ \frac{\sum_{Q\in\cC(c_1,\dots,c_{2n})}(-1)^{\varkappa'\cdot Q}\tau(Q)x(P_Q)}{\mathbb{E}_{G^*}^+[\sigma_{u_1'}\dots\sigma_{u_m'}]\cdot\cZ_{\opname{Ising}}(G,x)}\,,
\end{equation}
where~$c_1,\dots,c_{2n}$ are thought about as lying on a section of~$G^\rC_{[u_1',..,u_m']}$ constructed via~$\varkappa'$.

We are now able to justify the equivalence of the two formalisms discussed above: Grassmann variables (considered on double-covers) and disorder insertions. The next result claims that the formal correlation functions introduced in Sections~\ref{sub:grassmann} and~\ref{sub:disorders}, respectively, are essentially the same, with the correspondence given by the formal rule~$\chi_{c_k}=\mu_{v_k}\sigma_{u_k}$\,.

\begin{proposition}
\label{prop:equivalence}
Let~$c_1,\dots,c_{2n}\in V(G^\rC)$ be a collection of corners of~$G$ adjacent to pairwise distinct vertices~$v_k=v(c_k)$ and let~$u_k=u(c_k)$ be the face of~$G$ that contains~$c_k$. Then, for an arbitrary collection of faces~$u_1',\dots,u_m'$ of~$G$, one has
\[
\langle\mu_{v_1}\dots\mu_{v_{2n}}\sigma_{u_1}\dots\sigma_{u_{2n}}\sigma_{u_1'}\dots\sigma_{u_{m'}}\rangle ~=~
\pm \lan\chi_{c_1}\dots\chi_{c_{2n}}\!\ran_{[u_1',..,u_m']}\cdot\,\langle\sigma_{u_1'}\dots\sigma_{u_m'}\rangle\,,
\]
where the sign depends on the choice of representatives of the faces~$u_1,\dots,u_{2n},u_1',\dots,u_m'$ on the double-cover~$G^{[v_1,..,v_{2n}]}$ and representatives of the corners~$c_1,\dots,c_{2n}$ on~$G_{[u_1',\dots,u_m']}$\,. Above, we do \emph{not} assume that the faces~$u_1,\dots,u_{2n},u_1',\dots,u_m'$ are pairwise distinct.
\end{proposition}
\begin{proof}
Let us consider the special case~$m=0$ first. In this case, the equality
\[
\langle\mu_{v_1}\dots\mu_{v_{2n}}\sigma_{u_1}\dots\sigma_{u_{2n}}\rangle ~=~\pm \lan\chi_{c_1}\dots\chi_{c_{2n}}\!\ran
\]
directly follows from Lemma~\ref{lemma:mu-sigma} and Lemma~\ref{lemma:chi-2n-combinatorial} since both sides have identical combinatorial expansions. In the general situation, one just uses Remark~\ref{rem:x-mu-sigma} and formula~\eqref{eqn:chi-2n-spin-combinatorial} instead of these lemmas. The claim follows since~$\langle\sigma_{u_1'}\dots\sigma_{u_m'}\rangle=\mathbb{E}_{G^*}^+[\sigma_{u_1'}\dots\sigma_{u_m'}]$.
\end{proof}

\begin{remark}
\label{rem:two-formalisms}
\emph{(i)} It is worth noting that Lemma~\ref{lemma:chi-2n-combinatorial} and Proposition~\ref{prop:equivalence} \emph{fail} if one drops the assumption that the vertices~$v_k=v(c_k)$ are pairwise distinct. Indeed, if we consider two edges~$e_1,e_2$ such that~$c^+(e_1)=c^-(e_2)$, then the product~$\phi_{e_1}\phi_{e_2}$ is \emph{not} equal to the sum of \emph{four} terms since $\chi_{c^+(e_1)}\chi_{c^-(e_2)}=0$. Instead, we have only three terms and the combinatorial correspondence of configurations used in the proof of Lemma~\ref{lemma:chi-2n-combinatorial} breaks down.

\noindent \emph{(ii)} One can easily make sense of the notation~$\lan\phi_e\chi_c\dots\ran_{[u_1,..,u_m]}$, with the variables labeled by oriented edges \emph{or} corners of the canonical double-cover~$G_{[u_1,..,u_m]}$. In order to define these quantities, just rewrite all participating variables using one of the two sets~$\phi_e$ or~$\chi_c$, and compute the arising linear combination of the terms~\mbox{$\lan\phi_e\phi_{e'}\dots\ran_{[u_1',..,u_m']}$} or~\mbox{$\lan\chi_{c'}\chi_c\dots\ran_{[u_1',..,u_m']}$}\,: the result does not depend on which set of variables was used.  Following the proof of Lemma~\ref{lemma:chi-2n-combinatorial}, it is easy to obtain \emph{combinatorial expansions} of such quantities in the situation when all the corresponding vertices~$v(c_k)$ are pairwise distinct and do not coincide with the vertices~$o(e_j)$ for~$\phi_{e_j}$ involved in the formal correlation function under consideration.
\end{remark}

\subsection{Three-term relation for correlation functions} \label{sub:three-term}
It is well-known that the formal correlation functions~$\langle\chi_c\dots\rangle$ involving any three of the four corners surrounding a given edge~$e$ of~$G$ satisfy a linear relation known as the \emph{propagation equation} for discrete spinors or the Dotsenko equation. The latter name was suggested in~\cite{Mercat-CMP} to acknowledge the paper~\cite{dotsenko-dotsenko} where this propagation equation was discussed in the ``combinatorial'' context of the disorder insertions formalism, {though it is worth mentioning that similar relations appeared earlier, e.g. in the works of Perk~\cite{Perk,Perk-Dubna81}}. Below we start with a short derivation given in~\cite{dotsenko-dotsenko} and then discuss this equation from the (equivalent) viewpoint of the Grassmann variables formalism.

Informally speaking, the main idea is to apply the Kramers-Wannier duality \emph{locally} on a given edge~$e$. It is convenient to introduce the following parametrization of the edge weights:
\[
\theta_e:=2\arctan x_e\,,\qquad p_e:=\cos\theta_e =\frac{1-x_e^2}{1+x_e^2}\,,\qquad q_e:=\sin\theta_e=\frac{2x_e}{1+x_e^2}\,.
\]
By the definition of disorder insertions (see Section~\ref{sub:disorders}) and the equality~$q_ex_e^{\reps_e}=1-p_e\reps_e$ for~$\reps_e=\pm 1$, for any combination~$\cO[\mu,\sigma]$ of spins and (even number of) disorders, we have
\begin{align}
\label{eqn:mu-tau-local-duality}
q_e\cdot \langle\mu_{o(e)}\mu_{t(e)}\cO[\mu,\sigma]\rangle ~ & = ~ q_e\cdot\langle x_e^{\reps_e}\,\cO[\mu,\sigma] \rangle \notag \\ & = \langle \cO[\mu,\sigma] \rangle - p_e\cdot\langle \sigma_{u^-(e)}\sigma_{u^+(e)}\cO[\mu,\sigma] \rangle\,,
\end{align}
where~$\reps_e=\sigma_{u^-(e)}\sigma_{u^+(e)}$ and~$u^\pm(e)=u(c^\pm(e))$ are the two faces of~$G$ adjacent to the oriented edge~$e$, with~$u_-(e)$ being to the right and~$u^+(e)$ to the left of~$e$.
Note that the set of disorders involved in the left-hand side of this equality differs from that in the right-hand side, so one should be careful with the signs of the formal correlations even though there is a trivial correspondence between the faces of these double-covers. Above, the faces~$u^\pm(e)$ are assumed to be adjacent on the double-cover used to define the correlation~$\langle\cO[\mu,\sigma]\rangle$.

Let us now replace the collection of spins and disorders~$\cO[\mu,\sigma]$ by~$\mu_{o(e)}\sigma_{u^+(e)}\cO[\mu,\sigma]$ and recall that any repeating variables in these formal correlations cancel out. Rewriting~\eqref{eqn:mu-tau-local-duality} (note that now~$\cO[\mu,\sigma]$ must contain an odd number of disorders), one obtains
\begin{equation}
\label{eqn:three-term-mu-sigma}
\langle \mu_{o(e)}\sigma_{u^+(e)}\cO[\mu,\sigma] \rangle = p_e\cdot \langle \mu_{o(e)}\sigma_{u^-(e)}\cO[\mu,\sigma] \rangle  + q_e\cdot\langle\mu_{t(e)}\sigma_{u^+(e)}\cO[\mu,\sigma]\rangle\,,
\end{equation}
with a proper correspondence between the involved double-covers.
\begin{remark}
\label{rem:corner-universal-double-cover}
There exists a way to make this correspondence of double-covers canonical. Let~$\wh{G}^\rC$ denote the corner graph~$G^\rC$ with all the intersecting edges removed. Given~$\cO[\mu,\sigma]$, one considers a double-cover~$\wh{G}^\rC_{\cO[\mu,\sigma]}$ of~$\wh{G}^\rC$ branching around \emph{all} the vertices and the faces of~$G$ that are \emph{not} involved in~$\cO[\mu,\sigma]$, as well as around all the {edges} of~$G$. On this double-cover, the formal correlations~$\langle\mu_{v(c)}\sigma_{u(c)}\cO[\mu,\sigma]\rangle$ defined in Section~\ref{sub:disorders} obey the sign-flip symmetry between the sheets and satisfy \eqref{eqn:three-term-mu-sigma} around all the edges, see~\cite[pp.~209--210]{Mercat-CMP}.
\end{remark}

\begin{remark}
\label{rem:ad-hoc-combinatorics}
The propagation equation~\eqref{eqn:three-term-mu-sigma} can easily be derived using Lemma~\ref{lemma:mu-sigma} (or its generalization provided in Remark~\ref{rem:x-mu-sigma}) and playing with the natural correspondence (given by adding/removing the edge~$e$) between the sets of subgraphs of~$G$ involved in the relevant combinatorial expansions. This approach is conceptually equivalent to the derivation given above, but it allows one {to change the viewpoint} and to use these combinatorial expansions as (slightly mysterious) \emph{ad~hoc} \emph{definitions} of the objects of interest, making use of very elementary concepts only. Such a shortcut was advertised by Smirnov~\cite{smirnov-icm-2006,smirnov-icm-2010} and is very useful when working with complex-valued fermionic observables, see Section~\ref{sub:s-observables} for details.
\end{remark}

We now discuss how one can see the three-point relation~\eqref{eqn:three-term-mu-sigma} using the intrinsic structure of the Kac--Ward matrices or, more precisely, the matrices~$\rC_{[u_1,\dots,u_m]}$. Let us introduce a matrix~$\rY$ whose entries are labeled by the corners of~$G$ as follows:
\[
\rY_{c,c'}:=\begin{cases} \exp[{\textstyle\frac{i}{2}}\rw(c,\overline{c}')] & \text{if~$v(c)=v(c')$~but~$c\ne c'$;} \\
0 & \text{otherwise.}
\end{cases}
\]
Note that the matrix~$\rY$ is Hermitian and has a block-diagonal structure with blocks corresponding to vertices of~$G$. Further, for a given collection of cuts~$\varkappa=\varkappa_{[u_1,..,u_m]}$ linking the faces~$u_1,\dots,u_m$ and, possibly,~$u_{\opname{out}}$ on~$G^*$, let
\[
\left(\rD_{[u_1,..,u_m]}\right)_{c,c'}=\begin{cases} -i & \text{if~$c=c'$;} \\
p_e\cdot\exp[{\textstyle\frac{i}{2}}\rw(c,\overline{c}')] & \text{if~$c=c^+(e)$~and~$c'=c^-(e)$~for~some~$e$;} \\
q_e\cdot(-1)^{\varkappa\cdot e}\exp[{\textstyle\frac{i}{2}}\rw(c,\overline{c}')] & \text{if~$c=c^+(e)$~and~$c'=c^-(\overline{e})$~for~some~$e$}; \\
0 & \text{otherwise,}
\end{cases}
\]
where the rotation angle~$\rw(c,\overline{c}')$ in the third line is measured along the path~$c\oplus e\oplus \overline{c}'$.
\begin{remark}
\label{rem:spinors-functions}
The operator~$\rD_{[u_1,..,u_m]}$ can be viewed as the ``untwisted'' operator~$\rD$ acting on functions defined on the double-cover~$G^\rC_{[u_1,..,u_m]}$ and obeying a sign-flip symmetry between the sheets; such functions are sometimes called~\emph{spinors}. From this viewpoint, the above definition describes the action of~$\rD_{[u_1,..,u_,]}$ on a section of~$G^\rC_{[u_1,..,u_m]}$ given by the cuts~$\varkappa_{[u_1,..,u_m]}$\,.
\end{remark}

\begin{lemma} \label{lemma:C-three-terms}
For any collection of faces~$u_1,\dots,u_m$ of~$G$, the following identity holds:
\begin{equation}
\label{eqn:C^-1=D^-1}
4\rC^{-1}_{[u_1,..,u_m]}+\rY+i\rI = 2\rD^{-1}_{[u_1,..,u_m]}\,.
\end{equation}
\end{lemma}
\begin{remark}
\label{rem:chi-adjusted}
Before giving a proof of this identity, recall that~$\wh\rC={i\rU_\rC^*\rC\rU_\rC^{\vphantom{*}}}$ and introduce the \emph{real-valued} counterparts of the matrices~$\rY$ and~$\rD_{[u_1,..,u_m]}$ defined by
\[
\wh\rY:={i\rU_\rC^*\rY\rU_\rC^{\vphantom{*}}}\quad \text{and}\quad \wh\rD_{[u_1,..,u_m]}:={ i\rU_\rC^*\rD_{[u_1,..,u_m]}\rU_\rC^{\vphantom{*}}}\,.
\]
Then we have~$\wh\rD_{[u_1,..,u_m]} \cdot [4\wh\rC^{-1}_{[u_1,..,u_m]} - \wh\rY + \rI] = 2\rI$.
In other words, for any oriented edge~$e$ and any corner~$c\neq c^+(e)$, the quantities
\[
\lan\chi_{c'}\chi_{c}\ran_{[u_1,..,u_m]}-\wh\rY_{c',c}+\rI_{c',c}\,,\quad \text{where}\quad c'=c^\pm(e)\ \ \text{or}\ \ c'=c^-(\overline{e})\,,
\]
satisfy a three-term linear relation with coefficients provided by the matrix~$\wh{\rD}_{[u_1,..,u_m]}$\,, and one should replace the last term~$\rI_{c',c}$ by $-\rI_{c',c}$ if~$c=c^+(e)$. The local terms~$(-\wh\rY\pm\rI)_{c',c}$ compensate the mismatch between the Grassmann variables formalism and disorder insertions in the situation~$v(c')=v(c)$, see Remark~\ref{rem:two-formalisms}. Modulo this local adjustment of the formal correlation functions, equation~\eqref{eqn:C^-1=D^-1} is equivalent to the propagation equation~\eqref{eqn:three-term-mu-sigma} for two-disorders correlations. Its extension to~$2n$ disorders (equivalently,~$2n$ Grassmann variables~$\chi_c$) is then provided by the Pfaffian identities and linearity.
\end{remark}

\begin{proof}[Proof of Lemma~\ref{lemma:C-three-terms}]
Equality~\eqref{eqn:C^-1=D^-1} is equivalent to the following claim:
\begin{equation}
\label{eqn:C-factorization}
\rC_{[u_1,..,u_m]}= \rD_{[u_1,..,u_m]}\cdot [2\rI+{\textstyle\frac{1}{2}}(\rY+i\rI)\rC_{[u_1,..,u_m]}]\,,
\end{equation}
which can be easily checked in two steps. One begins by computing the matrix
\[
\rS_{[u_1,..,u_m]}:={\textstyle\frac{1}{2}}(\rY+i\rI)\rC_{[u_1,..,u_m]}\,,
\]
whose entries are given by
\begin{equation}
\label{eqn:Q-entries}
\left(\rS_{[u_1,..,u_m]}\right)_{c,c'}=\begin{cases}
-1 & \text{if~$c=c'$;} \\
-i\exp[\frac{i}{2}\rw(c,\overline{c}')] & \text{if~$c=c^-(e)$~and~$c'=c^+(e)$;} \\
i\exp[\frac{i}{2}\rw(c,\overline{c}')]\cdot(-1)^{\varkappa\cdot e}x_e^{-1} & \text{if~$c=c^-(e)$~and~$c'=c^\pm(\overline{e})$;}\\
0 & \text{otherwise,}
\end{cases}
\end{equation}
where the rotation angles~$\rw(c,\overline{c}')$ for~$c'=c^\pm(\overline{e})$ in the third line are measured along~$c\oplus e \oplus \overline{c}'$.
Then, another straightforward computation leads to~\eqref{eqn:C-factorization}.
\end{proof}

Let the matrix~$\rW_{[u_1,..,u_m]}$ be defined by
\[
\left(\rW_{[u_1,..,u_m]}\right)_{c,c'}:=\begin{cases} \exp[\frac{i}{2}\rw(c,\overline{c}')] & \text{if~$c=c^-(e)$~and~$c'=c^+(e)$;} \\
-\exp[\frac{i}{2}\rw(c,\overline{c}')]\cdot(-1)^{\varkappa\cdot e}x_e^{-1} & \text{if~$c=c^-(e)$~and~$c'=c^+(\overline{e})$;}\\
0 & \text{otherwise,}\end{cases}
\]
and note that~$\rW_{[u_1,..,u_m]}\rW_{[u_1,..,u_m]}^*$ is a \emph{diagonal} matrix with entries~$1+x_e^{-2}$ for~$c=c^\pm(e)$\,. A straightforward computation shows
\begin{align*}
{\textstyle\frac{1}{2}}(\rY+i\rI)\rC_{[u_1,..,u_m]} & = \rW_{[u_1,..,u_m]}\rD_{[u_1,..,u_m]}\,.
\end{align*}
Similarly, one can easily see that
\[
{\textstyle\frac{1}{2}}(\rY-i\rI)\rC_{[u_1,..,u_m]} = \rW_{[u_1,..,u_m]}^*\rD_{[u_1,..,u_m]}^*\,,
\]
for instance by checking the identity
\[
i\rC_{[u_1,..,u_m]}= \rW_{[u_1,..,u_m]}\rD_{[u_1,..,u_m]} - \rW_{[u_1,..,u_m]}^*\rD_{[u_1,..,u_m]}^*\,.
\]
\begin{remark}\label{rk:d-bar}
\emph{(i)} Since~$\rD_{[u_1,..,u_m]}$ can be {thought of} as some~$\overline{\partial}$-type operator acting on the corresponding double-cover~$G^\rC_{[u_1,..,u_m]}$,  identities of this sort are useful when studying the links between the Kac--Ward matrices
and discrete holomorphic functions, see~\cite{cimasoni-AIHP} for the discussion of general surface graphs.

\noindent \emph{(ii)}
Since the matrix $\rW\rW^*$ is diagonal, the last representation of the quadratic form~$\chi^\top\wh\rC\chi$ provides an appropriate {starting point for the} interpretation of the {(scaling limit of the)} Ising model defined on a general planar graph~$G$ from the \emph{``Free Fermionic Field''} {perspective}, cf.~\cite{Lieb-fermions,Itzykson-fermions,Plechko} and~\cite[Section 2.1.2]{YellowBook},~{\cite[Section 9.7]{MussardoBook}}.
\end{remark}

\subsection{Complex-valued fermionic observables and s-holomorphicity} \label{sub:s-observables}

The aim of this section is to discuss the complex-valued versions of the formal correlation functions (aka discrete fermionic observables) introduced in Section~\ref{sub:grassmann}--\ref{sub:equivalence}. Those can be defined as simple linear combinations of the real-valued ones, so one should not expect a major difference between the two viewpoints. Nevertheless, it turns out that the complex-valued observables are much better adapted to the analysis of \emph{boundary value problems} arising when studying the scaling limit of the Ising model in general planar domains. At the same time, they can be constructed \emph{ad~hoc} in a purely combinatorial way and all the needed local relations follow easily, cf.~Remark~\ref{rem:ad-hoc-combinatorics}. Such a \emph{definition} was advertised by Smirnov in the 2000s (see~\cite[Section~4]{smirnov-icm-2010} for historical remarks) and then used in a series of recent papers of Chelkak, Duminil-Copin, Hongler, Izyurov, Kemppainen, Kyt\"ol\"a and others devoted to the conformal invariance of correlation functions and interfaces arising in the scaling limit of the critical Ising model in bounded planar domains.

For simplicity, below we mostly discuss the ``untwisted'' situation. As usual, to handle the general case one should consider a relevant double-cover and work with spinors defined on this cover instead of functions defined on~$V(G^\rK)$ or~$V(G^\rC)$, see~\eqref{eqn:phi-spinor} and Remark~\ref{rem:corner-universal-double-cover}. {To simplify the notation, we assume that the global unimodular factor in the definition of $\eta_e$ and $\eta_c$ is chosen as $\zeta=e^{i\frac{\pi}{4}}$.} For the midpoint~$z_e$ of an edge~$e\in E(G)$, define
\begin{equation}
\label{eqn:psi(z_e)-def}
\psi(z_e):= {t_e\cdot (\eta_e\phi_e + \eta_{\bar{e}}\phi_{\bar{e}}),}
\end{equation}
where the additional normalizing factor~$t_e:=(x_e+x_e^{-1})^{1/2}$ is added for later convenience; note that~$\psi(z_e)$ does not depend on the orientation of~$e$. This allows us to speak about formal correlation functions of these new variables like~$\lan\psi(z_e)\phi_a\ran$ or~$\lan\psi(z_e)\chi_c\ran$.
In particular, for a given oriented edge~$a\in V(G^\rK)$ and~$z_e\neq z_a$, Theorem~\ref{thm:multipoint} implies
\begin{equation}
\label{eqn:stas-observable-def}
F_a(z_e)~:=~\lan \psi(z_e)\phi_a\ran~=~\frac{{(-i\eta_a)\,\cdot}\sum_{P\in \cC (a,z_e)}\exp[-{\textstyle\frac{i}{2}}\wind(\gamma_P)]\, t_e x(P)}{\cZ_{\opname{Ising}}(G,x)}\,,
\end{equation}
where~$\cC(a,z_e):=\cC(a,e)\cup\cC(a,\overline{e})$, the non-self-intersecting curve~$\gamma_P$ is obtained from a configuration~$P\in\cC(a,z_e)$ by an arbitrary resolution of all its crossings, and~$\wind(\gamma_P)$ stands for the total rotation angle of the (velocity vector of this) curve~$\gamma_P$ when it runs from~$a$ to~$z_e$. Among other papers, this \emph{combinatorial definition} can be found:
\begin{itemize}
\item In the original work of Smirnov devoted to the understanding of the scaling limit of \emph{interfaces} (domain walls) arising in the critical Ising model on the square lattice, with~$a$ being a boundary edge, see~\cite[Section~4]{smirnov-icm-2010} for references.
\item In the paper~\cite{chelkak-smirnov} devoted to the \emph{universality} of these scaling limits for the critical Ising models defined on arbitrary isoradial graphs; note that the normalizing factor~$t_e$ introduced above matches the factor~$(\cos{\textstyle\frac{1}{2}}\theta_e)^{-1}=(1+x_e^2)^{1/2}$ used in~\cite[Section~2.2]{chelkak-smirnov} since we included the half-weight of the last edge~$e$ into~$x(P)$.
\item In the paper~\cite{hongler-smirnov} and the PhD thesis~\cite{Hongler-thesis} of Hongler, which is devoted to the study of the scaling limit of the~\emph{energy density field} in the critical Ising model in bounded planar domains (on the square lattice), with~$a$ being an internal edge.
\item In the paper~\cite{Chelkak-Izyurov} and the PhD thesis~\cite{izyurov-thesis} of Izyurov, where the \emph{spinor version} of~\eqref{eqn:stas-observable-def} was first suggested as a tool to study scaling limits of (ratios of) spin correlations and interfaces in the critical Ising model considered in multiply-connected domains.
\item In the paper~\cite{CHI} devoted to the study of the scaling limit of the~\emph{spin field} in the critical Ising model in bounded planar domains (on the square lattice), where the branching ``source-at-corner'' observable~$\lan\psi(z_e)\chi_c\ran_{[u_1,..,u_m]}$ was used, with~$u_1=u(c)$.
\end{itemize}

Further, with a slight abuse of notation, let us denote
\[
\psi(c):={\eta_c}\chi_c\quad\text{for}~c\in V(G^\rC)
\]
and extend definition~\eqref{eqn:stas-observable-def} of the function~$F_a(\cdot)$ from the set of midedges~$z_e$ of the graph~$G$ to the set of its corners by defining, for~$c\in V(G^\rC)$,
\[
F_a(c)~:=~\lan\psi(c)\phi_a\ran~\in~{\eta_c}\R\,.
\]
These quantities admit combinatorial expansions similar to~\eqref{eqn:stas-observable-def}, see Lemma~\ref{lemma:chi-2n-combinatorial} and Remark~\ref{rem:two-formalisms}(ii); note that~$F_a(\cdot)$ depends on the choice of the square root {in the definition of}~$\eta_a$ but is \emph{independent} of all other choices. The following notion first appeared in~\cite{Smirnov-10,chelkak-smirnov} in the critical planar Ising model context and was recently discussed in~\cite{cimasoni-AIHP} for arbitrary surface graphs. Recall that we use the parametrization~$x_e=\tan\frac{1}{2}\theta_e$ of the edge weights.

\begin{definition}
\label{def:s-hol}
A complex-valued function~$F$ defined on edge midpoints~$z_e$ and, simultaneously, on corners~$c$ of a given weighted graph~$(G,x)$ embedded in the complex plane satisfies a \emph{generalized s-holomorphicity condition} for a pair~$z_e$ and~$c=c^\pm(e)$ if
\begin{equation}
\label{eqn:s-hol-def}
F(c) ~ = ~ e^{\frac{i}{2}(\rw(c,e)\mp(\pi-\theta_e))}\cdot \Proj{F(z_e)}{{e^{\pm \frac{i}{2}(\pi-\theta_e)}\eta_e}}\,,
\end{equation}
where, as usual,~$\rw(c,e)$ denotes the rotation angle between the decoration~$c$ oriented \emph{towards} the vertex~$v(c)=o(e)$ and the oriented edge~$e$, and~$\Proj{F}{\nu}:=\Re[F\bar{\nu}]\nu=\frac{1}{2}[F+\nu^2\overline{F}\,]$,
and thus the choice of the sign of~$\eta_e$ in~\eqref{eqn:s-hol-def} is irrelevant.
\end{definition}

It is well-known that
the complex-valued observables~$F_a(\cdot)=\lan\psi(\cdot)\phi_a\ran$ introduced above, as well as the ``source-at-corner'' observables~$F_c(\cdot)=\lan\psi(\cdot)\chi_c\ran$ and their spinor counterparts, satisfy the generalized s-holomorphicity condition for all pairs~$(z_e,c^\pm(e))$ except near the ``source'' edge~$a$ or the corner~$c$, respectively.
(Actually, if one uses direct combinatorial definitions, with a proper treatment of the values~$F_a(z_a)$ or~$F_c(c)$, instead of the formal correlations of Grassmann variables, these local relations are satisfied even near the ``source'' edge~$a$ or the corner~$c$, cf.~Remark~\ref{rem:chi-adjusted}.) A simple \emph{combinatorial} proof of~\eqref{eqn:s-hol-def} in the special situation when~$\rw(c^\pm(e),e)=\pm(\pi-\theta_e)$ can be found in many places (e.g. see~\cite[Section~4]{smirnov-icm-2010} for the square lattice case or a non-optimal version of the same argument~\cite[Section~2.2]{chelkak-smirnov} for isoradial graphs), and these proofs can be trivially adapted to the general situation.

Let us now sketch the proof of another well-known fact saying that the s-holomorphicity condition is essentially equivalent to the propagation equation~\eqref{eqn:three-term-mu-sigma}, see~\cite[Lemma~3.4]{chelkak-smirnov}, or the \emph{algebraic} identity~\eqref{eqn:C^-1=D^-1}, cf.~\cite[Section~2.1]{Lis-2013} and~\cite[Theorem~4.2]{cimasoni-AIHP}. Indeed, using~\eqref{eqn:phi-is-two-chi} one can rewrite definition~\eqref{eqn:psi(z_e)-def} in the following form, independent of the choices of~$\eta_e$ and~$\eta_c$:
\[
\psi(z_e)={\textstyle\frac{1}{2}t_ex_e^{-1/2}}\cdot \left[\sum\nolimits_{c=c^\pm(e)}e^{-\frac{i}{2}\rw(c,e)}\psi(c)+
\sum\nolimits_{c=c^\pm(\overline{e})}e^{-\frac{i}{2}\rw(c,\overline{e})}\psi(c)\right],
\]
where each of the sums contains two terms. Further, each of the four variables~$\psi(c)$ involved into these sums produces a term~$\lan\psi(c)\phi_a\ran$
with a prescribed complex phase~${\eta_c}$. Thus one can use the identity~$\Proj{\alpha r}{\nu}=\frac{1}{2}[1+\nu^2\overline{\alpha}^2]\cdot\alpha r$ for~$r\in\R$ and straightforward computations to get the following form of the right-hand side of~\eqref{eqn:s-hol-def} for~$F(z_e)=\lan \psi(z_e)\phi_a\ran$:
\[
{\textstyle\frac{1}{2}}e^{\frac{i}{2}\rw(c,e)}\cdot \left[ \sum\nolimits_{c=c^\pm(e)}e^{-\frac{i}{2}\rw(c,e)}\lan\psi(c)\phi_a\ran
~\mp~ ix_e^{-1}\sum\nolimits_{c=c^\pm(\overline{e})}e^{-\frac{i}{2}\rw(c,\overline{e})}\lan \psi(c)\phi_a\ran\right].
\]
Therefore, the s-holomorphicity condition~\eqref{eqn:s-hol-def} for the function~$F_a(\cdot)=\lan \psi(\cdot)\phi_a\ran$ and a corner~$c=c^-(e)$ can be \emph{equivalently} rewritten as
\begin{equation}
\label{eqn:x-s-hol-equiv}
\sum\nolimits_{c'\in\{c^\pm(e),c^\pm(\overline{e})\}} \rS_{c,c'}\lan\psi(c')\phi_a\ran ~=~ 0\,,
\end{equation}
with the coefficients~$\rS_{c,c'}$ given by~\eqref{eqn:Q-entries}. By definition,~$\phi_a$ is a (real) linear combination of the two corner variables~$\chi_{c^\pm(a)}$, so it is easy to see that
\[
\lan\psi(c')\phi_a\ran = {\kappa}_- {\rC}^{-1}_{c',c^-(a)} + {\kappa}_+ {\rC}^{-1}_{c',c^+(a)}\quad \text{for~any}~c'\in V(G^\rC)
\]
with some (complex) coefficients~$\kappa_\pm$. Since~$\rS\rC^{-1}=\frac{1}{2}(\rY+i\rI)$, equality~\eqref{eqn:x-s-hol-equiv} easily follows  provided~$v(c)\ne o(a)$. The analog of~\eqref{eqn:x-s-hol-equiv} for~$c=c^+(e)$ can be checked in the same way, this time with coefficients given by the entries of the matrix~$\frac{1}{2}(\rY-i\rI)\rC=\rS-i\rC$.

\begin{remark}
\label{rem:boundary-conditions} Let us assume that~$v_1\in V(G)$ is a degree~$1$ vertex and $e_1\in\EE(G)$ is the unique oriented edge of~$G$ satisfying~$t(e_1)=v_1$, cf.~Remark~\ref{rem:degree-one-vertices}. Despite the fact that adding/removing such vertices to the graph~$G$ does not affect the Ising model on the dual graph~$G^*$, allowing them is sometimes useful, notably when the outer face~$u_{\opname{out}}$ has a huge degree. In this case, it is convenient to add such a vertex~$v_1$ to each of the vertices~$v\in V(G)$ incident to~$u_{\opname{out}}$ in order to speak about \emph{boundary conditions} satisfied by complex-valued observables~$F_a(\cdot)$. As the only vertex incident to~$\overline{e}_1$ in the terminal graph~$G^\rK$ is~$e_1$, one has
\begin{equation}
\label{eqn:boundary-conditions}
F_a(z_{e_1})~ = ~ \lan \psi(z_{e_1})\phi_a\ran ~ = ~ t_{e_1}{\eta_{\overline{e}_1}}\cdot\lan \phi_{\overline{e}_1}\phi_a\ran ~ \in ~ { \eta_{\overline{e}_1}}\R ~ = ~ {i\eta_{e_1}}\R
\end{equation}
for all~$a\ne \overline{e}_1$, since~$\lan\phi_{e_1}\phi_a\ran  = \wh{\rK}^{-1}_{e_1,a}=0$ unless~$a=\overline{e}_1$. Clearly, this property holds for all versions of complex-valued fermionic observables discussed above, including the spinor ones. Again, these boundary conditions become even more transparent if one just starts with the combinatorial descriptions of these observables, e.g. with formula~\eqref{eqn:stas-observable-def} for~$F_a(z_e)$.
\end{remark}

We conclude this section by a brief discussion of the general strategy used in the papers~\cite{chelkak-smirnov,hongler-smirnov,Hongler-thesis,Chelkak-Izyurov,CHI} mentioned above to prove the convergence, as~$\delta\to 0$, of various correlation functions in the critical Ising model considered on refining discrete approximations~$\Omega_\delta$ to a given planar domain~$\Omega$. As an example for this discussion, we use the energy density expectations~\eqref{eqn:energy-Pfaff} treated in~\cite{hongler-smirnov,Hongler-thesis}. For these expectations, we need some simple preliminaries reflecting their algebraic structure. Similarly to~\eqref{eqn:psi(z_e)-def}, let us define
\[
\psi^{\star}(z_e):= {t_e\cdot (\overline{\eta}_e\phi_e + \overline{\eta}_{\bar{e}}\phi_{\bar{e}})}
\]
and let
\[
\begin{array}{rcl}
\Psi(z_e,z_a)&:=\ \lan\psi(z_e)\psi(z_a)\ran & = \ {t_a\cdot (\eta_a F_a(z_e) + \eta_{\bar{a}} F_{\bar{a}}(z_e))}\,,\\
\Psi^{\star}(z_e,z_a)&:=\ \lan\psi(z_e)\psi^{\star}(z_a)\ran  & = \ {t_a\cdot (\overline{\eta}_a F_a(z_e) + \overline{\eta}_{\bar{a}} F_{\bar{a}}(z_e))\,.}
\end{array}
\]
Note that~$\Psi(z_a,z_e)=-\Psi(z_e,z_a)$,~$\Psi^{\star}(z_a,z_e)=-\overline{\Psi^{\star}(z_e,z_a)}$, and all these functions are independent of the choices of~$\eta_a$ and~$\eta_e$. Moreover, it is easy to see that
\[
\matr{\Psi(z_e,z_a)}{\Psi^{\star}(z_e,z_a)}{\overline{\Psi^{\star}(z_e,z_a)}}{\overline{\Psi(z_e,z_a)}}= t_et_a\cdot
{\matr{\eta_e}{\eta_{\bar{e}}}{\overline{\eta}_e}{\overline{\eta}_{\bar{e}}}} \matr{\wh{\rK}^{-1}_{e,a}}{\wh{\rK}^{-1}_{e,\bar{a}}}{\wh{\rK}^{-1}_{\bar{e},a}}{\wh{\rK}^{-1}_{\bar{e},\bar{a}}}
{\matr{\eta_a}{\overline{\eta}_a}{\eta_{\bar{a}}}{\overline{\eta}_{\bar{a}}}},
\]
where $\wh{\rK}^{-1}_{e,a}\,=\,\lan\phi_e\phi_a\ran$ due to the definition of these formal correlation functions.

Therefore, in order to understand the scaling limit of the multi-point energy-density expectation~\eqref{eqn:energy-Pfaff}, it is enough to understand the scaling limit of the functions~$\Psi(z_e,z_a)$ and~$\Psi^{\star}(z_e,z_a)$ or, equivalently, the scaling limit of the complex-valued observables~$F_a(z_e)$ and~$F_{\overline{a}}(z_e)$. The latter satisfy the s-holomorphicity condition~\eqref{eqn:s-hol-def} everywhere in~$\Omega_\delta$ except near the ``source'' edge~$a$ or~$\overline{a}$ and the condition~\eqref{eqn:boundary-conditions} at the boundary. Therefore, the question amounts to the proof of convergence of solutions to such discrete boundary value problems as~$\delta\to 0$, including the careful analysis of their behavior near the ``source'' point and, in the more general setup, near the branching points~$u_1,\dots,u_m$, cf.~Remark~\ref{rem:formulas-via-K^-1}(ii).

\begin{remark}
\label{rem:towards-scaling-limits-single}
Let us emphasize that the combinatorial formulas discussed in this paper provide just a \emph{starting point} for the analysis of scaling limits of various correlation functions in discrete domains~$\Omega_\delta$\,. The boundary value problems for s-holomorphic functions satisfying boundary conditions~\eqref{eqn:boundary-conditions} are not easy to handle and one needs \emph{a lot of technical work} in order to prove the relevant convergence theorems for their solutions, even when considering the \emph{critical} Ising model on subgraphs of the square grid. The first breakthrough convergence results of this type were obtained by Smirnov in~\cite{smirnov-icm-2006,Smirnov-10} and more advanced methods were later developed in~\cite{chelkak-smirnov,hongler-smirnov,Chelkak-Izyurov,CHI,izyurov-free}. Away from criticality, a similar analysis does not look completely out of reach and some important algebraic tricks (notably, the definition of the discrete antiderivative~$\int\Im[(F(z))^2dz]$) are available in a fairly general setup, see~\cite{cimasoni-AIHP}. Nevertheless, even the near-critical (aka massive) model considered \emph{in bounded domains} has not been treated yet.
\end{remark}


\section{The surface case}
\label{sec:surface}

The aim of this section is to extend the results and methods of proof of Section~\ref{sec:planar} to arbitrary finite weighted graphs.
This requires the use of a geometrical tool known as a {\em spin structure\/}.
Therefore, we devote a first subsection to reviewing the main properties of spin structures on surfaces. We then use them to extend the Kac--Ward formula to graphs embedded in surfaces.
In a last subsection, we show how to use this result for the computation of spin correlations in this more general setting.

\setcounter{equation}{0}
\subsection{Spin structures, Kac--Ward matrices on surfaces, and quadratic forms}
\label{sub:homology}

Obviously, any finite graph can be embedded in a compact orientable surface. However, in order to define the associated Kac--Ward matrix, one needs to be able to
measure rotation angles along curves. For planar closed curves, there is one natural way to do so: one measures the rotation angle of the velocity vector of
the curve with respect to any constant vector field on the plane. For curves embedded in an arbitrary surface, there is no preferred way. However, there is a standard
geometrical tool for this, known as a {\em spin structure\/}.
We shall not recall its formal definition (see e.g.~\cite[p.55]{Atiyah}), but only
state without proof the properties that we shall need.

The first of these properties is that any spin structure on a compact orientable surface~$\Sigma$ can be given by a vector field on~$\SI$ with isolated zeroes of even index.
This already allows us to extend the definition of Kac--Ward matrices to this setting, as follows. Given a weighted graph~$(G,x)\subset\SI$ and a spin structure~$\lambda$
on~$\Sigma$, let us endow~$\SI$ with a Riemannian metric and fix a vector field~$X$ on~$\SI$ with isolated zeroes of even index in~$\SI\setminus G$ representing the
spin structure~$\lambda$. Finally, let us mark one point inside each edge of~$G$.

\begin{definition}
\label{def:KW2}
The Kac--Ward matrix associated to the weighted graph~$(G,x)$ embedded in~$\SI$ and to the spin structure~$\lambda$ is the~$|{\EE}(G)|\times|{\EE}(G)|$
matrix~$\KW_\lambda(G,x)=\rI-\rT_\lambda$, where~$\rI$ is the identity matrix and~$\rT_\lambda$ is defined by
\[
(\rT_\lambda)_{e,e'}=\begin{cases}
\exp\left(\frac{i}{2}\rw_\lambda(e,e')\right)(x_ex_{e'})^{1/2}& \text{if~$t(e)=o(e')$ but~$e'\neq \bar{e}$;} \\
0 & \text{otherwise,}
\end{cases}
\]
where~$\rw_\lambda(e,e')$ is the rotation angle of the velocity vector field along~$e$ followed by~$e'$ with respect to the vector field~$X$, from the marked
point in~$e$ to the marked point in~$e'$.
\end{definition}

Obviously, this matrix depends on the choice of the vector field representing~$\lambda$ and of the marked points in the edges. However, its determinant will turn out only
to depend on~$\lambda$ and on~$(G,x)\subset\SI$. The precise result is most conveniently stated using the terminology of homology and quadratic forms, that we now very
briefly recall. (We refer the interested reader to~\cite{Hat} for further details.)

Consider a graph~$G$ embedded in a compact connected orientable surface~$\SI$ of genus~$g$ in such a way that~$\SI\setminus G$ consists of a disjoint union of topological disks.
Let~$C_0$ (resp.~$C_1$,~$C_2$) denote the~$\Z_2$-vector space with basis the set of vertices (resp. edges, faces) of~$G\subset\SI$. Elements of~$C_k$ are called {\em~$k$-chains\/}.
Also, let~$\partial_2\colon C_2\to C_1$ and~$\partial_1\colon C_1\to C_0$ denote the {\em boundary operators\/} defined in the obvious way. Since~$\partial_1\circ\partial_2$ vanishes,
the space of {\em~$1$-cycles\/}~$\mathrm{ker}(\partial_1)$ contains the space~$\partial_2(C_2)$ of
{\em~$1$-boundaries\/}. The {\em first homology space\/}~$H_1(\SI;\Z_2):=\mathrm{ker}(\partial_1)/\partial_2(C_2)$ turns out not to depend on~$G$, but only on~$\SI$:
it has dimension~$2g$ if~$\SI$ is closed (i.e. compact without boundary), and dimension~$2g+b-1$ if~$\SI$ has~$b\ge 1$ boundary components.
Note that the intersection of curves defines a bilinear form on~$H_1(\SI;\Z_2)$ which is non-degenerate if~$\SI$ is closed; it will be denoted by~$(\alpha,\beta)\mapsto \alpha\cdot\beta$ as usual.
Finally, recall that the space~$H^1(\SI;\Z_2)=\Hom(H_1(\SI;\Z_2),\Z_2)$ can be understood as the set of (gauge equivalence classes of)~$\Z_2$-valued {\em flat connections\/};
these are maps~$\varphi\colon\EE(G)\to\Z_2$ such that~$\varphi(e)=\varphi(\overline{e})$ for each~oriented edge~$e$ and~$\sum_{e\in\partial f}\varphi(e)=0$ for each face~$f$
of~$G\subset\SI$.

This leads us to the statement of the second property of spin structures: the set~$\S(\SI)$ of spin structures on an oriented compact surface~$\SI$ is an affine space over~$H^1(\SI;\Z_2)$.
In other words, there is an action~$(\varphi,\lambda)\mapsto\varphi+\lambda$ of~$H^1(\SI;\Z_2)$ on~$\S(\SI)$ such that for any fixed~$\lambda$, the assignment~$\varphi\mapsto\varphi+\lambda$ defines a bijection from~$H^1(\SI;\Z_2)$ onto~$\S(\SI)$. This action is easy to understand at the level of vector fields, and therefore at the
level of Kac--Ward matrices: it is simply given by~$(\rT_{\varphi+\lambda})_{e,e'}=(-1)^{\varphi(e)}(\rT_\lambda)_{e,e'}$.

\begin{example}
\label{ex:domain}
As a first example, consider the case where~$\SI$ is an~$m$-punctured disk in the plane. A natural spin structure is given by any constant vector field,
and the corresponding Kac--Ward matrix is nothing but the classical one (see Section~\ref{sub:intro-KW-and-terminal}). On the other hand, this surface has genus zero and~$m+1$ boundary components, so it admits~$2^{m}$ different spin structures. They can be obtained from the first one as follows: draw a path
(transverse to the graph) from each of the punctures to the boundary of the disk, fix a subset of the punctures, and set~$\varphi(e)=1$ whenever~$e$
crosses the path corresponding to one of the chosen punctures (and set $\varphi(e)=0$ else). There are~$2^m$ choices of subsets of the punctures,
which correspond to the~$2^{m}$ different spin structures, and to the~$2^m$ different Kac--Ward matrices~$\KW_\lambda(G,x)$.
Note that, if the punctures are located at faces~$u_1,\dots,u_m$ of a planar graph~$G$, then we have
\begin{equation}
\label{eqn:KWl=KWu}
\KW_\lambda(G,x)~=~\rI_{[u_1,..,u_m]}\cdot \KW_{[u_1,\dots,u_m]}\,,
\end{equation}
where~$\KW_{[u_1,\dots,u_m]}$ are the matrices that appeared in Proposition~\ref{prop:spin1} (see Section~\ref{sub:intro-spin-correlations}).
\end{example}

\begin{example}
\label{ex:torus}
Another easy example is given by the torus~$\SI=\mathbb{T}^2$. Here again, it is possible to consider a constant vector field as a ``reference" spin structure.
Since~$H_1(\mathbb{T}^2;\Z_2)$ has dimension~$2$, there are~$4$ distinct spin structures on~$\mathbb{T}^2$. They can be obtained from the first one as follows: draw two closed
curves (transverse to~$G$) representing a basis of the homology, fix a subset of this basis, and set~$\varphi(e)=1$ whenever~$e$ crosses one of the chosen curves (and set $\varphi(e)=0$ else).
There are~$4$ choices of subsets of this basis, they correspond to the~$4$ different spin structures, and to the~$4$ different Kac--Ward matrices.
\end{example}

Let us now turn to quadratic forms.
Let~$H$ be a finite-dimensional~$\Z_2$-vector space endowed with {an \emph{alternating} (i.e., such that~$\alpha\cdot\alpha=0$ for all~$\alpha\in H$)} bilinear form~$(\alpha,\beta)\mapsto \alpha\cdot\beta$. A {\em quadratic form\/} on~$(H,\,\cdot\,)$ is a
map~$q\colon H\to\Z_2$ such that
\[
q(\alpha+\beta)=q(\alpha)+q(\beta)+\alpha\cdot\beta\ \ \text{for~all}\ \ \alpha,\beta\in H.
\]
Note that there are exactly~$|H|$ quadratic forms on~$(H,\,\cdot\,)$;
more precisely, the set of such forms is an affine space over~$\Hom(H,\Z_2)$.
This easily implies the equality
\begin{equation}
\label{eqn:q}
\frac{1}{|H|}\sum_{q}(-1)^{q(\alpha)}=\begin{cases}
1& \text{if~$\alpha=0$;} \\
0 & \text{otherwise,}
\end{cases}
\end{equation}
where the sum is over all quadratic forms on~$(H,\,\cdot\,)$.
Furthermore, if the {alternating bilinear form~$(\alpha,\beta)\mapsto\alpha\cdot\beta$} is non-degenerate, Arf showed~\cite{Arf} that the corresponding quadratic forms are classified by the invariant {$\Arf(\cdot)\in\Z_2$ defined by}
{
\begin{equation*}
(-1)^{\Arf(q)}=\frac{1}{\sqrt{|H|}}\sum_{\alpha\in H}(-1)^{q(\alpha)}\,,
\end{equation*}
which is now called the \emph{Arf invariant}} {(note that in this case the space~$H$ is necessarily even-dimensional)}.
We shall need a single property of this invariant (see e.g.~\cite[Lemma 2.10]{Loebl-Masbaum} for a proof), namely that it satisfies the equality
\begin{equation}
\label{equ:Arf}
\frac{1}{\sqrt{|H|}}\sum_{q}(-1)^{\Arf(q)+q(\alpha)}=1
\end{equation}
for any~$\alpha\in H$, where the sum is over all quadratic forms on~$(H,\,\cdot\,)$.

The relationship between spin structures and quadratic forms {on~$H_1(\SI;\Z_2)$} is given by the following classical result of Johnson~\cite{Joh}.
Consider a spin structure~$\lambda$ on~$\SI$ represented by a vector field~$X$ on~$\SI$ with zeroes of even index. Given a piecewise smooth closed curve~$C$ in~$\SI$
avoiding the zeroes of~$X$, let~$\wind_\lambda(C)\in 2\pi\Z$ denote the rotation angle of the velocity vector of~$C$ with respect to~$X$.
Then, given a homology class~$\alpha\in H_1(\SI;\Z_2)$ represented by {a disjoint union} of oriented simple closed curves~$C_j$, the
equality~$(-1)^{q_\lambda(\alpha)}=\prod_j(-\exp\left(\textstyle{\frac{i}{2}}\wind_\lambda(C_j)\right)$
gives a well-defined quadratic form on~$(H_1(\SI;\Z_2),\,\cdot\,)${, where $\cdot$ denotes the intersection form}. This implies in particular that, for any oriented closed curve~$C$ with~$\rt(C)$ transverse
self-intersection points,
\begin{equation}
\label{equ:rot}
-\exp\left(\textstyle{\frac{i}{2}}\wind_\lambda(C)\right)=(-1)^{q_\lambda(C)+\rt(C)}\,.
\end{equation}
Indeed, this can be checked by smoothing out the intersection points of~$C$ as in the proof of Lemma~\ref{lemma:Whitney} -- which is nothing but the~$g=0$ case of this equation.
Johnson's theorem asserts that the mapping~$\lambda\mapsto q_\lambda$ defines an~$H^1(\SI;\Z_2)$-equivariant bijection between the set~$\S(\SI)$ of spin structures on~$\SI$ and the
set of quadratic forms on~$(H_1(\SI;\Z_2),\cdot)$. Equation~\eqref{eqn:q} translates into the equality
\begin{equation}
\label{equ:q}
\frac{1}{|\S(\SI)|}\sum_{\lambda\in\S(\SI)}(-1)^{q_\lambda(\alpha)}=\begin{cases}
1& \text{if~$\alpha=0$;} \\
0 & \text{otherwise.}
\end{cases}
\end{equation}
Johnson's theorem also allows us to define the {\em Arf invariant\/} of a spin structure~$\lambda$ as the Arf invariant of the associated quadratic form~$q_\lambda$. In the case of a closed surface of genus~$g$, equation~\eqref{equ:Arf} then implies the equality
\begin{equation}
\label{eqn:Arf2}
\frac{1}{2^g}\sum_{\lambda\in\S(\SI)}(-1)^{\Arf(\lambda)+q_\lambda(\alpha)}=1
\end{equation}
for any~$\alpha\in H_1(\SI;\Z_2)$.

\subsection{The Kac--Ward formula on surfaces}
\label{sub:genKW}

We are finally ready to state and prove the main result of this section.
Note that the case of genus zero gives back the Kac--Ward formula (in an extended form actually, since we now allow edges that are not line segments).

\begin{theorem}
\label{thm:KW2}
Let~$(G,x)$ be a finite weighted graph embedded in an orientable compact surface~$\SI$. For any spin structure~$\lambda\in \S(\SI)$, we have the equality
\[
\det(\KW_\lambda(G,x))=\Big(\sum_{P\in\cE(G)}(-1)^{q_\lambda([P])} x(P)\Big)^2,
\]
where~$[P]\in H_1(\SI;\Z_2)$ denotes the homology class of~$P$.
\end{theorem}
\begin{proof}
First note that the set~$\mathcal{E}(G)$ endowed with the symmetric difference can be identified with the~$\Z_2$-vector space of~$1$-cycles in~$G$.
Therefore, for any fixed spin structure~$\lambda$ on~$\SI$, {the associated twisted partition function can be written as}
\[
\cZ_\lambda(G,x):=\sum_{P\in\mathcal{E}(G)}(-1)^{q_\lambda([P])} x(P)=\sum_{\alpha\in H_1(\SI;\Z_2)}(-1)^{q_\lambda(\alpha)}\sum_{P:[P]=\alpha}x(P)\,,
\]
where~$[P]\in H_1(\SI;\Z_2)$ denotes the homology class of the~$1$-cycle~$P$.
The (weighted) terminal graph~$(G^\rK,x^\rK)$ can be defined exactly as in the planar case, and the same arguments lead to the equality
\[
\cZ_\lambda(G,x)=\sum_{D\in\D(G^\rK)}(-1)^{q_\lambda([G\setminus D_G])+\rt(D)}\,x^\rK(D)\,.
\]
As in the planar case, let us consider the Hermitian matrix~$\rK_\lambda=\rJ\cdot \KW_\lambda(G,x)$. For each oriented edge~$e$ of~$G$, {fix a square root of the
direction of the velocity vector of~$e$ at the marked point inside~$e$, measured with respect to the vector field representing the spin structure~$\lambda$, and denote by~$\eta_e$ its complex conjugate multiplied by a global unimodular factor $\zeta$.} This allows
us to define the real skew-symmetric matrix~$\wh{\rK}_\lambda={i\rU^*\rK_\lambda\rU}$, where~$\rU$ is the diagonal matrix with coefficients~$\{\eta_e\}_{e\in V(G^\rK)}$. Comparing
the equation displayed above with the Pfaffian
\[
\Pf{\wh{\rK}_\lambda}=\sum_{D\in\D(G^\rK)}\veps_\lambda(D)x^\rK(D)\,,
\]
we are left with the proof of the equality
\[
\veps_\lambda(D)\veps_\lambda(D_0)=(-1)^{q_\lambda([G\setminus D_G])+\rt(D)}
\]
for any~$D\in\D(G^\rK)$, where~$D_0$ is the dimer configuration given by the set of long edges of the terminal graph~$G^\rK$.
To check this fact, one can use the exact same arguments as in the planar case replacing Lemma~\ref{lemma:Whitney} with its extension,
equation~\eqref{equ:rot}. This leads to
\[
\veps_\lambda(D)\veps_\lambda(D_0)=\prod_{j=1}^\ell(-\exp\left(\textstyle{\frac{i}{2}}\wind_\lambda(C_j)\right)=\prod_{j=1}^\ell(-1)^{q_\lambda(C_j)+\rt(C_j)}\,,
\]
where~$\bigsqcup_{j=1}^\ell C_j=D\btu D_0$\,. Since~$q_\lambda$ is a quadratic form, we get
\begin{align*}
\sum_{j=1}^\ell (q_\lambda(C_j)+\rt(C_j))& =q_\lambda([D\Delta D_0])+\sum_{1\le j<k\le \ell} C_j\cdot C_k + \sum_{j=1}^\ell\rt(C_j) \\ & =q_\lambda([D\Delta D_0])+\rt(D\Delta D_0)\,.
\end{align*}
The theorem now follows from the  equality~$[D\Delta D_0]=[G\setminus D_G]$ in~$H_1(\SI;\Z_2)$, together with the fact that~$\rt(D\Delta D_0)$ and~$\rt(D)$ coincide.
\end{proof}

The Kac--Ward formula for graphs embedded in surfaces, first derived in~\cite{Cim2}, is now an easy corollary.

\begin{corollary}
\label{cor:KW2}
Let~$(G,x)$ be a finite weighted graph embedded in an orientable closed surface~$\SI$ of genus~$g$. Then, the partition function~\eqref{eqn:Z-Ising-high} of the Ising model on~$G$ is equal to
\[
\cZ_{\opname{high}}(G,x)=\frac{1}{2^g}\sum_{\lambda\in \S(\SI)}(-1)^{\Arf(\lambda)}(\det\KW_\lambda(G,x))^{1/2},
\]
where~$\Arf(\lambda)\in\Z_2$ is the Arf invariant of the spin structure~$\lambda$, and~$(\det\KW_\lambda(G,x))^{1/2}$ denotes the square root with constant
coefficient equal to~$+1$.
\end{corollary}
\begin{proof}
By Theorem~\ref{thm:KW2} and the choice of the square root, we have
\begin{equation}
\label{eqn:sq}
(\det\KW_\lambda(G,x))^{1/2}=\sum_{P\in\cE(G)}(-1)^{q_\lambda([P])} x(P)~=\!\!\!\sum_{\alpha\in H_1(\SI;\Z_2)}(-1)^{q_\lambda(\alpha)}\sum_{P:[P]=\alpha}x(P)\,.
\end{equation}
The claim now follows from equation~\eqref{eqn:Arf2}:
\begin{align*}
\cZ_{\opname{high}}(G,x)~&=\sum_{\alpha\in H_1(\SI;\Z_2)}\sum_{P:[P]=\alpha}x(P) \\
&=\sum_{\alpha\in H_1(\SI;\Z_2)}\biggl[\Big(\,\frac{1}{2^g}\!\!\sum_{\lambda\in\S(\SI)}(-1)^{\Arf(\lambda)+q_\lambda(\alpha)}\Big)\sum_{P:[P]=\alpha}x(P)\biggr]\\
&=~\frac{1}{2^g}\sum_{\lambda\in \S(\SI)}\,\biggl[(-1)^{\Arf(\lambda)}\!\!\!\sum_{\alpha\in H_1(\SI;\Z_2)}(-1)^{q_\lambda(\alpha)}\sum_{P:[P]=\alpha}x(P)\biggr]\\
&=~\frac{1}{2^g}\sum_{\lambda\in \S(\SI)}(-1)^{\Arf(\lambda)}(\det\KW_\lambda(G,x))^{1/2}\,.\qedhere
\end{align*}
\end{proof}

\subsection{Spin correlations on surfaces}
\label{sub:spinsurface}

Similarly to Sections~\ref{sub:intro-spin-correlations} and~\ref{sub:spin}, below we prefer to work with the Ising model defined on the \emph{dual} graph~$G^*$. Recall that the partition function of this model is given by~\eqref{eqn:Z-Ising-low}, with the sum taken over the set $\cE_0(G)=\{P\in\cE(G):[P]=0\}$ of possible \emph{domain walls} configurations and not over the whole set~$\cE(G)$ as in Corollary~\ref{cor:KW2}. Also, recall that in case~$\SI$ has a boundary, we impose~`$+$' boundary conditions on \emph{all} of its components, and that the `free' boundary conditions on (some of) these components can be obtained just by setting the corresponding interaction parameters~$x_e=\exp[-2\beta J_{e^*}]$ to~$1$.

\begin{corollary}
\label{cor:low}
Let~$G$ be a finite graph embedded in an orientable compact surface~$\SI$, possibly with boundary. Then,
the partition function~~\eqref{eqn:Z-Ising-low} of the Ising model on the dual graph~$G^*$ is equal to
\[
\cZ_{\opname{low},\SI}(G,x)=\frac{1}{|\S(\SI)|}\sum_{\lambda\in \S(\SI)}(\det{\mathcal{K}\mathcal{W}}_\lambda(G,x))^{1/2}\,,
\]
where~$\det(\mathcal{K}\mathcal{W}_\lambda(G,x))^{1/2}$ denotes the square root with constant coefficient equal to~$+1$.
\end{corollary}
\begin{proof}
Using equations~\eqref{equ:q} and~\eqref{eqn:sq}, one easily gets
\begin{align*}
\cZ_{\opname{low},\SI}(G,x)~
&=\sum_{\alpha\in H_1(\SI;\Z_2)}\biggl[\Big(\frac{1}{|\S(\SI)|}\sum_{\lambda\in\S(\SI)}(-1)^{q_\lambda(\alpha)}\Big)\sum_{P:[P]=\alpha}x(P)\biggr]\\
&=\frac{1}{|\S(\SI)|}\sum_{\lambda\in \S(\SI)}\biggl[\,\sum_{\alpha\in H_1(\SI;\Z_2)}(-1)^{q_\lambda(\alpha)}\sum_{P:[P]=\alpha}x(P)\biggr]\\
&=\frac{1}{|\S(\SI)|}\sum_{\lambda\in \S(\SI)}(\det\mathcal{K}\mathcal{W}_\lambda(G,x))^{1/2}\,.\qedhere
\end{align*}
\end{proof}

\begin{example} Let~$\SI$ be an~$m$-punctured disk in the plane with the punctures located at faces~$u_1,\dots,u_m$ of a planar graph~$G$, as in Example~\ref{ex:domain}. Recall that spin structures~$\lambda\in\S(\SI)$ are in natural $1$--to--$1$ correspondence with subsets~$\mathcal{U}$ of the set $\{u_1,\dots,u_m\}$, and the corresponding Kac--Ward matrices~$\KW_\lambda$ are related to the matrices~$\KW_{[\mathcal{U}]}$ from Proposition~\ref{prop:spin1} by~(\ref{eqn:KWl=KWu}). Therefore, Corollary~\ref{cor:low} gives
\begin{align*}
\frac{\cZ_{\opname{low},\SI}(G,x)}{\cZ_{\opname{low}}(G,x)} ~ & = ~\frac{1}{|\S(\SI)|}\sum_{\lambda\in \S(\SI)}\frac{(\det{\mathcal{K}\mathcal{W}}_\lambda(G,x))^{1/2}}{(\det{\mathcal{K}\mathcal{W}}(G,x))^{1/2}} \\
& = ~ 2^{-m}\sum\nolimits_{\mathcal{U}\subset \{u_1,\dots,u_m\}}\mathbb{E}_{G^*}^+\biggl[\prod\nolimits_{u\in\mathcal{U}}\sigma_u\biggr]~ = ~ \mathbb{E}_{G^*}^+[\,2^{-m}(1+\sigma_{u_1})\dots (1+\sigma_{u_m})\,]\,.
\end{align*}
Note that this is consistent with the definition of~$\cZ_{\opname{low},\SI}(G,x)$, which is the partition function of the Ising model on~$G^*$ with~`$+$' boundary conditions at all the faces $u_1,\dots,u_m$ and~$u_{\opname{out}}$.
\end{example}

Let us now show how Proposition~\ref{prop:spin1} extends to the case of a finite graph~$G$ embedded in an orientable compact surface-with-boundary~$\Sigma$.
Fix~$m$ faces~$u_1,\dots,u_m$ of~$G\subset\Sigma$ and some collection of edge-disjoint paths~$\varkappa=\varkappa_{[u_1,..,u_m]}$ in~$G^*$ linking these
faces to the boundary of~$\Sigma$ or to each other. As in the planar case, this induces a diagonal matrix~$\mathrm{I}_{[u_1,\dots,u_m]}$
and, for any spin structure~$\lambda\in\mathcal{S}(\Sigma)$, a modified Kac--Ward matrix
\[
\mathcal{K}\mathcal{W}_{\lambda+[u_1,\dots,u_m]}=\mathrm{I}_{[u_1,\dots,u_m]}-\mathrm{T}_\lambda\,.
\]
Note that~$\mathcal{K}\mathcal{W}_{\lambda+[u_1,\dots,u_m]}$ can be {interpreted} as the Kac--Ward matrix associated to a spin structure
on the punctured surface~$\Sigma\setminus\{u_1,\dots,u_m\}$. However, this spin structure is not canonically associated to~$\lambda$ and
$u_1,\dots,u_m$, as it depends on the choice of the (homology class of the) paths~$\varkappa$. More precisely, a straightforward extension of the proof of Theorem~\ref{thm:KW2} leads to the equality
\begin{equation}
\label{equ:spin2}
\det\mathcal{K}\mathcal{W}_{\lambda+[u_1,\dots,u_m]}=\Big(\sum_{P\in\mathcal{E}(G)}(-1)^{q_\lambda([P])}(-1)^{\varkappa\cdot P}x(P)\Big)^2\,,
\end{equation}
which implies the following result.

\begin{proposition}
\label{prop:spin2}
Let~$G$ be a finite graph embedded in an orientable compact surface-with-boundary~$\Sigma$.
The correlation of spins at faces~$u_1,\dots,u_m$ with~`$+$' boundary conditions on all the boundary components of~$\Sigma$ is given by
\[
\mathbb{E}^+_{G^*}[\sigma_{u_1}\dots\sigma_{u_m}]=\frac{\sum_{\lambda\in \S(\SI)} (\det\KW_{\lambda+[u_1,\dots,u_m]}(G,x))^{1/2}}{\sum_{\lambda\in \S(\SI)}(\det\KW_\lambda(G,x))^{1/2}},
\]
where~$x_e$ stands for~$\exp(-2\beta J_{e^*})$.
\end{proposition}
\begin{proof}
The bijection between spin configurations on~$G^*$ and domain walls in~$G$ gives
\[
\mathbb{E}^+_{G^*}[\sigma_{u_1}\dots\sigma_{u_m}] ~=~ \left[{\cZ_{\opname{low},\SI}(G,x)}\right]^{-1}\cdot{\sum_{P:[P]=0}(-1)^{\varkappa\cdot P}x(P)}\,.
\]
By equation~\eqref{equ:q}, we get
\begin{align*}
\sum_{P:[P]=0}(-1)^{\varkappa\cdot P}x(P)~ & = \sum_{P\in\mathcal{E}(G)}\biggl[\Big(\frac{1}{|\S(\SI)|}\sum_{\lambda\in \S(\SI)}(-1)^{q_\lambda([P])}\Big)(-1)^{\varkappa\cdot P}x(P)\biggr] \\
& = \frac{1}{|\S(\SI)|}\sum_{\lambda\in \S(\SI)}\biggl[\,\sum_{P\in\mathcal{E}(G)}(-1)^{q_{\lambda}([P])}(-1)^{\varkappa\cdot P}x(P)\biggr]\,.
\end{align*}
The result now follows from equation~\eqref{equ:spin2} and Corollary~\ref{cor:low}.
\end{proof}

Let us conclude this section by mentioning that Remark~\ref{rem:formulas-via-K^-1} readily extends to the surface case, as well as Theorem~\ref{thm:multipoint}. More precisely, the latter result takes the following form: for any spin structure~$\lambda$, the Pfaffian of a submatrix of the corresponding matrix~$\wh{\rK}^{-1}_\lambda$ admits a combinatorial expansion similar to Theorem~\ref{thm:multipoint}, provided definition~\eqref{eqn:tau-def} of the sign~$\tau(P)$ of a configuration~$P\in\cC(e_1,\dots,e_{2n})$ smoothed into~$C\sqcup\gamma$ takes the additional factor~$(-1)^{q_\lambda(C)}$. Also, one can easily extend to the surface case the dimer techniques discussed in Section~\ref{sub:planar-dimers} and the formalism of Grassmann variables discussed in Section~\ref{sub:grassmann}. On the other hand, one should be more careful when using the disorder insertions discussed in Section~\ref{sub:disorders} for graphs embedded in surfaces, and, especially, when matching the two formalisms in the style of Section~\ref{sub:equivalence}. The subtle point is that the formal correlation functions~$\langle\mu_{v_1}\dots\mu_{v_{2n}}\sigma_{u_1}\dots\sigma_{u_m}\rangle^{[\varkappa]}$ now \emph{depend} on the homology class~$[\varkappa]$ of the paths~$\varkappa=\varkappa^{[v_1,..,v_{2n}]}$ chosen to define the set of configurations~$C^{[\varkappa]}(v_1,\dots,v_{2n}):=\{P:P\btu \varkappa\in\cE_0(G)\}$, and one should sum over~$[\varkappa]$ to recover the spin-disorder duality. Further, an appropriate analog of Lemma~\ref{lemma:mu-sigma} for the fixed homology class~$[\varkappa]$ involves a summation over~$\S(\SI)$ similar to that in the right-hand side of the equations of Corollary~\ref{cor:low} and Proposition~\ref{prop:spin2}, and one should keep track of additional signs (coming from the non-canonical identifications of the double-covers~$G^{[\varkappa]}$) when summing over~$[\varkappa]$. These technicalities become even more involved when dealing with general formal correlations considered in Proposition~\ref{prop:equivalence} though, in principle, all the details can be fixed.


\setcounter{equation}{0}
\section{The double-Ising model}
\label{sec:double}

In this section, we assume the weighted graph  $(G,x)$ to be embedded in the plane and consider the double-Ising model defined on the faces of the modified graph~$\widetilde{G}$, see Section~\ref{sub:intro-dblI} for the definitions and notation. Recall that the partition function of this model is given by
\begin{equation}\label{eqn:dblI-domain-walls}
\cZ_{\opname{dbl-I}}(G,x)~=~\sum\nolimits_{P,P'\in\cE(\widetilde{G})\,:\, (P\btu P')\cap E_\partial(G)=\emptyset} x(P)x(P')\,,
\end{equation}
where~$E_\partial(G)\subset E(G)$ is the set of boundary edges of~$G$, and the modified Kac--Ward matrix~$\widetilde{\rK}$ is defined as
\begin{equation}
\label{eqn:tilde-K-entries}
\widetilde{\rK}_{e,e'}=\rK_{e,e'}+ \begin{cases}ix_e &\text{if~$e'=e$ is an inward oriented boundary edge;} \\ 0 & \text{otherwise.} \end{cases}
\end{equation}
Below we describe how the combinatorial methods developed for the study of the (single-) Ising model on faces of~$G$ may be extended and used to study the double-Ising model. We give the proof of Theorem~\ref{thm:KW3} and Proposition~\ref{prop:spin4} in Section~\ref{sub:dblI-partition-fct}. In Section~\ref{sub:dblI-Dobrushin}, we prove an analog of Theorem~\ref{thm:KW3} for Dobrushin boundary conditions~\eqref{eqn:dblI-Dobrushin-bc}. The last Section~\ref{sub:dblI-s-hol-fct} is devoted to a discussion of s-holomorphic observables arising in the double-Ising model context and discrete boundary value problems for them.

\subsection{Proofs of Theorem~\ref{thm:KW3} and Proposition~\ref{prop:spin4}}\label{sub:dblI-partition-fct}

The aim of this section is to give a proof of Theorem~\ref{thm:KW3} and Proposition~\ref{prop:spin4}. Let us begin by recalling that the Kac--Ward formula~\eqref{eqn:KWformula} and Theorem~\ref{thm:KW1} for the single Ising model can be rewritten as
\begin{equation}
\label{eqn:det-K=Ising^2}
(\cZ_{\opname{Ising}}(G,x))^2 = ~(\Pf{\wh\rK})^2= ~\det\wh\rK ~ = ~ (-1)^{|E(G)|}\det\rK\,,
\end{equation}
where the Hermitian matrix~$\rK$ is given by~\eqref{eqn:K-entries}. The proof of this theorem discussed in Section~\ref{sec:planar} uses the combinatorial expansion of the \emph{Pfaffian} of the anti-symmetric matrix~$\wh\rK$. Alternatively, one could expand the \emph{determinant} of~$\rK$ for this purpose. The resulting proof of Theorem~\ref{thm:KW1} is slightly more cumbersome as we need to compare sums over more complicated configurations (double-dimers on the terminal graph, instead of {single dimers}). However, this approach does have an advantage: it generalizes to a proof of Theorem~\ref{thm:KW3}. Having this goal in mind, we start with a direct combinatorial expansion of the determinant of~$\rK$.

\begin{proof}[Second proof of Theorem~\ref{thm:KW1}] Squaring the right-hand side of the identity
\[
\cZ_{\opname{Ising}}(G,x)~=\sum_{D\in\cD(G^\rK)}(-1)^{\rt(D)}x^\rK\!(D)
\]
provided by Lemma~\ref{lemma:terminal} leads to a sum over pairs of dimer configurations on the terminal graph~$G^\rK$.
Any such pair defines a union~$C$ of vertex-disjoint unoriented loops of even length covering all the vertices of~$G^\rK$, some of length two, called
{\em double-edges\/}, and some of greater length, called {\em cycles\/}. We shall write~$\Gamma_{\opname{even}}(G^\rK)$ for the set of such configurations of edges of~$G^\rK$. Given a configuration $C\in\Gamma_{\opname{even}}(G^\rK)$, let~$C_0$ be the collection of its double-edges,~$C_1,\dots,C_\ell$ denote the (unoriented) cycles, and let
\[
x^\rK(C):=\prod_{e\in C_0} (x^\rK_e)^2\cdot \prod_{j=1}^\ell x^\rK(C_j)\,.
\]
For each of these~$\ell=\ell(C)$ cycles there are two ways to split it into two sets of dimers, and therefore
\begin{align}
\label{eqn:Z^2-expansion}
(\cZ_{\opname{Ising}}(G,x))^2~&= \sum_{D,D'\in\D(G^\rK)}(-1)^{\rt(D)+\rt(D')}x^\rK(D)x^\rK(D') \notag\\
	&= \sum_{C=D\cup D'\in\Gamma_{\opname{even}}(G^\rK)}(-1)^{\rt(C)+D\cdot D'}2^{\ell(C)}x^\rK(C)\,.
\end{align}
Note that we slightly abuse the notation using the intersection number~$D\cdot D'$ in the sum as this number depends on the particular way how~$C$ is split into~$D$ and~$D'$. Nevertheless, we shall see below that the parity of this intersection number depends on~$C$ only.

On the other hand, expanding the determinant of~$\rK$ leads to
\begin{equation}
\label{eqn:K-det-expansion}
\det\rK ~= ~\sum\nolimits_{\overrightarrow{C}\in\overrightarrow{\Gamma}(G^\rK)} \tau(\overrightarrow{C})x^\rK(C)\,,
\end{equation}
where~$\overrightarrow{\Gamma}(G^\rK)$ is the set of all vertex-disjoint unions~$\overrightarrow{C}= C_0\sqcup \bigsqcup_{j=1}^\ell\overrightarrow{C_j}$ of double-edges and  {\em oriented\/} cycles of {\em arbitrary\/} length covering all the vertices of~$G^\rK$, the weight~$x^\rK(C)$ does not depend on the orientation of these cycles and
\[
\tau(\overrightarrow{C})=(-1)^{|C_0|}\cdot \prod_{j=1}^\ell \tau(\overrightarrow{C_j})
\]
is some sign, that we now determine.

\smallskip

\noindent{\em Claim A.} Let~$\overleftarrow{C_j}$ denote the cycle~$\overrightarrow{C_j}$ with the orientation reversed. If the length~$|C_j|$ of {this cycle} is odd, then~$\tau(\overleftarrow{C_j})=-\tau(\overrightarrow{C_j})$. If~$|C_j|$ is even, then
\begin{equation}
\label{eqn:tau-def-section-5}
\tau(\overleftarrow{C_j})=\tau(\overrightarrow{C_j})=(-1)^{|C_j|/2}\,(-1)^{\rt(C_j)+v_-(\overrightarrow{C_j})}\,,
\end{equation}
where~$v_-(\overrightarrow{C_j})$ is the number of vertices of~$G^\rK$ visited by~$\overrightarrow{C_j}$ in such a way that both adjacent edges are short and this oriented cycle makes a clockwise turn at this vertex.

\smallskip

Before giving a proof of Claim~A, note that it has the following consequence: {the expansion~\eqref{eqn:K-det-expansion} for} the determinant of~$\rK$ {simplifies} to a sum over elements of~$\Gamma_{\opname{even}}(G^\rK)$. After that, the {only fact remaining to be proved in order to obtain~\eqref{eqn:det-K=Ising^2}} is the {equality} of the signs in the two expansions~\eqref{eqn:Z^2-expansion} and~\eqref{eqn:K-det-expansion}, which is done in the Claim~B below.

\smallskip

\noindent {\em Proof of Claim~A.} By definition of the matrix~$\rK$, {and expanding its determinant as in~\eqref{eqn:K-det-expansion}}, we have
\[
\tau(\overrightarrow{C_j})=(-1)^{|C_j|+1}\,(-1)^{s(C_j)}\,\omega(\overrightarrow{C_j})\,,
\]
where~$s(C_j)$ denotes the number of short edges in~$C_j$ and~$\omega(\overrightarrow{C_j})$ is the product of the
coefficients~$\exp[\frac{i}{2}\rw(\overline{e},e')]$ along the oriented cycle~$\overrightarrow{C_j}$. Computing the total rotation angle of (the velocity vector of) this cycle leads to
\[
\exp\left[{\textstyle\frac{i}{2}}\wind(\overrightarrow{C_j})\right] = \omega(\overrightarrow{C_j})\cdot i^{v_+(\overrightarrow{C_j})}(-i)^{v_-(\overrightarrow{C_j})}\,,
\]
with~$v_+(\overrightarrow{C_j})$ defined as~$v_-(\overrightarrow{C_j})$ but counting counterclockwise turns instead of clockwise ones, see Fig.~\ref{Fig:V+-}. Therefore, Lemma~\ref{lemma:Whitney} and the equality~$2s(C_j)=|C_j|+v_-(\overrightarrow{C_j})+v_+(\overrightarrow{C_j})$ imply
\begin{align*}
\tau(\overrightarrow{C_j}) & =(-1)^{\rt(C_j)+|C_j|+s(C_j)}\cdot i^{v_-(\overrightarrow{C_j})-v_+(\overrightarrow{C_j})}\\
&=(-1)^{\rt(C_j)+v_-(\overrightarrow{C_j})}\cdot (-i)^{|C_j|}.
\end{align*}
The claim follows easily since the two numbers~$v_-(\overrightarrow{C_j})$ and~$v_-(\overleftarrow{C_j})=v_+(\overrightarrow{C_j})$ have the opposite parity if~$|C_j|$ is odd and
the same parity if~$|C_j|$ is even.

\smallskip

Note that we have $|C_0|+\sum_{j=1}^\ell|C_j|/2=|E(G)|$ for any~$C\in \Gamma_{\opname{even}}(G^\rK)$. Thus, to deduce formula~\eqref{eqn:det-K=Ising^2} from expansions~\eqref{eqn:Z^2-expansion} and~\eqref{eqn:K-det-expansion}, it is enough to prove the following claim.

\smallskip

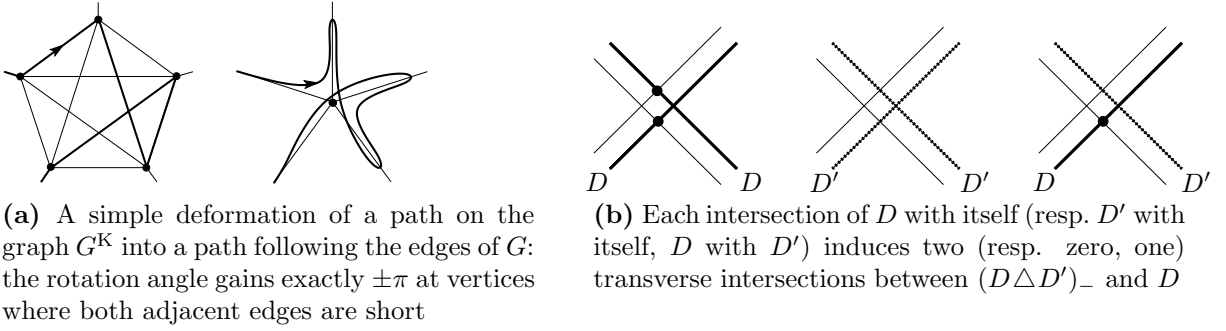
\begin{figure}
\captionsetup[subfigure]{position=below,justification=justified,singlelinecheck=false,labelfont=bf}
\begin{subfigure}[t]{.45\textwidth}
\psfig{file=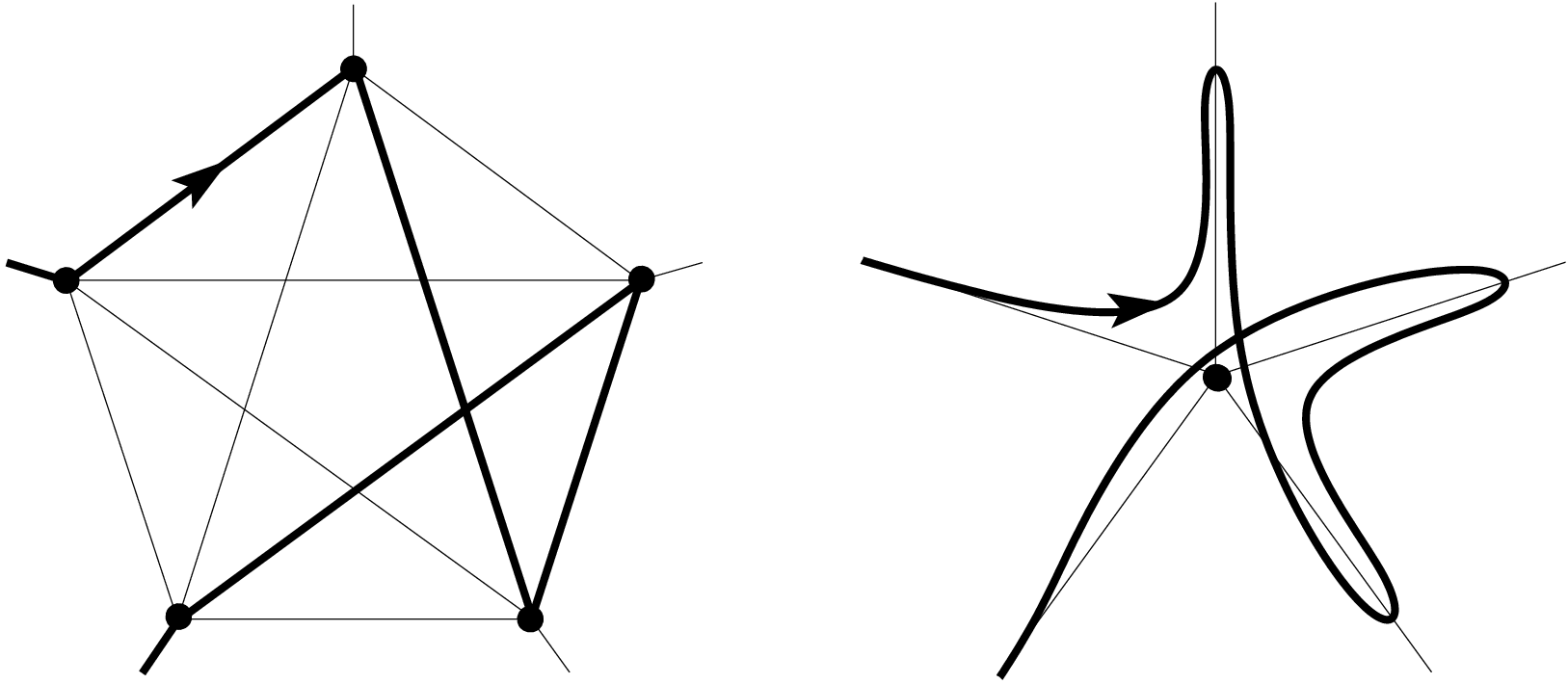,width=.8\linewidth}
\caption{{A simple deformation of a path on the graph~$G^\rK$ into a path following the edges of~$G$:
the rotation angle gains exactly~$\pm \pi$ at vertices where both adjacent edges are short}}
\label{Fig:V+-}
\end{subfigure}
\hfill
\begin{subfigure}[t]{.5\textwidth}
\labellist\small\hair 2.5pt
\pinlabel {$D$} at 6 6
\pinlabel {$D$} at 280 6
\pinlabel {$D$} at 804 6
\pinlabel {$D'$} at 412 6
\pinlabel {$D'$} at 680 6
\pinlabel {$D'$} at 1075 6
\endlabellist
\psfig{file=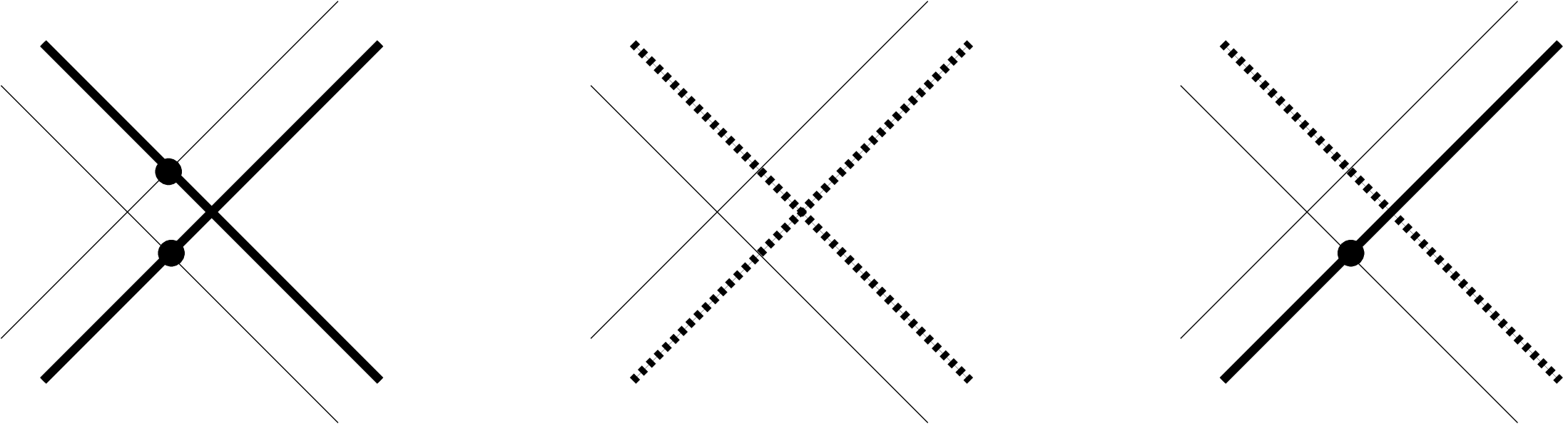,width=\linewidth}
\caption{Each intersection of~$D$ with itself (resp.~$D'$ with itself,~$D$ with~$D'$) induces two (resp. zero, one) transverse intersections between~$(D\btu D')_-$ and~$D$}
\label{Fig:DD'}
\end{subfigure}
\caption{Topological arguments used in the proofs of Claim A and Claim B}
\end{figure}

\noindent{\em Claim B.} For any~$D,D'\in\D(G^\rK)$, the number~$D\cdot D'$ has the same parity as~$v_-(D\btu D')$ defined as the sum of $v_-(\overrightarrow{C})$ over all components $C$ of~$D\btu D'$ oriented arbitrarily.

\smallskip

\noindent{\em Proof of Claim B.}
Let us denote by~$(D\btu D')_-$ the collection of disjoint, arbitrarily oriented, loops~$D\btu D'$ pushed slightly to their left, so that they intersect the edges of~$G^\rK$ transversally. Further, if~$D_0$ denotes the set of all long edges of~$G^\rK$, then $D\btu D_0$ is also a collection of loops, which gives
\[
(D\btu D')_-\cdot(D\btu D_0) = 0 \mod 2\,.
\]
Now observe that the number of intersections of~$(D\btu D')_-$ with short edges of~$D$ has the same parity as~$D\cdot D'$, see Fig.~\ref{Fig:DD'}. At the same time, the number of intersections of~$(D\btu D')_-$ with (long) edges of~$D_0\setminus D$ is exactly~$v_-(D\btu D')$, {since such intersections happen only in vicinities of the vertices, where both adjacent edges of
$D\btu D'$ are short and it makes a clockwise turn so that~$(D\btu D')_-$ is pushed ``towards'' the long dimer from~$D_0\setminus D$.} Thus the claim follows.
\end{proof}

We now move on to the proof of Theorem~\ref{thm:KW3}. Let~$\EE_\partial(G)$ denote the set of \emph{inward} oriented boundary edges.
Recall that we denote by~$(G^\diamondsuit,x)$ the weighted graph obtained from~$(G,x)$ by adding a vertex~$z_e$ in the middle of each edge~$e$ of~$G$, and by assigning the weight~$x_e^{1/2}$ to both resulting edges of~$G^\diamondsuit$. For a subset~$\rE=\{e_1,\dots,e_{2n}\}\subset\EE_\partial (G)$, we denote by~$\mathcal{C}(\rE)$ the set of subgraphs~$P$ of~$G^\diamondsuit$ that contain the edges~$(z_{e_k},t(e_k))$, and such that each vertex of~$G^\diamondsuit$ different from~$z_{e_1},\dots,z_{e_{2n}}$ has an even degree in~$P$, see Section~\ref{sub:intro-Pfaffian}. Note that
~$\cC(\rE)$ is empty if~$|\rE|$ is odd and~\eqref{eqn:dblI-domain-walls} can be written as
\begin{equation}
\label{eqn:Z-dblI-as-sum-over-E}
\cZ_{\opname{dbl-I}}(G,x)=\sum_{\rE\subset \EE_\partial(G)} x(\rE)\biggl[\,\sum_{P\in \cC(\rE)} x(P)\biggr]^2
\end{equation}
since, by definition, each boundary edge~$e\in\rE$ contributes only~$x_e^{1/2}$ to the weight~$x(P)$.

\begin{proof}[Proof of Theorem~\ref{thm:KW3}]
For a subset of boundary edges~$\rE\subset \EE_\partial(G)\subset \EE(G)\cong V(G^\rK)$, let~$\rK^\rE$ denote the matrix~$\rK$ with all the rows and the columns that are indexed by~$\rE$ removed. It immediately follows from definition~\eqref{eqn:tilde-K-entries} of the matrix~$\widetilde{\rK}$ that
\[
(-1)^{|E(G)|}\det\widetilde{\rK} ~= ~(-1)^{|E(G)|}\!\!\sum\nolimits_{\rE\subset \EE_\partial(G)} i^{|\rE|}x(\rE)\det\rK^\rE,
\]
thus we need to show that this expansion coincides with~\eqref{eqn:Z-dblI-as-sum-over-E}. Note that the case~$\rE=\emptyset$ was already treated in the second proof of Theorem~\ref{thm:KW1} given above, which we now generalize.

As in the proof of Theorem~\ref{thm:multipoint}, let~$G^\rK_\rE$ be the graph obtained by removing from~$G^\rK$ all the boundary (univalent) vertices corresponding to~$\rE$. A straightforward generalization of Lemma~\ref{lemma:terminal} gives
\begin{align*}
\biggl[\,\sum_{P\in \cC(\rE)} x(P)\biggr]^2 ~& = ~\biggl[\,\sum_{D\in \cD(G^\rK_\rE)}(-1)^{\rt(D)}x^\rK(D)\biggr]^2 \\
& = \sum_{C=D\cup D'\in\Gamma_{\opname{even}}(G^\rK_\rE)}(-1)^{\rt(C)+D\cdot D'}2^{\ell(C)}x^\rK(C)\,,
\end{align*}
where we use the same notation as in the second proof of Theorem~\ref{thm:KW1}. On the other hand, similarly to~\eqref{eqn:K-det-expansion} we have
\begin{equation}
\label{eqn:KE-det-expansion}
\det\rK^\rE ~= ~\sum\nolimits_{\overrightarrow{C}\in\overrightarrow{\Gamma}(G^\rK_\rE)} \tau(\overrightarrow{C})x^\rK(C)\,,
\end{equation}
with the signs~$\tau(\overrightarrow{C})$ given by~\eqref{eqn:tau-def-section-5}. Using this expansion and Claim~A, it is easy to see that~$\det\rK^\rE$ vanishes unless~$|\rE|$ is even: when~$|\rE|$ is odd, each configuration~$C\in\Gamma(G^\rK_\rE)$ contains an odd length cycle, whose two orientations yield a cancellation. Theorem~\ref{thm:KW3} now follows from the equality~$|V(G^\rK_\rE)|=2|E(G)|-|\rE|$ and a proper generalization of Claim~B from the second proof of Theorem~\ref{thm:KW1}, which we formulate below for completeness.

\smallskip

\noindent {\em Claim B'.} For any~$D,D'\in\D(G^\rK_\rE)$, the number~$D\cdot D'$ has the same parity as~$v_-(D\btu D')$ defined as the sum of $v_-(\overrightarrow{\gamma})$ over all components $\gamma$ of~$D\btu D'$ oriented arbitrarily.

\smallskip

\noindent{\em Proof of Claim B'.} The proof repeats the proof of Claim~B given above. The only difference is that~$D\btu D_0$ is \emph{not} a collection of closed loops anymore: if the set~$\rE$ is non-empty,~$D\btu D_0$ also contains $|\rE|/2$ paths linking the boundary edges~$e\in\rE$. Nevertheless, the equality
\[
(D\btu D')_-\cdot (D\btu D_0)=0\mod 2
\]
remains correct and the claim follows due to the same arguments.
\end{proof}

\begin{remark} It is worth noting that one can use Theorem~\ref{thm:multipoint} in the proof given above to handle the case~$|\rE|$ even. Indeed, if~$\rE$ is a collection of \emph{boundary} edges, then it is not hard to check that the sign~$\tau(P)$ defined by~\eqref{eqn:tau-def} is independent of~$P\in\cC(\rE)$, so the equality
\[
\Big[\,\sum\nolimits_{P\in \cC(\rE)} x(P)\,\Big]^2= ~ (\,\Pf{\wh{\rK}^\rE}\,)^2~=~\det \wh{\rK}^\rE ~=~ (-1)^{|V(G^\rK_\rE)|/2}\det \rK^\rE
\]
follows easily. Nevertheless, one still needs some additional arguments in the spirit of Claim~A to see that~$\det\rK^\rE$ vanishes if~$|\rE|$ is odd.
\end{remark}

We now prove Proposition~\ref{prop:spin4}.

\begin{proof}[Proof of Proposition~\ref{prop:spin4}] It easily follows from the domain walls representation~\eqref{eqn:dblI-domain-walls} that
\[
\mathbb{E}_{\opname{dbl-I}}^+ [\widetilde{\sigma}_{u_1}\dots\widetilde{\sigma}_{u_m}] = \frac{\sum\nolimits_{P,P'\in\cE(\widetilde{G})\,:\, (P\btu P')\cap E_\partial(G)=\emptyset}\;(-1)^{\varkappa\cdot P}\!x(P)\,(-1)^{\varkappa\cdot P'}\!x(P')}{\cZ_{\opname{dbl-I}}(G,x)}\,.
\]
Similarly to Proposition~\ref{prop:spin1}, repeating the proof of Theorem~\ref{thm:KW3} with additional signs~$(-1)^{\varkappa\cdot P}$ one concludes that the numerator is equal to~$(-1)^{|E(G)|}\det\widetilde{\rK}_{[u_1,..,u_m]}$\,.
\end{proof}

\subsection{Dobrushin boundary conditions}\label{sub:dblI-Dobrushin}

In this section, we prove a version of Theorem~\ref{thm:KW3} for the double-Ising model with Dobrushin boundary conditions~\eqref{eqn:dblI-Dobrushin-bc}. For this purpose, let us introduce a matrix
\[
\widetilde{\rK}^{[a,b]}= \big[\,\widetilde{\rK}^{[a,b]}_{e,e'}\,\big]_{e\ne a, e'\ne b}
\]
(for definiteness with signs of determinants, below we always assume that the column of~$\widetilde{\rK}^{[a,b]}$ labeled by~$a$ corresponds to the row labeled by~$b$) with the entries
\begin{equation}
\label{eqn:tilde-Kab-entries}
\widetilde{\rK}^{[a,b]}_{e,e'}:=\rK_{e,e'}+
\begin{cases}ix_e &\text{if~$e'=e$ is an inward oriented boundary edge on~$(ab)$;} \\
-ix_e &\text{if~$e'=e$ is an inward oriented boundary edge on~$(ba)$;} \\
0 & \text{otherwise.} \end{cases}
\end{equation}

\begin{theorem} \label{thm:KW4}
The partition function~$\cZ_{\opname{dbl-I}}^{[a,b]}(G,x)$ of the double-Ising model with Dobrushin boundary conditions~\eqref{eqn:dblI-Dobrushin-bc} is given by
\[
{\textstyle\frac{1}{2}}(x_ax_b)^{-1/2}\cdot\cZ_{\opname{dbl-I}}^{[a,b]}(G,x)~=~
\overline{\rw_{b,\overline{a}}}\,\cdot(-1)^{|E(G)|-1}\det\widetilde{\rK}^{[a,b]},
\]
where the prefactor~${\rw_{b,\overline{a}}}$ is defined as follows:
${\rw_{b,\overline{a}}}:=\exp[{\textstyle\frac{i}{2}}\wind(\gamma_{b,\overline{a}})]$
for any non-self-intersecting path~$\gamma_{b,\overline{a}}$ running from~$b$ to~$\overline{a}$ along edges of~$G$ (note that the total rotation angle~$\wind(\gamma_{b,\overline{a}})$ does not depend on the choice of this path).
\end{theorem}

\begin{proof} The proof goes along the same lines as the proof of Theorem~\ref{thm:KW3}.
Using domain walls representations of spin configurations~$\sigma,\sigma'$ satisfying~\eqref{eqn:dblI-Dobrushin-bc}, one easily obtains
\begin{align*}
{\textstyle\frac{1}{2}}(x_ax_b)^{-1/2}\cZ_{\opname{dbl-I}}^{[a,b]}(G,x)~&=~\sum\nolimits_{\rE\subset \EE_\partial(G)\setminus\{a,b\}} \biggl[\,x(\rE)\sum\nolimits_{P\in \cC(\rE)} x(P)\sum\nolimits_{P'\in \cC(\{a,b\}\cup\rE)} x(P')\biggr] \\
&+~\sum\nolimits_{\rE\subset \EE_\partial(G)\setminus\{a,b\}}\biggl[\,x(\rE)\sum\nolimits_{P\in \cC(\{a\}\cup\rE)} x(P)\sum\nolimits_{P'\in \cC(\{b\}\cup\rE)} x(P')\biggr]\,,
\end{align*}
where even subsets~$\rE$ contribute only to the first sum while odd ones only to the second. Passing from~$P,P'$ to dimer configurations~$D,D'$ on the graphs~$G^\rK_\rE$,~$G^\rK_{\rE\cup\{a,b\}}$ or~$G^\rK_{\rE\cup\{a\}}$,~$G^\rK_{\rE\cup\{b\}}$
and considering their union exactly as in the proof of Theorem~\ref{thm:KW3}, one gets
\[
{\textstyle\frac{1}{2}}(x_ax_b)^{-1/2}\cZ_{\opname{dbl-I}}^{[a,b]}(G,x)~=~\sum\nolimits_{\rE\subset \EE_\partial(G)\setminus\{a,b\}} \sum\nolimits_{C=D\cup D'\in\Gamma^{[a,b]}_{\opname{even}}(G^\rK_\rE)}(-1)^{\rt(C)+D\cdot D'}2^{\ell(C)}x^\rK(C),
\]
where~$\Gamma^{[a,b]}_{\opname{even}}(G^\rK_\rE)$ denotes the set of all covers of the graph~$V(G^\rK_\rE)$ by a vertex-disjoint union of a collection of double-edges~$C_0$, even length cycles~$C_1,\dots,C_\ell$ and a \emph{path~$\gamma$} linking~$a$ and~$b$.

At the same time, the following expansion holds:
\[
\det\widetilde{\rK}^{[a,b]} = \sum\nolimits_{\rE\subset \EE_\partial(G)\setminus\{a,b\}} i^{|\rE\cap (ab)|}(-i)^{|\rE\cap (ba)|}x(\rE) \det\rK^{[\{a\}\cup\rE,\{b\}\cup\rE]},
\]
where~$\rK^{[\{a\}\cup\rE,\{b\}\cup\rE]}$ denotes the matrix~$\rK$ with all the rows indexed by~$\{a\}\cup\rE$ and the columns indexed by~$\{b\}\cup\rE$ removed. A straightforward expansion of the determinant of this matrix similar to~\eqref{eqn:K-det-expansion} and~\eqref{eqn:KE-det-expansion}, together with Claim~A from the second proof of Theorem~\ref{thm:KW1}, leads to
\begin{equation}
\label{eqn:det-Kab-E-expansion}
\det\rK^{[\{a\}\cup\rE,\{b\}\cup\rE]} ~= ~\sum\nolimits_{C\in\Gamma^{[a,b]}_{\opname{even}}(G^\rK_\rE)} \tau(C)2^{\ell(C)}x^\rK(C)\,,
\end{equation}
where
\[
\tau(C)=(-1)^{|C_0|}\cdot\prod\nolimits_{j=1}^\ell \left[(-1)^{|C_j|/2}(-1)^{\rt(C_j)+v_-(C_j)}\right]\cdot\tau(\gamma)
\]
(recall that the parity of the number~$v_-(C_j)$ does not depend on the orientation of a cycle~$C_j$ provided it has an even length) and~$\tau(\gamma)$ is the additional contribution of the path~$\gamma$, which we now determine. Similarly to the proof of Claim~A, we have
\[
\tau(\gamma)=(-1)^{|\gamma|+1}(-1)^{s(\gamma)}\omega(\gamma),
\]
where~$|\gamma|$ denotes the number of edges in the path~$\gamma$ and~$\omega(\gamma)$ is the product of the coefficients~$\exp[\frac{i}{2}\rw(\overline{e},e')]$ along this path, when explored \emph{from~$b$ to~$a$}. Comparing~$\omega(\gamma)$ with the total rotation angle of~$\gamma$ leads to
\[
\exp[{\textstyle\frac{i}{2}}\wind(\gamma)] =\omega(\gamma)i^{v_+(\gamma)}(-i)^{v_-(\gamma)}\,.
\]
Using Lemma~\ref{lemma:Whitney}, it is easy to conclude that~$\exp[{\textstyle\frac{i}{2}}\wind(\gamma)]=\rw_{b,\overline{a}}(-1)^{\rt(\gamma)}$. Putting these two equalities together, we arrive at
\begin{align*}
\tau(\gamma)~&=~\rw_{b,\overline{a}}\,\cdot\,(-1)^{\rt(\gamma)+|\gamma|+1+s(\gamma)}\cdot i^{v_-(\gamma)-v_+(\gamma)} \\
& = ~ \rw_{b,\overline{a}}\,\cdot\,(-1)^{\rt(\gamma)+v_-(\gamma)}\cdot (-i)^{|\gamma|-1}\,,
\end{align*}
where we also used the equality~$2s(\gamma)=|\gamma|-1+v_-(\gamma)+v_+(\gamma)$ to pass to the second line. Noting that
\[
\textstyle 2|C_0|+\sum_{j=1}^{\ell}|C_j|+|\gamma|-1= |V(G^\rK_{\rE\cup\{a,b\}})|=2|E(G)|-|\rE|-2\,,
\]
one can rewrite expansion~\eqref{eqn:det-Kab-E-expansion} as
\[
\overline{\rw_{b,\overline{a}}}\,\cdot (-1)^{|E(G)|-1}\det\widetilde{\rK}^{[\rE\cup\{a\},\rE\cup\{b\}]}~ = ~ i^{|\rE|}\sum\nolimits_{C\in\Gamma^{[a,b]}_{\opname{even}}(G^\rK_\rE)} (-1)^{\rt(C)+v_-(C)}2^{\ell(C)}x^\rK(C)\,,
\]
and Theorem~\ref{thm:KW4} is now reduced to the following analog of Claim~$B$ and Claim~$B'$.

\smallskip

\noindent {\em Claim B''.} For any $D\in \cD(G^\rK_\rE)$ and~$D'\in\cD(G^\rK_{\rE\cup\{a,b\}})$ or $D\in \cD(G^\rK_{\rE\cup\{a\}})$ and~$D'\in\cD(G^\rK_{\rE\cup\{b\}})$, depending on the parity of~$|\rE|$, one has
\[
D\cdot D'~=~v_-(D\btu D')+|\rE\cap (ab)|\mod 2,
\]
provided the path~$\gamma$ linking~$a$ and~$b$ in~$D\btu D'$ is oriented from~$b$ to~$a$.

\smallskip

\noindent\emph{Proof of Claim~B''.} Similarly to the proofs of Claim~B and Claim~B', the result follows from the equality
\[
(D\btu D')_-\cdot(D\btu D_0)=|\rE\cap (ab)|\mod 2\,,
\]
where~$(D\btu D')_-$ is a collection of loops and a path~$\gamma_-$ running from~$b$ to~$a$ pushed slightly closer to the arc~$(ab)$, while~$(D\btu D_0)$ is a collection of loops and paths linking the boundary edges~$\rE$, possibly together with~$a$. The number of intersections of~$\gamma_-$ with the latter paths has the same parity as~$|\rE\cap (ab)|$ due to topological reasons.
\end{proof}

\subsection{S-holomorphic functions in the double-Ising model} \label{sub:dblI-s-hol-fct}
The aim of this section is to discuss s-holomorphic functions appearing in the double-Ising model context, see Section~\ref{sub:s-observables} for the terminology. These functions have some potential for the study of the (critical) double-Ising model in bounded planar domains though one needs to develop new tools in order to handle the arising boundary conditions and at the moment it is not clear if one can obtain meaningful convergence results using this approach, see Remark~\ref{rem:dblI-boundary-conditions} and Remark~\ref{rem:dblI-martingale-strategy-oops} below.

Recall that, given a ``source'' edge~$a\in\EE(G)$, the basic s-holomorphic observable in the \emph{single}-Ising model can be defined as
\[
F_a(z_e)~=~ {t_e\cdot (\eta_e\wh{\rK}^{-1}_{e,a}+\eta_{\overline{e}}\wh{\rK}^{-1}_{\overline{e},a})}
~=~ t_e\cdot {(-i\eta_a)}(\rK^{-1}_{e,a}+\rK^{-1}_{\overline{e},a})\,,
\]
where~$\wh{\rK}={i\rU^*\rK\rU}$ is a \emph{real} (anti-symmetric) matrix,~$\eta_e$,~$\eta_{\overline{e}}$ and~$\eta_a$ {denote complex conjugates of the square roots of the directions of the corresponding oriented edges multiplied by~$\zeta=e^{i\frac{\pi}{4}}$ for the notational convenience,}~$t_e=(x_e+x_e^{-1})^{1/2}$ and~$z_e$ is the midpoint of an edge~$e$. In the double-Ising model context, one can use the same definition with the Kac--Ward matrix~$\rK$ replaced by its appropriate modification~$\widetilde{\rK}$:
\[
\widetilde{F}_a(z_e)~:=~ t_e\cdot {(-i\eta_a)}(\widetilde{\rK}^{-1}_{e,a}+\widetilde{\rK}^{-1}_{\overline{e},a})\,.
\]

It is easy to check that the matrix~$i\rU^*\widetilde{\rK}\rU$ is real-valued and hence~$\widetilde{\rK}^{-1}_{e,e'}\in { i{\eta}_e^{\vphantom{|}}\overline{\eta}_{e'}}\R$, similarly to the entries of the matrix~$\rK^{-1}$. Further, it is not hard to argue that the functions~$\widetilde{F}_a(\cdot)$ introduced above satisfy the (generalized) s-holomorphicity condition given by Definition~\ref{def:s-hol}, for properly defined values~$\widetilde{F}_a(c)$ at corners~$c\in V(G^\rC)$. Such values can be constructed, for instance, using the same linear combinations
\[
\widetilde{F}_a(c)~:=~{(-2i\eta_a)\,\cdot}\sum\nolimits_{e:o(e)=v(c)} \big[({\rB}^*)^{-1}_{c,e}\cdot t_e\rK^{-1}_{e,a}\big]
\]
as in the single-Ising model context (where this linear relation is nothing but the expression~\eqref{eqn:B-chi-to-phi} of the ``corner'' variable~$\chi_c$ in terms of the nearby ``edge'' variables~$\phi_e$). The reason why~$F_a(\cdot)$ satisfies the s-holomorphicity condition away from the ``source'' edge~$a$ is as follows: the collection of equalities for all pairs~$(e,c)$ around a fixed vertex~$v=o(e)=v(c)\in V(G)$ is \emph{equivalent} to the collection of equalities
\[
\rK^{-1}_{\overline{e},a}~=-~\sum\nolimits_{e':o(e')=v}\rK^{\vphantom{-1}}_{e,e'}\rK^{-1}_{e,a}\quad \text{for~all~$e$~with~$o(e)=v$},
\]
provided $\rK^{-1}_{e,a}\in {i{\eta}_e^{\vphantom{|}}\overline{\eta}_a}\R$\,, e.g. see~\cite[Section~2.1]{Lis-2013}. Since~$\widetilde{\rK}_{e,e'}=\rK_{e,e'}$ unless~$e=e'\in \EE_\partial(G)$, the latter local relations hold true with the entries~$\rK^{-1}_{e,a}$ replaced by~$\widetilde{\rK}^{-1}_{e,a}$ whose complex phases coincide with those of~$\rK^{-1}_{e,a}$\,.

\begin{remark} Of course, expanding the corresponding minors of~$\widetilde{\rK}$, it is also possible to give a purely combinatorial definition of all the values~$\widetilde{F}_a(\cdot)$ in the style of definition~\eqref{eqn:stas-observable-def} but with a more involved set of ``double-Ising domain walls'' configurations to sum over, and to check the s-holomorphicity relations directly in the style of~\cite[Section~4]{smirnov-icm-2010} or~\cite[Section~2.2]{chelkak-smirnov}.
\end{remark}

Let us now briefly discuss the boundary conditions for s-holomorphic functions which naturally replace~\eqref{eqn:boundary-conditions} in the double-Ising model context. If~$e\ne\overline{a}$ is an \emph{outward} oriented boundary edge of~$G$, then
\[
\widetilde{\rK}^{\vphantom{1}}_{\overline{e},\overline{e}}\cdot\widetilde{\rK}^{-1}_{\overline{e},a} +\widetilde{\rK}^{\vphantom{1}}_{\overline{e},e}\cdot\widetilde{\rK}^{-1}_{e,a}~=~0\,,
\]
which leads to
\[
\widetilde{F}_a(z_e)~=~t_e\cdot {(-i\eta_a)}(\widetilde{\rK}^{-1}_{\overline{e},a} +\widetilde{\rK}^{-1}_{e,a})~=~{t_e(1-ix_e)\cdot (-i\eta_a)}\widetilde{\rK}^{-1}_{\overline{e},a}\,.
\]
As~$(1-ix_e)=(1+x_e^2)e^{-\frac{i}{2}\theta_e}$ and $\eta_a\widetilde{\rK}^{-1}_{\overline{e},a}\in i\eta_{\overline{e}}\R=\eta_e\R$,
we arrive at the following claim:
\begin{equation}
\label{eqn:dblI-boundary-conditions}
\widetilde{F}_a(z_e)~\in~{ie^{-\frac{i}{2}\theta_e} \eta_{e}}\R\quad\text{for~all~outward~oriented~boundary~edges~}e\neq\overline{a}.
\end{equation}
(Note that the factor~$e^{-\frac{i}{2}\theta_e}$ appeared as a result of adjusting the entry~$\widetilde{\rK}_{\overline{e},\overline{e}}$ and would disappear, leading back to~\eqref{eqn:boundary-conditions}, if one works with the original matrix~$\rK$ instead of~$\widetilde{\rK}$.)

\begin{remark}
\label{rem:dblI-boundary-conditions} The s-holomorphic functions~$\widetilde{F}_a(\cdot)$ discussed above (or their spinor analogs constructed via the matrices~$\widetilde{\rK}_{[u_1,..,u_m]}^{-1}$ instead of~$\widetilde{\rK}^{-1}$, see Proposition~\ref{prop:spin4}) can be used to study (the scaling limits of) correlation functions in the critical double-Ising model, cf.~Remark~\ref{rem:formulas-via-K^-1} and Remark~\ref{rem:towards-scaling-limits-single}. Nevertheless, let us emphasize that the additional factor~$e^{-\frac{i}{2}\theta_e}$ in~\eqref{eqn:dblI-boundary-conditions} is \emph{lattice dependent} on the one hand, and the s-holomorphicity condition is \emph{not} complex-linear on the other (thus, even when working on regular lattices with constant weights~$x_e=x_{\opname{crit}}$, one cannot just multiply the function~$\widetilde{F}_a(\cdot)$ by~$e^{\frac{i}{2}\theta_e}$ to pass from~\eqref{eqn:dblI-boundary-conditions} to~\eqref{eqn:boundary-conditions}). Interestingly enough, boundary conditions~\eqref{eqn:dblI-boundary-conditions} admit a \emph{lattice independent} reformulation in terms of the discrete primitive~$\int\Im[(\widetilde{F}_a(z))^2dz]$ but it is still not clear how to pass to the limit of such s-holomorphic functions even in general smooth planar domains, cf. Remark~\ref{rem:towards-scaling-limits-single}.
\end{remark}

We conclude this section with a brief informal discussion of s-holomorphic functions that appear when considering the double-Ising model with Dobrushin boundary conditions~\eqref{eqn:dblI-Dobrushin-bc} instead of `$+$' ones. For a fixed pair of inward oriented boundary edges~$a$ and~$b$, let us slightly modify definition~\eqref{eqn:tilde-Kab-entries} and introduce a matrix~$\widetilde{\rK}^{\pm}_{e,e'}$, whose entries are still given by~\eqref{eqn:tilde-Kab-entries} but the row~$a$ and the column~$b$ are \emph{not} removed this time (also, note that we do not adjust the values~$\rK_{a,a}=\rK_{b,b}=0$ when defining~$\widetilde{\rK}^\pm$). Further, let
\[
\widetilde{F}^\pm_a(z_e):= t_e\cdot {(-i\eta_a)}((\widetilde{\rK}^\pm)^{-1}_{e,a} +(\widetilde{\rK}^\pm)^{-1}_{\overline{e},a})\qquad\text{and}\qquad M_{G,a,b}(z_e):= {\widetilde{F}^\pm_a(z_e)}\,\big/\,{\widetilde{F}^\pm_a(z_b)}\,.
\]

Similarly to~\cite[Section~2.2]{chelkak-smirnov}, one can work out combinatorial expansions of~$F^\pm_a(z_e)$ and show that the quantities~$M_{G,a,b}(z_e)$ are discrete martingales in the double-Ising model with respect to the \emph{interface}~$\gamma_{a,b}$ (domain wall separating double-Ising spins~$\widetilde{\sigma}_u=\pm 1$) generated by Dobrushin boundary conditions~\eqref{eqn:dblI-Dobrushin-bc}. Note that the key observation leading to this martingale property is given by Theorem~\ref{thm:KW4}, which relates the denominator of~$M_{G,a,b}(\cdot)$ with the partition function of the double-Ising model with Dobrushin boundary conditions. Having such a collection of discrete martingales, one could try to implement the same strategy as in~\cite{chelkak-smirnov,CDCHKS} for the critical \emph{double}-Ising model, first proving the convergence of these s-holomorphic observables to some scaling limits and then analyzing the putative limit of the interfaces~$\gamma_{a,b}$ using the limits of discrete martingales~$M_{G,a,b}(\cdot)$.

\begin{remark}
\label{rem:dblI-martingale-strategy-oops}
Though the above strategy of studying the single interface~$\gamma_{a,b}$ generated by Dobrushin boundary conditions~\eqref{eqn:dblI-Dobrushin-bc} in the critical double-Ising model looks rather promising at first sight, we expect conceptual obstacles along the way. The reason is that the critical double-Ising model with Dobrushin boundary conditions is (rather surprisingly) conjectured to~\emph{lose} the domain Markov property when passing from discrete to continuum, see~\cite[p.~3]{Wilson-xor}. At the level of convergence of solutions to the relevant discrete boundary value problems, this should mean that the boundary conditions~\eqref{eqn:dblI-boundary-conditions} have \emph{different} limits on the smooth boundary and on the fractal one, generated by the (first part of the) interface~$\gamma_{a,b}$ itself. In fact, it is even not easy to guess what the latter limit should be, not speaking about proving the relevant convergence theorem for discrete martingales~$M_{G,a,b}(\cdot)$ discussed above.
\end{remark}


\end{document}